\documentclass[12pt,a4paper]{article}
\usepackage[utf8]{inputenc}
\usepackage{lmodern}
\usepackage[T1]{fontenc}
\usepackage[english]{babel}
\usepackage{ifpdf}
\usepackage[usenames,dvipsnames]{xcolor}
\usepackage{graphicx}
\usepackage[shortlabels]{enumitem}
\usepackage[numbers,sort]{natbib}
\usepackage[font=footnotesize,width=\textwidth]{caption}
\usepackage{booktabs}
\usepackage{hyperref}
\usepackage{uri}
\usepackage[top=1.8cm, bottom=2.8cm]{geometry}
\usepackage[affil-it]{authblk}
\usepackage[parfill]{parskip}
\usepackage[setbuildercolon]{matmatix}
\usepackage{relsize}
\usepackage{xstring}
\usepackage{tocloft}

\ifpdf
  \hypersetup{
    pdftitle={A branching process with coalescence to model random phylogenetic networks}
  }
\fi

\hypersetup{colorlinks, allcolors = {MidnightBlue}}

\newenvironment{mathlist}
{\begin{enumerate}[label={\upshape(\roman*)}, align=left, widest=iii, leftmargin=*]}
{\end{enumerate}\ignorespacesafterend}

\makeatletter
\renewcommand*{\defas}{%
\mathrel{\rlap{\raisebox{0.3ex}{$\m@th\cdot$}}\raisebox{-0.3ex}{$\m@th\cdot$}}=}
\makeatother

\makeatletter
\newtheorem*{rep@theorem}{\rep@title}
\newcommand{\newreptheorem}[2]{%
\newenvironment{rep#1}[1]{%
 \def\rep@title{#2 \ref*{##1}}%
 \begin{rep@theorem}}%
 {\end{rep@theorem}}}
\makeatother

\makeatletter
\DeclareRobustCommand\bigop[2][1]{%
  \mathop{\vphantom{\sum}\mathpalette\bigop@{{#1}{#2}}}\slimits@
}
\newcommand{\bigop@}[2]{\bigop@@#1#2}
\newcommand{\bigop@@}[3]{%
  \vcenter{%
    \sbox\z@{$#1\sum$}%
    \hbox{\resizebox{\ifx#1\displaystyle#2\fi\dimexpr\ht\z@+\dp\z@}{!}{$\m@th#3$}}%
  }%
}
\makeatother

\newtheorem{theorem}{Theorem}[section]
\newreptheorem{theorem}{Theorem}
\newtheorem{proposition}[theorem]{Proposition}
\newreptheorem{proposition}{Proposition}
\newtheorem{lemma}[theorem]{Lemma}
\newreptheorem{lemma}{Lemma}
\newtheorem{corollary}[theorem]{Corollary}
\theoremstyle{definition}
\newtheorem{definitionX}[theorem]{Definition}
\newenvironment{definition}
  {\pushQED{\qed}\definitionX}
  {\popQED\enddefinitionX}
\newtheorem{notationX}[theorem]{Notation}
\newenvironment{notation}
  {\pushQED{\qed}\notationX}
  {\popQED\endnotationX}
\newtheorem{remarkX}[theorem]{Remark}
\newenvironment{remark}
  {\pushQED{\qed}\remarkX}
  {\popQED\endremarkX}

\newcommand{\newtext}[1]{#1}

\newcommand{\TweakSpacingL}[1]{\,#1}
\newcommand{\TweakSpacingLR}[1]{\,#1\!}

\newcommand{\biasedby}[1][normalsize]{%
  \IfEqCase{#1}{%
    {normalsize}{\mathbin{\dagger}}%
    {big}{\mathrel{\mathlarger{\dagger}}}%
  }[\PackageError{biasedby}{Undefined option to biasedby: #1}{}]%
}
\newcommand{\nucirc}{\nu_\circ}
\newcommand{\ExpecCirc}[1]{\Expec[][\circ]{#1}}
\newcommand{\ExpecCircNormSize}[1]{\Expec*[][\circ]{\normalsize}{#1}}
\newcommand{\altoverline}[1]{\,\overline{\!#1}}
\newcommand{\altunderline}[1]{\underline{#1\!}\,}
\newcommand{\altoverlinesup}[2]{\,\smash{\overline{\!#1}}^{(#2)}}
\newcommand{\altunderlinesup}[2]{\smash{\underline{#1\!}}^{\:(#2)}}
\newcommand{\kettenbruch}{\DOTSB\bigop[.96]{\mathrm{K}}}
\newcommand{\mrca}{\wedge}

\newcommand{\dt}[1]{dt}

\newcommand{\GHPdist}{d_{\mathrm{GHP}}}
\newcommand{\Xmut}{X^{\mathfrak{m}}} 
\newcommand{\paste}{\wr}
\newcommand{\dis}{\underline{\mathrm{dis}}}
\newcommand{\rf}{r_{\!f}}
\newcommand{\rg}{r_{\!g}}
\newcommand{\nf}{n_{\!f}}
\newcommand{\rn}{r_{\!n}}
\newcommand{\gammabar}{\,\bar{\!\gamma}}
\newcommand{\gammatilde}{\,\tilde{\!\gamma}}

\title{A branching process with coalescence to model random phylogenetic networks}

\begin{document}

\newgeometry{top=0.5cm, bottom=1.8cm}

\author[1]{François Bienvenu}
\author[2]{Jean-Jil Duchamps}

\renewcommand\Affilfont{\itshape\small\renewcommand{\baselinestretch}{1}}
\affil[1]{Institute for Theoretical Studies, ETH Zürich,
8092 Zürich, Switzerland\vspace{1ex}}
\affil[2]{Laboratoire de mathématiques de Besançon, UMR 6623,               
Université de Franche-Comté, CNRS, F-25000 Besançon, France}

\maketitle

\begin{abstract}
We introduce a biologically natural, mathematically tractable model of
random phylogenetic network to describe evolution in the presence of
hybridization. One of the features of this model is that the hybridization
rate of the lineages correlates negatively with their phylogenetic distance.
We give formulas\,/\,characterizations for quantities of biological interest
that make them straightforward to compute in practice. We show that the
appropriately rescaled network, seen as a metric space, converges to the
Brownian continuum random tree, and that the uniformly rooted network has a
local weak limit, which we describe explicitly.
\end{abstract}


\renewcommand{\cfttoctitlefont}{\large\bf}
\setlength{\cftbeforesecskip}{6pt}
{\footnotesize \tableofcontents}

\thispagestyle{empty}
\pagebreak
\restoregeometry

\section{Introduction} \label{secIntro}

\subsection{Biological context} \label{secBio}

Random trees play a central role in evolutionary biology: ultimately, much of
what we know about evolution relies on a random tree being used as
a null model. Meanwhile, the genomic revolution of the past decades has
shown that phenomena once thought to play a minor role in large-scale
evolution, such as hybrid
speciation~\cite{Linder2004reconstructing,Mallet2007hybrid} or horizontal gene
transfers~\cite{Bergthorsson2003,Hao2004patterns, Mower2004gene}, are in fact
widespread and crucial to our understanding of evolutionary processes.
As a result, there have been growing calls by biologists to replace trees by
networks when studying
phylogenies~\cite{Doolittle2007pattern,Dagan2009getting,Bapteste2013},
which lead to the emergence of the flourishing field of phylogenetic networks
(see e.g.~\cite{KPKW22} for a recent review).

Despite this, there is still a notable lack of biologically relevant,
mathematically tractable models of random phylogenetic networks. To the best of
our knowledge, so far only two models of random phylogenetic networks have been
studied extensively from a probabilistic standpoint: uniform ranked
tree-child networks~\cite{BLS22} and uniform \mbox{level-$k$} networks
\cite{Stu22}. Uniform ranked tree-child networks are generated by a biologically
natural process where species split at constant rate and pairs of species
hybridize at a constant rate. They turn out to be highly
tractable~\cite{BLS22,CFY22,FLY22}. However, they fail to take into account
the fact that phylogenetically distant species are less likely to hybridize
than closely related ones, which results in a very non-tree-like structure
whose biological relevance is questionable. By contrast, uniform level-$k$
networks have a tree-like large-scale structure~\cite{Stu22}; but they do not
have a biological interpretation that would justify their relevance as a model
of random phylogenetic network, and they are
not as mathematically tractable (at least for generic values of~$k$).

In this work, we introduce a model of random phylogenetic network
that has a natural biological interpretation while remaining mathematically
tractable. The idea of this model is to consider species that (1) speciate
and go extinct at constant rates and (2) hybridize, subject to some constraints:
each species has a type, which can be thought of as a proxy for the
genetic distance, and species of the same type hybridize at a constant rate.
Types are created at a constant rate in an infinite-allele fashion,
and inherited by descendants. The formal description of the model is
given in the next section, along with an overview of our main results.

\subsection{Setting and main results} \label{secResults}

Starting from one colored lineage at time~$t=0$, consider the continuous-time
interacting particle system where:
\begin{itemize}
  \item each lineage splits into two lineages at rate $1$ (\emph{branching});
  \item each lineage dies at rate \newtext{$\alpha > 0$} (\emph{death});
  \item each pair of lineages of the same color merge at rate \newtext{$2\beta > 0$} (\emph{coalescence});
  \item each lineage takes a new, never-seen-before color at rate \newtext{$\mu > 0$}
    (\emph{mutation}).
\end{itemize}

As illustrated in Figure~\ref{figDefProcess}, this process defines a
time-embedded random network which can be seen as a random metric measure
space $(\mathscr{G}, d_{\mathscr{G}}, \lambda_{\mathscr{G}})$. Formally,
a point $x \in \mathscr{G}$ corresponds to a lineage $\ell$ and a time~$t$
at which that lineage is alive. Since the lineages can be seen as segments,
$\mathscr{G}$ can be seen as a collection of segments glued together at their
endpoints, and $\lambda_{\mathscr{G}}$ as the usual Lebesgue measure on this
union of segments.

There is a natural metric $d_{\mathscr{G}}$ on $\mathscr{G}$, obtained by
defining $d_{\mathscr{G}}(x,y)$ to be the
length of a shortest path between two points $x, y \in \mathscr{G}$.
Since genetic material cannot be transmitted back in time, a more biologically
relevant notion of distance between two points $x,y\in\mathscr{G}$
would consist in considering only the paths that lie in the past of the focal points.
Letting $h(x) = t$ denote the \emph{height} of a
point $x = (\ell, t) \in \mathscr{G}$ and $x \mrca y$ the most recent common
ancestor of $x$ and $y$, this notion of distance can be expressed as
\[
  h(x) + h(y) - 2 h(x \mrca y).
\]
However, this does not define a distance in the
mathematical sense, as the triangle inequality is not satisfied when
the network contains coalescence points.

\begin{figure}[t]
  \centering
  \includegraphics[width=0.5\linewidth]{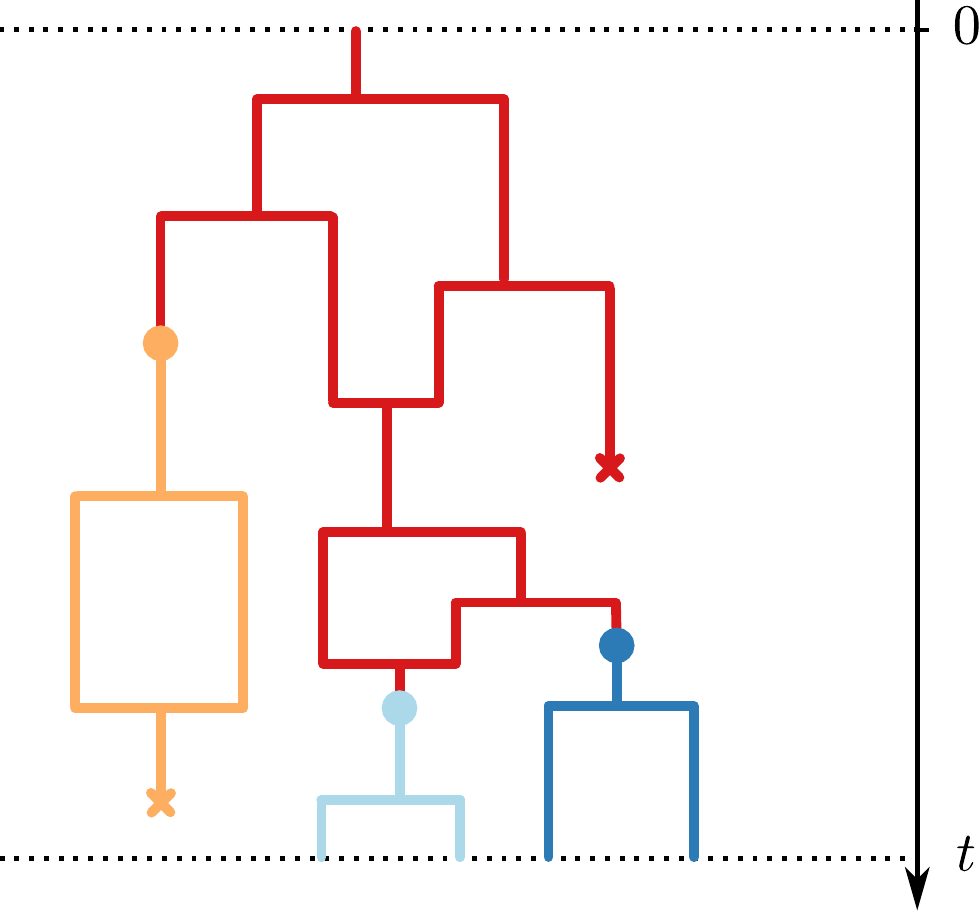}
  \caption{Graphical depiction of a realization of the process generating
  the network $\mathscr{G}$. The vertical axis is the time,
  flowing from top to bottom, and the vertical lines represent the lineages.
  Dots correspond to mutations (i.e.\ to a lineage changing color)
  and crosses correspond to deaths (i.e.\ to a lineage stopping).
  Horizontal lines correspond to either branching or coalescence,
  and serve to indicate the genealogical relationship between lineages; they
  should be treated as having length~0.}
  \label{figDefProcess}
\end{figure}

In this document, we are mostly interested in the structure of
$\mathscr{G}$ conditioned on being large. More specifically, we consider
$(\mathscr{G}_n, d_{\mathscr{G}_n}, \lambda_{\mathscr{G}_n})$, the metric
measure space having the law of $(\mathscr{G}, d_{\mathscr{G}},
\lambda_{\mathscr{G}})$ conditioned on having $n$ colors, and we study various
limits of $\mathscr{G}_n$ as $n$ goes to infinity.

For this, it will be convenient to see $\mathscr{G}$ as a
decorated Galton--Watson tree. For each color~$k$, let $\mathscr{X}_k$
denote the subnetwork of $\mathscr{G}$ formed by the lineages of color~$k$,
endowed with the information of which endpoint corresponds to the
creation of the color~$k$ (henceforth referred to as the \emph{root} of
$\mathscr{X}_k$) and of which of the other endpoints correspond to
mutations as opposed to deaths.
Let $\mathscr{T}$ denote the genealogical tree of the colors -- that is,
the ordered tree whose vertices are the colors of the lineages and where
$k'$ is a child of $k$ if and only if $k'$ was created by the mutation
of a lineage of color~$k$, the children of a color being ordered according
to the order of apparition of the corresponding mutations.
Finally, let $\mathscr{T}^\star$ denote the tree~$\mathscr{T}$ where each
vertex~$k$ is decorated by the corresponding network $\mathscr{X}_k$.
Note that $\mathscr{G}$ and $\mathscr{T}^\star$ contain the same information,
since to reobtain $\mathscr{G}$ from $\mathscr{T}^\star$ it suffices to glue,
for each color $k$ and each child of that color, the root of the decoration the
$i$-th child of~$k$ to the endpoint of $\mathscr{X}_k$ corresponding to its
$i$-th mutation.

In Section~\ref{secFourreTout}, we list miscellaneous results that are used
throughout the document, starting with properties of the process describing the
dynamics of the number of lineages of a given color (the so-called \emph{logistic
branching process}). We then study the random variable $M$ giving
the number of new colors that a color produces over its lifetime, i.e.\
the offspring distribution of the Galton--Watson tree~$\mathscr{T}$.
We show that its expected value is given by
\[
  \Expec{M} \,=\;
  \mu \sum_{j \geq 1} \prod_{k = 1}^j \frac{1}{\rho_k} \,, 
\]
where $\rho_k = \alpha + \mu + (k-1) \beta$. Since every
color has an almost surely finite lifetime, the process
generating $\mathscr{G}$ goes extinct with probability~1 if and only if
${\Expec{M} \leq 1}$. We also show that the probability generating function
of~$M$ can be expressed as the continued fraction
\[
  g(z) \;=\; \cfrac{\alpha + \mu z}{1 + \rho_1 -
  \cfrac{\alpha + \beta + \mu z}{1 + \rho_2 -
  \cfrac{\alpha + 2\beta + \mu z}{1 + \rho_3 - \ddots}}} \;. 
\]
This expression makes it straightforward to numerically compute the probability
of extinction of the process -- that is, the smallest fixed-point of $g$
in $[0, 1]$. Similarly, we give a characterization of the
asymptotic growth rate of the total number of lineages that makes it possible
to compute it in practical applications.

In Sections~\ref{secCRT} and~\ref{secLocalLimit}, we study the geometry of
the network $\mathscr{G}_n$. Section~\ref{secCRT} deals with the
global, large-scale structure of $\mathscr{G}_n$ as $n$ goes to infinity. This
structure is tree-like: in Theorem~\ref{thmCRTconvergence} we show that, letting
$|\mathscr{G}_n| = \lambda_{\mathscr{G}_n}(\mathscr{G}_n)$, for some
well-characterized constant $C$ the rescaled space
\[
  \mleft(\mathscr{G}_n,\,
         \tfrac{C}{\sqrt{n}} d_{\mathscr{G}_n},\,
         \tfrac{1}{|\mathscr{G}_n|} \lambda_{\mathscr{G}_n}\mright)
\]
converges to Aldous' Brownian continuum
random tree in distribution for the Gromov--Hausdorff--Prokhorov topology.
Finally, Section~\ref{secLocalLimit} focuses on the local structure of
$\mathscr{G}_n$: we show that $\mathscr{G}_n$ rooted at a uniform point
has a local weak limit, which we describe explicitly.

\subsection{Comments and perspectives} \label{secComments}

Our proof of the convergence to the CRT is based on~\cite{Stu22}, where most
of the ideas that we use in Section~\ref{secCRT} can already be found.
Nevertheless, some specificities of our model -- in particular the fact that
the number of new colors produced by a color during its lifetime and
the total length of the corresponding subnetwork are not bounded random
variables -- require a different treatment and have necessitated a
fine-grained study of the logistic branching process with mutation.
\newtext{The existence of various local weak limits was also obtained in \cite{Stu22}. The
ideas are similar, insofar as we are also dealing with local weak limits of
blow-ups of Galton--Watson trees. However, there are some notable differences
-- such as the fact that our focal point is chosen uniformly with
respect to the length measure of our time-embedded network (as opposed to
uniformly on the vertices of the underlying graph) and that we are able to give a
more explicit description of the local geometry of the limit.}

In an effort to make this paper accessible to mathematical biologists who do not
have specific knowledge about Galton--Watson trees or Gromov--Hausdorff--Prokhorov
convergence, we have strived to make it as self-contained as possible by
(1) providing detailed reminders about most of the notions and results that
are used and (2) whenever possible, expressing our results as general
statements that are not tied to our particular setting. In particular,
Proposition~\ref{propKeyToCRT} provides a general recipe for proving
Gromov--Hausdorff--Prokhorov convergence to a random
$\R$-tree, and Lemma~\ref{lemUniformBoundSumF} makes it straightforward to
apply this proposition to decorated Galton--Watson trees.

We close this introduction by mentioning an interesting line of research:
our study hinges on the fact that our model can be seen as a decorated
Galton--Watson tree. This crucial connection stems from the fact that the
hybridization rate is a 0-1 function of the phylogenetic distance, which has
the simple form $d(\ell, \ell') = 1$ if $\ell$ and $\ell'$ are the same color,
and 0 otherwise.  However, from a modelling point of view it would be more
natural to use a more nuanced notion of phylogenetic distance, and to let the
phylogenetic distance be a gradually decreasing function of that distance.

For instance -- as a first step and in keeping with the idea of colors representing
incompatibility alleles -- one could let lineages carry several colors, and make
the hybridization rate between two lineages a decreasing function of the number of colors
that differ between these lineages. Based on the biological interpretation,
one might expect such a model to have properties that are very similar to our
model. However, because the link with Galton--Watson trees is lost, it is not
clear whether this is the case, and how to study this.  Therefore, studying
such models of phylogenetic networks -- whose large-scale geometry is expected
to be tree-like, even though there is no immediate, rigorous connection with
branching processes -- seems like an interesting and challenging problem that
will likely require developing new tools and methods.

\section{Probability of extinction and growth rate} \label{secFourreTout}

\subsection{The logistic branching process} \label{secLogistic}

Throughout this document, we denote by $X = (X_t : t \geq 0)$ the process
counting the number of lineages of the first color. It is a birth-death
process started from $X_0 = 1$, killed in~$0$, and with
transition rates:
\begin{itemize}
  \item $k \to k+1$ at rate $k$;
  \item $k \to k-1$ at rate $k \rho_k = (\alpha + (k -1) \beta + \mu) k$.
\end{itemize}
This process has been called the \emph{branching process with logistic growth}
(or, more succinctly, the \emph{logistic branching process}) and has been
studied, e.g, in \cite{Lam05,Par18}.
It is also a special case of a branching process with interactions,
see~\cite{Kal02,CaPL21,OjPa22}.

The qualitative behaviour of $X$ can be described as transient
fluctuations in a potential well. Indeed, letting
$K = 1 + (1-\alpha-\mu) / \beta$,
when $X$ is smaller than $K$ it tends to increase whereas when it is greater
than $K$ it tends to decrease. Thus, in particular when $K$ is large,
typical trajectories of $X$ quickly relax towards a quasi-stationary
distribution and then fluctuate until they eventually hit~0, which happens
in finite time with probability~1.

Although this qualitative behaviour is well-understood,
the quadratic term in the death rate makes the obtention of exact quantitative
results difficult -- and, to some extent, impossible. For instance, a classic
approach to study birth-death processes consists in using the Kolmogorov forward
equations to obtain a characterization of the probability generating function
$f(z, t) = \Expec*{\normalsize}{z^{X_t}}$ 
as the solution of a partial differential equation.
Here, standard calculations show that $f$ is the unique analytic solution on
$\ClosedInterval{0, 1} \times \R_+$ of
\[
  \partial_t f \;=\;
  (z - \alpha - \mu)(z - 1)\, \partial_z f \;+\; \beta z(1-z)\,  \partial_{zz} f
\]
with $f(z, 0) = z$. However, this partial differential equation is known
not to have a closed-form solution -- see Proposition~1.2 in~\cite{ACM20}.
Another powerful approach to study birth-death processes is the
integral representation of the transition probabilities
using orthogonal polynomials introduced by
Karlin and McGregor \cite{KaMc57a,KaMc57b}, but to our knowledge in the
case of the logistic branching process this does not yield useful
explicit expressions.

One of the important properties of the logistic branching process is that
it comes down from infinity (meaning that there is a unique way to start it from
$X_0 = \infty$ and yet have $X_t < \infty$ for any $t > 0$), as shown by
Lambert in \cite{Lam05}. We denote by $\E_{\infty}$ the
expectation under the initial condition $X_0 = \infty$.
A recurring quantity throughout this paper is the extinction time
$T = \inf\Set{t \geq 0 \suchthat X_t = 0}$. In his Theorem~2.3, Lambert
gives Laplace transform of $T$ under $\E_{\infty}$ as a function of the
solution of a Riccati equation, and shows that its expected value is finite. In
fact, $T$ also has finite exponential moments under $\E_{\infty}$. This can be
deduced, e.g, from~\cite[Proposition~2.4]{FLZ20}
or~\cite[Proposition~2.2]{BMR16}, and will play an important
role in our study~-- even though, for reasons that will become clear,
we actually need a variant of this result
(namely Lemma~\ref{lemLExpMomentUnderMBiais} in Section~\ref{secTechLemmas}).

%

Because in our setting the mutations associated to the logistic branching
process play a crucial role, the following change of measure will be useful.
In what follows, we fix the parameters $\alpha$
and $\beta$, and we denote by $\E_{\mu}$ the expectation under a
logistic branching process with mutation rate~$\mu$.

\begin{proposition} \label{propMeasureChangeLogisticBP}
Let $M$ be the number of mutations associated to the logistic branching process
$X$, and let $L = \int_0^\infty X(t)\,dt$. Then, for any positive number $s$
and any nonnegative measurable functional $f$ of the trajectory of $X$,
\[
  \Expec[\mu]{f(X)\,s^M} \;=\;
  \Expec[s\mu]{f(X)\,e^{(s-1)\mu L}} \,.
\]
\end{proposition}

\begin{proof}
For this proof, it will be convenient to use the ``extended'' chain~$\bar{X}$
which, in addition to the trajectory of $X$, contains the information about
which transitions correspond to mutations. In other words,
$\bar{X}$ is a continuous-time Markov chain where there are two distinct types of transitions from $i$~to~$j$, one with rate $q_{ij}^\circ$ and one with rate $q_{ij}^\bullet$, where
\[
  q_{ij}^\circ \;=\; i\,\Indic{j = i + 1} + (\rho_i - \mu)\,i\,\Indic{j = i - 1} 
  \quad\text{and}\quad
  q_{ij}^\bullet \;=\;  \mu \,i\,\Indic{j = i - 1}
\]
so that $q_{ij} = q_{ij}^\circ + q_{ij}^\bullet$.
For convenience, we also use the notation $q_i = \sum_j q_{ij}$, i.e.\
in our case $q_i= (1 + \rho_i)\, i$.

\newtext{ First, note that $\bar{X}$ almost surely has a finite number of jumps
before hitting 0. Thus, for any $n\in \mathbb{N}$, to any trajectory $\gammabar$ starting from $1$ and
ending in $0$ after $n$ jumps, encoded as a sequence $\tilde{\gamma}$ of $i
\overset{\circ}{\to} j$ and $i \overset{\bullet}{\to} j$ transitions and a
vector $x = (x_1, \ldots, x_n)$ of corresponding holding times, we can
associate the probability density }
\[
  \Prob*{\normalsize}{\bar{X}\in d\gammabar} \;=
  \Big(\prod_{\;\,i\overset{\circ\;}{\to}j} \!q_{ij}^\circ\;\Big)
  \Big(\prod_{\;\,i\overset{\bullet\;}{\to}j} \!q_{ij}^\bullet\;\Big)
  \Big(\,\prod_{k = 1}^n e^{-q_{\tilde{\gamma}_k}x_{k}} \Big)\,dx ,
\]
where $\gammatilde_k$ denotes the position of $\gammabar$ before the $k$-th
transition, and $dx = dx_1 \cdots dx_n$.
Noting that $M(\gammabar) = \sum_{i\overset{\bullet\;}{\to}j} 1$ and that
$L(\gammabar) = \int_0^\infty \gammabar(t)\, dt = \sum_{k = 1}^n \gammatilde_k\, x_k$,
and recalling that $q_{ij}^\bullet = \mu \,i\,\Indic{j = i - 1}$, we get
\begin{align*}
  &\Expec[\mu]{f(\bar{X})\,s^M} \\
  &\quad=\; \sum_{\gammatilde} s^{M(\gammatilde)} \!\int f(\gammatilde, x)
    \Big(\prod_{\;\,i\overset{\circ\;}{\to}j} \!q_{ij}^\circ\;\Big)
    \Big(\prod_{\;\,i\overset{\bullet\;}{\to}j} \!\mu i\,\Big)
    \Big(\,\prod_{k = 1}^n e^{-q_{\gammatilde_k}x_{k}}\Big)\, dx\\
  &\quad=\; \sum_{\gammatilde} \int f(\gammatilde, x)
    \Big(\prod_{\;\,i\overset{\circ\;}{\to}j} \!q_{ij}^\circ\;\Big)
    \Big(\prod_{\;\,i\overset{\bullet\;}{\to}j} \!s\mu i\,\Big)
    \Big(\,\prod_{k = 1}^n e^{-(q_{\gammatilde_k} + (s-1)\mu \gammatilde_k )x_k} \Big)
    \Big(\,\prod_{k = 1}^n e^{(s-1)\mu \gammatilde_k x_k} \Big)\,dx\\
  &\quad=\; \Expec[s\mu]{f(\bar{X})\,e^{(s-1)\mu L}} \,,
\end{align*}
\newtext{
where the sums run on all sequences $\tilde{\gamma}$ that
end in 0 after a finite number of steps, and the integrals
are over $\R_+^n$, with $n$ is the number of jumps of $\tilde{\gamma}$.}
This concludes the proof.
\end{proof}

Finally, it will also be useful to describe the trajectory of $X$ as seen
from a uniform mutation time. For this, we first need to introduce some
notation for yet another type of changes of measures that will appear several
times in the paper.

\begin{notation} \label{notationBiasedTrajectory}
Let $A$ and $B$ be random variables defined on the same probability space
such that $B$ is almost surely nonnegative and $0 < \Expec{B} < \infty$.
We write $\mathscr{L}(A \biasedby B)$ for the distribution of $A$ biased by $B$,
that is, under the probability measure defined by
$\Prob{\,\cdot \biasedby B} = \Expec*{\normalsize}{\Indic{\cdot} B} \,/\, \Expec{B}$.
\end{notation}

With this notation, by ``the process $X$ as seen from a uniform mutation time''
we rigorously mean
\[
  \Xmut \;\sim\; \mathscr{L}\big((X_{U+t})_{-U\leq t < T-U} \biasedby[big] M\big)\,,
\]
where $U$ is chosen uniformly at random
among the atoms of the point process $\mathcal{M}$ giving the times of the
mutations associated to the trajectory of $X$; note that $U$ need not be defined
when $\mathcal{M}$ is empty because $\Prob{M = 0 \biasedby M} = 0$.
Equivalently, the distribution of $\Xmut$ is characterized by
\[
  \Expec{F(\Xmut)} \;=\;
  \frac{1}{\Expec{M}}\,\ExpecBrackets{\sum_{\,u\in \mathcal{M}}\!\!F\big((X_{u+t})_{-u\leq t < T-u}\big)}
\]
for any measurable bounded functional $F$.

It turns out that it is also possible to obtain $\Xmut$ by a simple construction.
For this, we need to introduce one last definition.

\begin{definition} \label{defBackToBackPasting}
Let $f\colon\COInterval{0, T_f}\to \R$ and $g\colon\COInterval{0, T_g}\to \R$
be two càdlàg functions.  The \emph{back-to-back pasting of $f\!$ to $g$} is the
càdlàg function $f\paste g\colon \COInterval{-T_f, T_g} \to \R$
defined by
\[
  f\paste g\colon t \mapsto
  \begin{cases}
    \lim_{s\downarrow t}f(-s) & \text{if } t < 0 \\
    g(t) & \text{if } t\geq 0\,.
  \end{cases}
  \qedhere
\]
\end{definition}

\begin{proposition} \label{propNuCircXfromMut}
Let $\nucirc$ be the probability distribution on the positive integers
defined by
\begin{equation} \label{eqDefNuCirc}
  \nucirc(n) \;=\; C \prod_{k=1}^{n}\frac{1}{\rho_k}\,,
\end{equation}
with $C$ the corresponding normalizing constant. Let $\TweakSpacingL{K \sim \nucirc}$ and,
conditional on~$K$, let $X'\!$ and $X''\!$ be two independent realizations
of the logistic branching process~$X$ started from $X'_0 = K$ and $X''_0 = K - 1$.
Then,
\[
  \Xmut \;\overset{d}{=}\; X'\paste X''.
\]
\end{proposition}

The proof uses general results about the decomposition of trajectories of
Markov chains that are recalled in Appendix~\ref{secPathDecomposition}, and
therefore is deferred to the end of that appendix.

\subsection{Offspring distribution and extinction probability of~$\TweakSpacingL{\mathscr{T}}$} \label{sec}

In this section, we focus on the law of the random variable $M$
giving the number of mutations of a color (that is, the number of new
colors that it produces; also the offspring distribution of the
Galton--Watson tree $\mathscr{T}$).
Our main result is the following theorem, on which much of our study relies.

\begin{theorem} \label{thmM} Let $M$ be the offspring distribution
of~$\TweakSpacingLR{\mathscr{T}}$, and let $g$ be its probability generating
function. Then, letting
$\rho_k = \alpha + \mu + (k-1) \beta$,
\[
  \Expec{M} \,=\;
  \mu \sum_{j \geq 1} \prod_{k = 1}^j \frac{1}{\rho_k} 
\]
and
\[
  g(z) \;=\; \cfrac{\alpha + \mu z}{1 + \rho_1 -
  \cfrac{\alpha + \beta + \mu z}{1 + \rho_2 -
  \cfrac{\alpha + 2\beta + \mu z}{1 + \rho_3 - \ddots}}} \,, 
\]
which, using Gauss's notation for continued fractions, can be written
\[
  g(z) \;=\; - \kettenbruch_{k = 1}^{\infty}
               \frac{\mu - \rho_k - \mu z}{1 + \rho_k} \,.
\]
Moreover, $g$ is meromorphic on $\TweakSpacingL{\C}$. The radius of
convergence of its power series expansion around~0 is $R > 1$, and $g$ has a pole
in $R$.
\end{theorem}
 
\begin{proof}
Let $\tilde{X}$ be the embedded chain of $X$, that is,
$\tilde{X}_i = X_{\tau_i}$ where $\tau_0 = 0$ and
$\tau_{i + 1} = \inf\Set{t > \tau_i \suchthat X_t \neq X_{\tau_i}}$.
Note that, conditional on the trajectory of $\tilde{X}$,
each step from $k$ to $k-1$ corresponds to a mutation with probability
$p_k \defas \mu / \rho_k$, independently of everything else.
Let us refer to a trajectory of $\tilde{X}$ started from $k$ and killed
when it first hits~$k-1$ as a \emph{$k$-excursion of $\tilde{X}$}.
Every $k$-excursion of $\tilde{X}$ can be decomposed into $N_k$ independent
$(k+1)$-excursions, followed by a single step from $k$ to $k-1$. By the
strong Markov property, $N_k$ follows a geometric distribution on
$\Set{0, 1, \ldots}$ with parameter $\theta_k \defas \rho_k / (1 + \rho_k)$.
Therefore, letting $M_k$ have the distribution of the number of mutations
along a $k$-excursion of $\tilde{X}$, we have
\begin{equation} \label{eqProofM}
  M_k \;\;\overset{d}{=}\;\;
  \sum_{i = 1}^{N_k} M_{k + 1}^{(i)} \;+\; \mathrm{Ber}(p_k)\,,
\end{equation}
where $M_{k+1}^{(i)}$ are independent copies of $M_{k+1}$ that are also
independent of $N_k$, and $\mathrm{Ber}(p_k)$ is a Bernoulli variable that is
independent of everything else.

Applying Wald's formula to Equation~\eqref{eqProofM} gives
\[
  \Expec{M_k} \;=\; \frac{1}{\rho_k} \, \Expec{M_{k+1}} \;+\; \frac{\mu}{\rho_k}\,, 
\]
and solving this first-order linear recurrence yields the formula
for the expected value of $M \overset{d}{=} M_1$.

Let us now turn to the generating function of $M$ and let 
$g_k(z) \defas \Expec*{\normalsize}{z^{M_k}}$,
\newtext{defined for $\Abs{z}<R_k$, where $R_k=\sup\{r\geq 0 :  \Expec*{\normalsize}{r^{M_k}}<\infty\}$.
Note that since $M_k$ stochastically dominates $M_{k+1}$, the $R_k$ are nondecreasing.}
Then, for all $z$ such that
$\Abs{(1 - \theta_k)\,g_{k+1}(z)} < 1$ \newtext{-- note that this is true for all $\Abs{z}<1$ --} we have
\begin{align*}
  g_{k}(z) \;
  &=\; (1 - p_k + p_k z)\,
    \sum_{i \geq 0} g_{k+1}(z)^i \, \theta_k\,(1 - \theta_k)^i \\
  &=\; \frac{(1 - p_k + p_k z) \,\theta_k}{1 - (1 - \theta_k)\, g_{k+1}(z)}\\[1ex]
  &=\; \frac{\alpha + (k-1)\beta + \mu z}{1 + \rho_k - g_{k+1}(z)}\,.
\end{align*}
This gives the representation of $g = g_1$ as the continued fraction of the
theorem.

\newtext{We now show that $g$ extends to a meromorphic function on all of $\C$}.
Note that $M_k$ is stochastically
dominated by~$H_k$, the hitting time of~0 by the simple random walk
started from~1 that goes up with probability
$1 - \theta_k$ and down with probability $\theta_k$,
independently of its current position. A standard calculation
(see e.g.~\cite[Section~6.4]{BoNg96}) shows that the probability generating
function of $H_k$ is
\[
  h_k(z) \;=\;
  \frac{1 - \sqrt{1 - 4 \theta_k (1-\theta_k) z^2}}{2 (1 -\theta_k) z} \,, 
\]
whose power series expansion around zero has a radius of convergence equal to
$(4 \theta_k (1 - \theta_k))^{-1/2}$.
Moreover, $\smash{\Prob{H_k < \infty} = h_k(1^-) =
\frac{1 - \Abs{1 - 2\theta_k}}{2(1 - \theta_k)}}$
is equal to 1 \linebreak for all~$k$ large enough, and $\Prob{H_k < \infty}=1$ implies that
$\Expec*{\normalsize}{z^{H_k}} = h_k(z)$ inside the disk of convergence of
$\Expec*{\normalsize}{z^{H_k}}$.
Since $\Expec*{\normalsize}{z^{M_k}} \leq \Expec*{\normalsize}{z^{H_k}}$
for $z \geq 1$ and since 
$(4 \theta_k (1 - \theta_k))^{-1/2} \to +\infty$
as $k$ goes to infinity, this shows that for any $r > 0$ there exists
$k_r$ such that $g_{k_r}$ is analytic on
$\mathds{D}_r = \Set{z \suchthat \Abs{z} < r}$. It then follows by induction
that $g_{k_r - 1}, \ldots, g_1$ are meromorphic on $\mathds{D}_r$.

Finally,  recall that the \emph{dominant} singularities of a function that is
analytic at~0 are those singularities that are closest to the origin.
To see that the dominant pole of~$g$ is in $R > 1$, note that since the
power series representation of~$g$ around the origin has nonnegative
coefficients, Pringsheim's theorem (see e.g.\ \cite[Theorem~IV.6]{FlSe09})
ensures that it has a dominant singularity in $\OpenInterval{0, +\infty}$.
Since $g(1) = 1$ is finite and since all singularities of $g$ are poles,
this means that $g$ has no singularity in~$1$. Hence, $g$ has a dominant pole in
$R$ for some $R > 1$.
\end{proof}

One of the advantages of the expression of $g$ as a generalized continued
fraction is that this makes its numerical evaluation straightforward and
very efficient. Indeed, modified convergents of this continued fraction
provide us with upper and lower bounds on $g$, as the next proposition shows.
The rapid convergence of these bounds is illustrated in
Figure~\ref{figConvergentsG}.

\begin{proposition} \label{propBoundsG}
For all $z \in \ClosedInterval{0, 1}$ and all $n \geq 1$,
\[
  \scalebox{0.8}{$%
  \cfrac{\alpha + \mu z}{1 + \rho_1 -
  \cfrac{\alpha + \beta + \mu z}{1 + \rho_2 -
  \cfrac{\alpha + 2\beta + \mu z}{\ddots
  \raisebox{-2ex}{$\;-\cfrac{\alpha + (n-1)\beta + \mu z}{
    1 + \rho_n - \;\bar{\!g}_n(z)}$}}}}
  $}
  \;\leq\; g(z) \,\leq\,
  \scalebox{0.8}{$%
  \cfrac{\alpha + \mu z}{1 + \rho_1 -
  \cfrac{\alpha + \beta + \mu z}{1 + \rho_2 -
  \cfrac{\alpha + 2\beta + \mu z}{\ddots
  \raisebox{-2ex}{$\;-\cfrac{\alpha + (n-1)\beta + \mu z}{\rho_n}$}}}} 
  $}
\]
where
\[
  \bar{\!g}_n(z) \;=\; \frac{1}{2}
  \mleft(1 + \rho_n - \sqrt{(1 - \rho_n)^2 - 4\mu(z - 1)}\mright) \,.
\]
Letting $\altoverline{R}_n \geq 1$ denote the radius of convergence of
the power series expansion around~0 of the left-hand side, for $n$ large enough
the reversed inequalities hold for
$z \in \COInterval[\normalsize]{1, \altoverline{R}_n}$.
Moreover, the difference between the right-hand side and
the left-hand side is $O(n^c \beta^{-n} / n!)$, where
$c = 1 - \frac{\alpha + \mu}{\beta}$, 
uniformly in $z \in \ClosedInterval{0, 1}$.
\end{proposition}

\begin{figure}[t]
  \centering
  \includegraphics[width=0.95\linewidth]{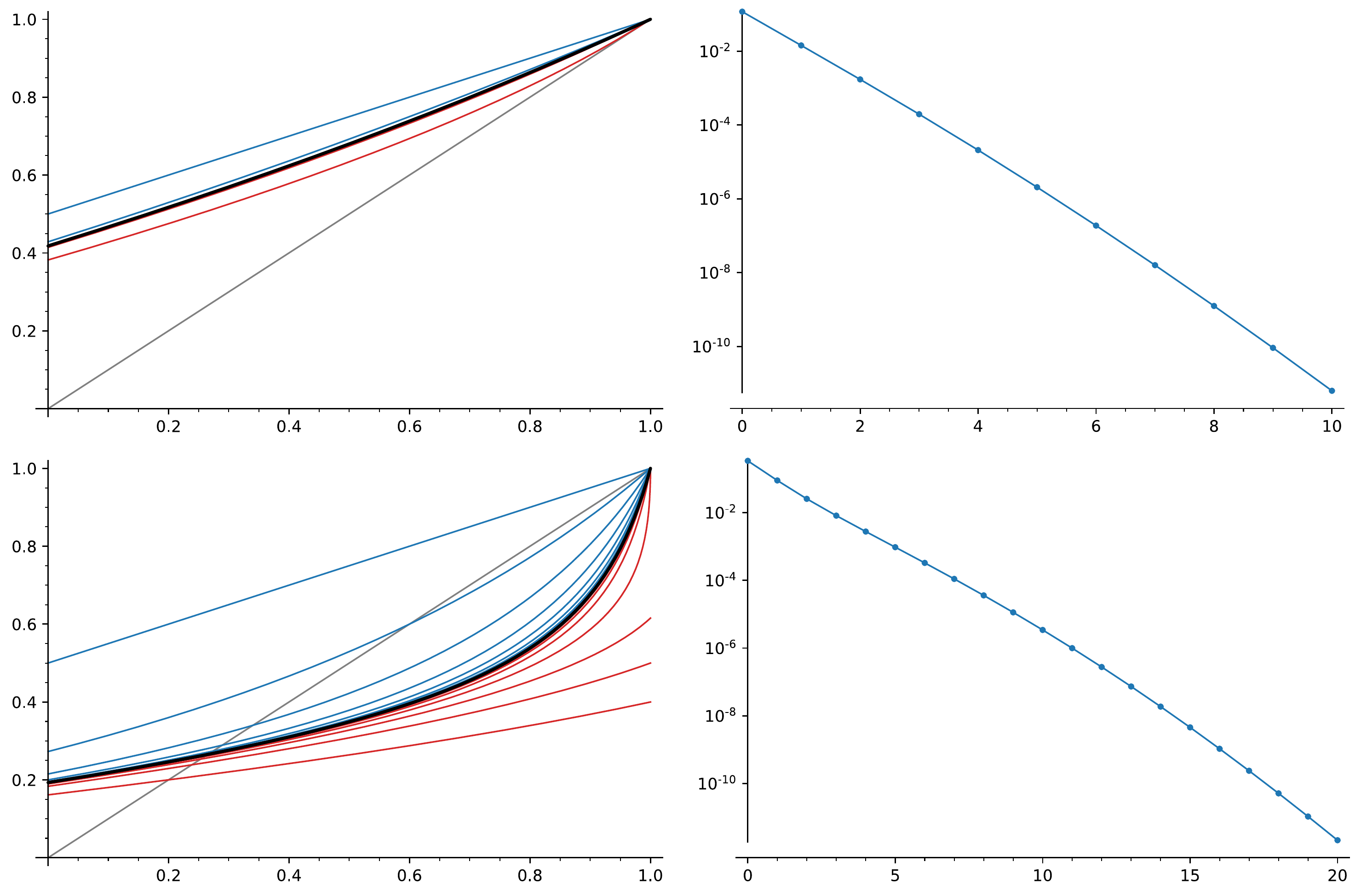}
  \caption{Left, the generating function of $M$, in black, along with the
  modified convergents of Proposition~\ref{propBoundsG} giving upper (blue)
  and lower (red) bounds; right, supremum on $\ClosedInterval{0, 1}$ of the
  difference between the upper and the lower bound, as a function of~$n$.
  Top, subcritical regime, with $(\alpha, \beta, \mu) = (1, 1, 1)$;
  bottom, supercritical regime, with $(\alpha, \beta, \mu) = (0.2, 0.2, 0.2)$.}
  \label{figConvergentsG}
\end{figure}

\begin{proof}
We give a probabilistic proof. Although it is possible to give a shorter
analytic proof, we think that the probabilistic one is more
instructive.

Let $\altoverlinesup{M}{n}$ denote the number of mutations associated to
a modified version~$\altoverlinesup{X}{n}$ of the process $X$,
where coalescences happen at rate
$n(n-1) \beta$ instead of $k(k-1) \beta$ whenever
$\altoverlinesup{X}{n} = k \geq n$.
Let $\altunderlinesup{M}{n}$ denote the number of mutations 
that correspond to transitions from $k$ to $k-1$ with $k \leq n$ in the
original process~$X$. For all $n \geq 1$, we have the stochastic dominations
\begin{equation} \label{eqStocDomMProof}
  \altunderlinesup{M}{n}
  \;\overset{d}{\leq}\; M \;\overset{d}{\leq}\;
  \altoverlinesup{M}{n} .
\end{equation}
Let $\altoverline{G}_n(z)$ and $\altunderline{G}_{n}(z)$ denote the
left- and right-hand sides of the inequality of the proposition, respectively.
The same reasoning as in the proof of Theorem~\ref{thmM} shows that
$\altoverline{G}_n$ and $\altunderline{G}_{n}$ are, respectively,
the generating functions of $\altoverlinesup{M}{n}$ and of~$\altunderlinesup{M}{n}$.
Note however that, in the case of $\altoverlinesup{M}{n}\!$, for small values
of $n$ there can be a positive probability that $\altoverlinesup{X}{n}$
never hits 0 -- in which case it is not possible to decompose its trajectory
into finite excursions. Nevertheless, letting $A_{k}$ be the event
that $\altoverlinesup{X}{n}$ started from $k$ never hits~$k-1$ and
$\altoverlinesup{N}{n}_{\!\smash{k}}$ a geometric variable with parameter
$\rho_{k \wedge n} / (1 + \rho_{k \wedge n})$, we have
\[
  \Prob{A_{k}} \;=\;
  \Prob{\!\!\bigcup_{i = 1}^{\;\;\altoverlinesup{N}{n}_{\!\smash{k}}}
  \!\!A_{k+1}^{(i)}} ,
\]  
where $A_{k+1}^{(i)}$ are independent realizations of $A_{k+1}$ that
are also independent of~$\altoverlinesup{N}{n}_{\!\smash{k}}$.
Since, up to a negligible event, $A_k$ and
$\Set*{\altoverlinesup{M}{n}_{\!\smash{k}} = \infty}$ are equal, 
this means that Equation~\eqref{eqProofM} holds for
$\altoverlinesup{M}{n}_{\!\smash{k}}$, mutatis mutandis,
even when $\Prob*{\normalsize}{\altoverlinesup{M}{n}_k = \infty} > 0$. Finally,
the expression of $\,\bar{\!g}_n$ is obtained by solving
\[
  \,\bar{\!g}_n(z) \;=\;
  \frac{\alpha + (k-1)\beta + \mu z}{1 + \rho_k - \;\bar{\!g}_n(z)}\,, 
\]
since, by construction, $\altoverlinesup{M}{n}_{\!\smash{n+1}}
\overset{d}{=} \altoverlinesup{M}{n}_{\!\smash{n}}$.

Being generating functions, $\altoverline{G}_n$ and $\altunderline{G}_{n}$ are
analytic at~0 with radius of convergence at least~1, and we have
$\Expec*{\normalsize}{z^{\altoverlinesup{M}{n}}} = \altoverline{G}_n(z)$ and
$\Expec*{\normalsize}{z^{\altunderlinesup{M}{n}}} = \altunderline{G}_{n}(z)$
for all $z \in \COInterval{0, 1}$.
Combining this with~\eqref{eqStocDomMProof} and taking the limit
$z \to 1^-$ proves the inequality of the
proposition for all $z \in \ClosedInterval{0, 1}$.

For $z > 1$, since $\altunderlinesup{M}{n}$ is stochastically dominated by $M$
and since $M$ is almost surely finite, for all $n$ we have
$\altunderline{G}_{n}(z) = \Expec*{\normalsize}{z^{\altunderlinesup{M}{n}}}
\leq \Expec*{\normalsize}{z^M} = g(z)$ for all ${z \in \COInterval{1, R}}$, where
$R$ is the radius of convergence of $g$ around~0.
Similarly, for $n$ large enough
$\Prob*{\normalsize}{\altoverlinesup{M}{n} < +\infty} = 1$ and thus
$g(z) \leq \altoverline{G}_{n}(z)$ for all
$z\in \COInterval[\normalsize]{1, \altoverline{R}_n}$, where $\altoverline{R}_n$
is the radius of convergence of $\altoverline{G}_{n}$ around~0. Note however
that for small $n$ we can have $\altoverline{G}_n(1) =
\Prob*{\normalsize}{\altoverlinesup{M}{n} < +\infty} < 1$, and thus
$\altoverline{G}_n(z) < g(z)$ for
$z\in \COInterval[\normalsize]{1, \altoverline{R}_n}$.

Finally, to see that $\sup_{z \in \ClosedInterval{0, 1}}
\Abs[normalsize]{\altunderline{G}_{n}(z) - \altoverline{G}_{n}(z)} =
O(n^c \beta^{-n}/n!)$,
note that
\[
  \frac{\alpha + (k-1)\beta + \mu z}{1 + \rho_k - A}
  - \frac{\alpha + (k-1)\beta + \mu z}{1 + \rho_k - B} 
  \;=\;
  \frac{(\alpha + (k-1) \beta + \mu z)(A-B)}{(1 + \rho_k - A)(1 + \rho_k - B)}\,, 
\]
so that for all $z, A, B \in \ClosedInterval{0, 1}$,
\[
  \Abs{\frac{\alpha + (k-1)\beta + \mu z}{1 + \rho_k - A}
  - \frac{\alpha + (k-1)\beta + \mu z}{1 + \rho_k - B} }
  \;\leq\; \frac{1}{\rho_k}\, \Abs{A - B}\,.
\]
Since $\Abs{\;\bar{\!g}_n(z) - 1} \leq 1$, an immediate induction gives
\[
  \sup_{z \in \ClosedInterval{0, 1}}
  \,\Abs{\,\altunderline{G}_{n}(z) - \altoverline{G}_{n}(z)} \;\leq\;
  \prod_{k = 1}^n \frac{1}{\rho_k} \;\sim\;
  \Gamma(\tfrac{\alpha + \mu}{\beta})\; n^{1 - \frac{\alpha + \mu}{\beta}} \, 
  \beta^{-n} \,/\, n! \,, 
\]
finishing the proof.
\end{proof}

Besides numerical evaluation, the bounds of Proposition~\ref{propBoundsG} can
be used to obtain rigorous bounds on the probability of extinction
of the model. For instance, taking $n = 2$ for the left-hand side, 
$n = 3$ for the right-hand side, and finding the corresponding fixed points,
we get the simple bounds
\[
  \scalebox{0.9}{$\displaystyle%
  \frac{\alpha}{2\mu}
  \mleft(\beta + \mu - 1 +\sqrt{(\beta + \mu - 1)^2 + 4\mu}\mright)
  $}
  \;\leq\; p_{\mathrm{ext}} \;\leq\;
  \scalebox{0.9}{$\displaystyle%
  \frac{\alpha \,((\alpha + \beta + \mu) (\alpha + 2\beta + \mu) + \mu)}{
    \mu \, (1 + 2 \alpha + 2 \beta + \mu)}\,.
  $}
\]
In fact, it is possible to get one such upper bound up to $n = 9$.
However, the resulting expression, although very sharp, is too complex to be of
any practical use -- so we do not reproduce it here.

Let us now point out two immediate consequences of Theorem~\ref{thmM} that
will be useful in the rest of this document.

\begin{corollary}~ \label{corExpTiltM}
\begin{mathlist}
  \item $M$ has finite exponential moments:
    \[
      \exists \epsilon > 0 \st\quad \Expec{e^{\epsilon M}} < +\infty.
    \]
  \item There is an exponential tilt of $\TweakSpacingL{M}$ with mean 1:
    \[
      \exists \zeta > 0 \st \quad
      \frac{\Expec*{\normalsize}{M \zeta^M}}{\Expec{\zeta^M}} = 1.
    \]
\end{mathlist}
\end{corollary}

\begin{proof}
(i) is merely saying that the radius of convergence of $g$ is greater than~1;
(ii)~is a classic consequence of the fact that $g(s) \to +\infty$ as
$s \uparrow R$, see e.g.\ point~(iv) of Lemma~3.1 in \cite{Jan12}.
For the sake of completeness, we recall the proof here:
for any $a \geq 0$ and any $s \in \COInterval{0, R}$,
\[
  \frac{\Expec*{\normalsize}{M s^M}}{\Expec{s^M}} - a \;\geq\;
  \frac{\Expec{(M - a) s^M \Indic{M < a}}}{\Expec{s^M}} \,.
\]
Since $\Abs[\normalsize]{\Expec*{\normalsize}{(M - a) s^M \Indic{M < a}}} < a R^{\,a}$
and $\Expec*{\normalsize}{s^M} \to +\infty$ as $s \uparrow R$, the
right-hand side of this inequality goes to 0 as $s \uparrow R$. Therefore,
$\lim_{s\uparrow R}\Expec*{\normalsize}{M s^M} / {\Expec*{\normalsize}{s^M}} \geq a$
for all $a\geq 0$, i.e.\
$\Expec*{\normalsize}{M s^M} / {\Expec*{\normalsize}{s^M}} \to +\infty$
as $s \uparrow R$. The existence of $\zeta \in \OpenInterval{0, R}$
such that $\Expec*{\normalsize}{M \zeta^M} / {\Expec*{\normalsize}{\zeta^M}} = 1$
follows by continuity.
\end{proof}

The main consequence of Corollary~\ref{corExpTiltM} is that, for all $\alpha,
\beta, \mu > 0$, when conditioned to have $n$ vertices
$\mathscr{T}$ is distributed as a critical Galton--Watson tree conditioned to
have $n$ vertices. We will come back to this in Section~\ref{secCRT}.

Finally, we close this section with a brief remark about $\mathcal{M}$, the point
process of mutation times.
We state it as a proposition for ease of reference, but it
is not specific to our setting and follows readily from the infinitesimal
definition of a continuous-time Markov chain -- so we omit the proof.

\begin{proposition} \label{propIntensityM}
Let $\mathcal{M}$ be the point process on $\R_+$ giving the birth
times of the children of the first color (that is, every atom
$t \in \mathcal{M}$ corresponds to a mutation of a lineage of the first color).
The intensity measure of $\mathcal{M}$ is $\mu\,\Expec{X_t}\,dt$ -- that is,
for any Borel set $A \subset \R_+$, 
\[
  \Expec*{\big}{\#(\mathcal{M}\cap A)}  
  \;=\;
  \mu \int_A \!\Expec{X_t} \, dt \,.
\]
In particular, $\Expec{M} = \mu\, \Expec{\int_0^{\infty} \! X_t \, dt}$.
\end{proposition}

\subsection{Growth rate of the number of lineages} \label{secGrowthRate}

Let us start by focusing on the number the colors. We will turn to
the number of lineages at the end of the section. Let $Z_t$ denote
the number of colors alive at time~$t$. The process $Z = (Z_t : t \geq 0)$ is a
Crump--Mode--Jagers process, or CMJ for short, where individuals give
birth according to a point process distributed as~$\mathcal{M}$, the point
process of mutations of the first color; and die after a time distributed
as $T$, the extinction time of the logistic branching process started from~1.

The next proposition is an application of standard results from the theory
of CMJ processes \cite{CrMo68,CrMo69,Jag75}. Essentially, CMJ processes
grow\,/\,decrease exponentially with a growth rate
known as their \emph{Malthusian parameter}. Proposition~\ref{propGrowthRateCMJ}
recalls the precise meaning of this ``exponential growth''
and gives the usual, generic characterization of the growth rate.
Another characterization -- one that is specific to our setting and makes
it possible to compute the growth rate numerically --
will be given in Proposition~\ref{propCharGrowthRate}.

\begin{proposition} \label{propGrowthRateCMJ}
Let $\lambda$ be the unique solution of
\[
  \Expec{\sum_{t \in \mathcal{M}}\! e^{-\lambda t}} \;=\; 1
  \quad\text{or, equivalently,}\quad
\mu\,\Expec{\int_0^\infty\! X_t\, e^{-\lambda t} \, dt} = 1\,.
\]
Then, $\TweakSpacingL{\lambda}$ has the same sign as ${\TweakSpacingL{\Expec{M}-1}}$.
Moreover,
\begin{mathlist}
\item If $\TweakSpacingL{\Expec{M} > 1}$, \newtext{then the process $Z$ of the number of colors satisfies}
  \[
    e^{-\lambda t}\, Z_t \;\;\tendsto[\;L^2, \text{ a.s.}\;]{t\to\infty} \;\; W, 
  \]
  where $\TweakSpacingL{W}$ is a random variable with $\TweakSpacingL{\Expec{W}} = 1$
  that is almost surely positive on non-extinction, i.e.\ on the event
  \{$Z_t > 0$ for all $\TweakSpacingL{t}$\},
  and where the convergence holds both almost surely and in mean square.
\item If $\Expec{M} \leq 1$, then $\Expec{Z_t} \sim Ce^{\lambda t}$
as $t\to\infty$ for some constant $C > 0$.
\end{mathlist}
\end{proposition}

\begin{proof}
The fact that the two characterizations of $\lambda$ are equivalent
follows from Proposition~\ref{propIntensityM} and Campbell's formula.
The uniqueness of $\lambda$ is standard, and so is the fact that $\lambda$
is guaranteed to exist whenever $\Expec{M} \geq 1$, see \cite[Section~6]{CrMo68}.
To see that $\lambda$ also exists when $\Expec{M} < 1$, letting $\tau$ denote
the time of the first jump of $X$ one can consider
the random variable $Y_{\eta}$ that takes the value $e^{-\eta\tau}$ if the
first jump of $X$ is a mutation, and 0 otherwise. Thus, 
$Y_\eta \leq \sum_{t \in \mathcal{M}} e^{-\eta t}$. A~straightforward
calculation then shows that
$\Expec{Y_\eta} = \mu \int_0^{\infty} e^{-(1 + \alpha + \mu + \eta)t} dt\to +\infty$
as $\eta$ decreases to $-(1+\alpha+\mu)$ and therefore can be made greater
than~1 by decreasing $\eta$. By the same reasoning as in the proof of the
meromorphy of $g$ in Theorem~\ref{thmM}, $\eta \mapsto \Expec{\sum_{t \in
\mathcal{M}}\! e^{-\eta t}}$ cannot jump to infinity. Since it is equal to
$\Expec{M} < 1$ when $\eta = 0$, the existence of $\lambda$ follows by
continuity.

Since Theorem~\ref{thmM} entails that $\Expec{M^2} < \infty$, the mean-square convergence in
point~(i) follows immediately from \cite[Theorem~3.1]{CrMo69}. Similarly, the
almost sure convergence follows from \cite[Theorem~3.2]{CrMo69}, provided
that the intensity function of $\mathcal{M}$, namely $m\colon t\mapsto \mu\, \Expec{X_t}$,
is differentiable and such that $\int_0^\infty |m'(t)|^p\, dt < \infty$
for some $p > 1$. Now, since
\[
  \frac{d}{dt} \Expec{X_t} \;=\;
  (1 - \alpha - \mu)\, \Expec{X_t} \;-\; \beta\, \Expec*{\big}{X_t (X_t-1)}
\]
and that $X_t$ is integer-valued, we have $|m'(t)| < K\, \Expec{X_t^2}$ for
some constant $K$. Thus, by Jensen's inequality,
to complete the proof of point~(i) it suffices
to show that $\int_0^\infty \Expec{X_t^p}\, dt < \infty$ for some $p > 2$.
Standard calculations, again using the decomposition of the trajectory
of $X$ into excursions, as in the proof of Theorem~\ref{thmM}, show that
\[
  \int_0^\infty \!\Expec{X_t^p}\, dt \;=\;
  \sum_{j \geq 1} \mleft(\prod_{k=1}^j \frac{1}{\rho_k}\mright) j^p. 
\]
Since $\prod_{k=1}^j \rho_k^{-1} = O(n^c \beta^{-n} / n!)$ for some constant
$c$, as already seen in the proof of Proposition~\ref{propBoundsG}, the
integral is finite for all $p$, finishing the proof of point~(i).

Finally, letting $T$ denote the extinction time of $X$,
$\Indic{T > t} \leq X_t$ and therefore 
$\Prob{T > t} \leq \Expec{X_t}$. Since, by definition of $\lambda$,
$\int_0^\infty \Expec{X_t}\, e^{-\lambda t} \, dt = 1/\mu$, we have
$\int_0^\infty \Prob{T > t}\, e^{-\lambda t} \, dt < \infty$ and thus
point~(ii) follows from \cite[Theorem~6.2]{CrMo68}.
\end{proof}

We now give another characterization of $\lambda$, which makes use of the
measure $\nucirc$ introduced in Proposition~\ref{propNuCircXfromMut}. Here, we let
$\ExpecCirc{\,\cdot\,}$ denote the expectation for the process $X$ started from a
random state with distribution $\nucirc$.

\begin{proposition} \label{propCharGrowthRate}
Let $\TweakSpacingL{T}$ denote the extinction time of $X$. Then,
the growth rate~$\lambda$ is the unique solution of
\[
 \Expec{M}\,  \ExpecCirc{e^{-\lambda T}} \;=\; 1.
\]
Furthermore,
\[
  \Expec{M}\,  \ExpecCirc{e^{-\lambda T}} \;=\;
  \mu \sum_{j\geq 1}\prod_{k=1}^j \frac{f_k(\lambda)}{\rho_k},
\]
where $f_k(\lambda)$ is given by the continued fraction
\[
  f_k(\lambda) \;=\; \cfrac{\rho_k}{1 + \rho_k + \frac{\lambda}{k} -
  \cfrac{\rho_{k + 1}}{1 + \rho_{k+1} + \frac{\lambda}{k+1} - \ddots}}
  \;=\;
  -\kettenbruch_{i=k}^\infty \frac{-\rho_i}{1 + \rho_i+\frac{\lambda}{i}}\,.
\]
\end{proposition}

The interest of this proposition is that, since the functions $f_k(\lambda)$
can be evaluated efficiently, so can $\Expec{M}\,\ExpecCircNormSize{e^{-\lambda T}}$.
This makes it straightforward to determine $\lambda$ numerically, for instance
using the bisection method.

\begin{proof}
The first part of the proposition is a consequence of the
standard characterization of $\lambda$, which is recalled in Proposition~\ref{propGrowthRateCMJ},
and of the construction of the process $\Xmut$ given
in Proposition~\ref{propNuCircXfromMut}. Indeed, first note that
$\Expec*{\normalsize}{\sum_{t \in \mathcal{M}} e^{-\lambda t}} =
\Expec{M}\, \Expec*{\normalsize}{e^{-\lambda U} \biasedby M}$,
where $U$ is a uniform atom of $\mathcal{M}$, and also corresponds to minus
the infimum of the times for which $\Xmut$ is defined. Second,
recall that $\Xmut$ is distributed as $X'\paste X''$, where $X'$ is distributed as~$X$
started from~$\nucirc$, and that in this construction $U$ corresponds to the
extinction time of~$X'$. As a result,
$\Expec*{\normalsize}{e^{-\lambda U} \biasedby M} = \ExpecCircNormSize{e^{-\lambda T}}$.

To express $\Expec{M}\, \ExpecCircNormSize{e^{-\lambda T}}$ as a function
of $\lambda$, for $k\geq 1$ let
\[
  f_k(\lambda) \;=\; \Expec{e^{-\lambda T_{k-1}} \given X_0 = k},
\]
where $T_{k-1}$ denotes the hitting time of $k-1$ by $X$.
By the strong Markov property,
\[
  \Expec{M}\, \ExpecCirc{e^{-\lambda T}} \;=\;
  \Expec{M}\sum_{j\geq 1} \nucirc(j)\prod_{k = 1}^j f_k(\lambda).
\]
From the expression of $\Expec{M}$ in Theorem~\ref{thmM}, we see that the normalizing
constant in Equation~\eqref{eqDefNuCirc} where $\nu_\circ$ is defined is
$C = \mu / \Expec{M}$, from which deduce that
\[
  \Expec{M}\, \nucirc(j) \;=\; \mu\prod_{k=1}^j\frac{1}{\rho_k} \,.
\]
Therefore, to finish the proof it only remains to show that
\[
  f_k(\lambda) \;=\;
  \frac{\rho_k}{1 + \rho_k + \frac{\lambda}{k} - f_{k+1}(\lambda)} \,.
\]
The reasoning is exactly the same as for the expression of the generating function
of~$M$ in Theorem~\ref{thmM}, so we do not detail it.
\end{proof}

So far, we have been focusing on the growth rate of $Z_t$, the number of colors
at time~$t$. But from a biological point of view it is arguably more natural
to consider~$\Upsilon_t$, the number of lineages at time~$t$. We therefore close
this section with a proposition showing that the asymptotic
growth rate of the number of lineages is the same as that of the number of
colors. For simplicity we do not try to state the results in full generality.

\begin{proposition}
Let $\lambda$ be the growth rate of $Z$, as given in Proposition~\ref{propGrowthRateCMJ}
and~\ref{propCharGrowthRate}, and let $\Upsilon_t$ be the number of lineages
alive at time~$t$. If $\lambda > 1$, then
\[
  e^{-\lambda t}\, \Upsilon_t
  \;\; \tendsto[\text{a.s.}]{\;t \to \infty\;} \;\;
  \Xi\,, 
\]
where $\Xi$ is a random variable that is almost surely positive on non-extinction.
\end{proposition}

\begin{proof}
Again, this is a standard application of general results for CMJ processes
counted with a random characteristic, see e.g.~\cite{Ner81}.
More specifically, let the characteristic associated to each color be the
number of lineages of that color. Note that the characteristic of a color
is not independent of its lifespan and of its reproduction, but that the
characteristics of different colors are independent. Since 
$\Expec{M} < \infty$, Condition~5.1 in \cite{Ner81} holds with
$g(t) = e^{-\lambda t}$. Moreover, by using the same argument as for $M$
it is straightforward to show that the total number of jumps of $X$ has finite
exponential moments. Since $X$ has bounded jumps, this implies that
$\Expec{\sup_t X_t} < \infty$, and so Condition~5.2 in \cite{Ner81} holds with
$h(t) = e^{-\lambda t}$. As a result, the proposition follows
from \cite[Theorem~5.4]{Ner81}.
\end{proof}

\section{Convergence to the CRT} \label{secCRT}

In this section, we study the large-scale geometry of $\mathscr{G}$.
We will show that, after being conditioned to have $n$ colors and appropriately
rescaled, as $n$ goes to infinity $\mathscr{G}$ converges in distribution to
the Brownian continuum random tree (CRT) for the rooted
Gromov--Hausdorff--Prokhorov topology.

The Brownian CRT, introduced by Aldous in~\cite{Ald91CRT}, is the universal
scaling limit of critical Galton--Watson trees when the offspring distribution
has finite variance.  Since its first description as a random subset of
$\ell^1$ obtained by successively glueing segments of random lengths along
orthogonal directions, it has become standard (see e.g.\ \cite[Section
2]{LeGall05} and \cite[Section 2.4]{Eva08}) to view
it as the random rooted compact metric probability space $(\mathscr{C}, r, d,
\lambda)$ defined in the following way:
\begin{itemize}
  \item Take a standard Brownian excursion $(e(t))_{t\in [0,1]}$.
  \item Define a pseudo-metric $d_e$ on $\ClosedInterval{0, 1}$ by
    $d_e(x,y) = e(x)+e(y) - 2\inf_{z\in [x,y]}e(z)$, where $[x,y]$ is a slight
    abuse of notation for the segment $[x\wedge y, x \vee y]$.
  \item Let $(\mathscr{C},d)$ be the quotient metric space obtained by
    identifying the points of~$[0,1]$ at distance zero for $d_e$, and let the root
    $r\in\mathscr{C}$ be the equivalence class of $0$.
  \item Let $\lambda$ be the pushforward on $\mathscr{C}$ of the Lebesgue
    measure on $\ClosedInterval{0, 1}$.
\end{itemize}

The rest of this section is organized as follows:
first, we give a brief reminder about convergence in the rooted
Gromov--Hausdorff--Prokhorov topology.  Coming back to our
model, we then detail how to condition $\mathscr{G}$ on having~$n$ colors, and
we introduce some notation. Finally, we prove a series of technical lemmas
which, when put together, readily give us the desired convergence to the CRT.

\subsection{The rooted Gromov--Hausdorff--Prokhorov distance} \label{secTopology}

Here we recall, mostly without proof, the minimal set of notions about
convergence of metric probability spaces that are needed to state and prove
our results. More detailed treatments can be found, e.g, in \cite[Section
6]{Mie09} or in \cite[Section~4]{Eva08}.
In particular, Proposition~\ref{propKeyToCRT} below provides a
general-purpose, simple way to establish convergence to the CRT by following
the approach used in \cite{Stu22}. See also \cite{SS18a} for related results.

Since our network $\mathscr{G}$ has a distinguished point, namely the point
that corresponds to the first lineage at time~$0$, it is natural to work with
a rooted version of the Gromov--Hausdorff--Prokhorov distance.
We adapt the definition of \cite[Section~6.2]{Mie09} to the rooted setting: let
$\mathbb{M}$ be the set\footnote{\,It is not obvious that $\mathbb{M}$ can be
defined as a set, because the class of compact metric spaces is not a set.
However, since a compact metric space has cardinal at most $\mathfrak{c} =
\Card \R$, all isometry classes are obtained by considering subsets of $\R$
endowed with a metric and a measure -- and these do indeed form a well-defined
set.} of isometry classes of rooted compact metric probability spaces
$\mathscr{X}=(\mathscr{X}, r, d, \lambda)$, where $r\in \mathscr{X}$ is called
the root of $\mathscr{X}$; $d$ is a metric on $\mathscr{X}$; and $\lambda$
is a probability measure on $\mathscr{X}$.
The rooted Gromov--Hausdorff--Prokhorov distance $\GHPdist(\mathscr{X}, \mathscr{X}')$
between two elements 
$(\mathscr{X}, r, d, \lambda)$ and $(\mathscr{X}', r', d', \lambda')$ of~$\mathbb{M}$
is defined as the infimum of the $\epsilon>0$ such that there exists a well-defined metric
$\delta$ on the disjoint union $\mathscr{Y}\defas\mathscr{X}\sqcup\mathscr{X}'$ satisfying:
\begin{mathlist}
\item For all $x,y\in \mathscr{X}$ and $ x',y'\in \mathscr{X}'$, $\delta(x,y)=d(x,y)$ and $\delta(x',y')=d'(x',y')$.
\item $\delta(r,r')\leq \epsilon$.
\item The Hausdorff distance $\delta_{\mathrm{H}}(\mathscr{X},\mathscr{X}')$
  between $\mathscr{X}$ and $\mathscr{X}'$ is at most $\epsilon$; in other
  words, $\mathscr{X}' \subset \mathscr{X}^{\epsilon}$ and $\mathscr{X} \subset
  (\mathscr{X}')^\epsilon$, where $A^\epsilon = \Set{y\in
  \mathscr{Y} \suchthat \exists x\in A, \delta(x,y) < \epsilon}$.
\item Extending $\lambda$ and $\lambda'$ to $\mathscr{Y}$ via
  $\lambda_{\mathscr{Y}}(A) = \lambda(A\cap \mathscr{X})$ and
  $\lambda'_{\mathscr{Y}}(A) = \lambda'(A\cap \mathscr{X}')$, 
  the Prokhorov distance between $\lambda_{\mathscr{Y}}$ and $\lambda'_{\mathscr{Y}}$
  is at most $\epsilon$, i.e.\ for all Borel subset $A\subset \mathscr{Y}$, we have
  $\lambda_{\mathscr{Y}}(A) \leq \lambda'_{\mathscr{Y}}(A^{\epsilon})+\epsilon$.
\end{mathlist}

The space $(\mathbb{M},\,\GHPdist)$ is a complete separable
metric space (see e.g.\ \cite[Theorem~6 and Proposition~8]{Mie09} for a proof in the
unrooted setting; we let the interested reader check that the proof carries over
to the rooted setting, and refer them to \cite[Section 4.3.3]{Eva08} where
this is done for the Gromov--Hausdorff distance).

Because our metric spaces are tree-like, in our setting it will be more
convenient to work with height processes than to manipulate
$\GHPdist$ directly. Let us start by recalling how one can obtain a
metric space from a càdlàg function, and introducing some notation.
\newtext{Note that this construction is simply a generalization of the
construction of the Brownian CRT recalled at the beginning of this section,
but with more general functions as contour processes.}

\begin{definition} \label{defRTree}
Let $h\colon \ClosedInterval{0, 1} \to \R$ be a nonnegative càdlàg
function such that $h(0) = 0$. We denote by $d_h$ the pseudometric on
$\ClosedInterval{0, 1}$ defined by
\[
  d_h(x,y) \;=\; h(x) \,+\,  h(y) \,-\; 2\!\! \inf_{z\in [x,y]}\! h(z),
\]
where, as previously, $\ClosedInterval{x, y}$ is shorthand for
$\ClosedInterval{x\wedge y, x \vee y}$.  We then denote by $\mathscr{T}_h$ the
rooted compact metric probability space obtained by: (1)~identifying points
$x, y \in \ClosedInterval{0, 1}$ such that $d_h(x, y) = 0$; (2) taking the
completion of the space with respect to $d_h$; (3) taking the
equivalence class of $0$ as the root; and (4) endowing the resulting rooted
metric space with the pushforward of the Lebesgue measure on
$\ClosedInterval{0, 1}$. This metric space is a subset of an
$\R$-tree and consists of a countable number of connected components -- see
Figure~\ref{figDefRForest} for an illustration, and e.g.\ \cite{Eva08} for an
introduction to $\R$-trees.
\end{definition}

\begin{figure}[h!]
  \centering
  \includegraphics[width=0.9\linewidth]{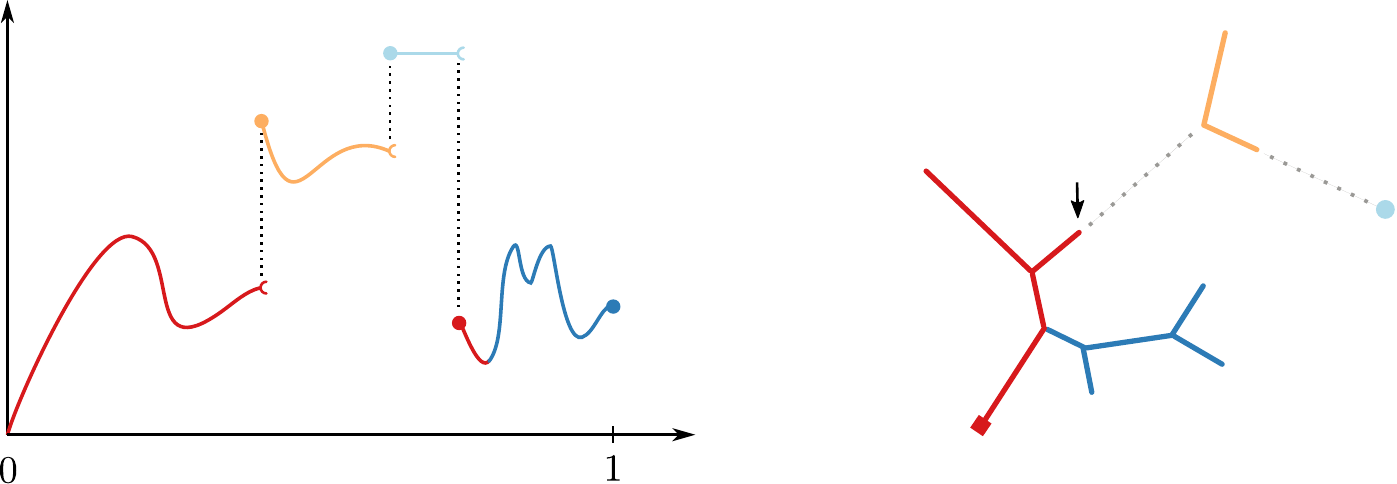}
  \caption{A nonnegative càdlàg function and the corresponding metric space,
  as given by Definition~\ref{defRTree}. The square indicates the root.
  The colors are irrelevant and are merely here to help show which
  part of the function corresponds to which part of the metric space. The
  dotted lines are not part of the metric space but are here to indicate
  how it can be embedded in an $\R$-tree. Note that we have added the tip of
  the red branch pointed at by the arrow in step (2) of the construction.}
  \label{figDefRForest}
\end{figure}

The interest of working with $\R$-trees and their height processes comes from
the following lemma, which is a straightforward extension of~\cite[Lemma~2.4]{LeGall05}.
Let us denote by $D$ the space of càdlàg functions from $[0,1]$ to $\R$ that are
also continuous at $1$, endowed with the usual Skorokhod topology~\cite{Bil99}.

\begin{lemma} \label{lemLeGall}
The map $h\in D \mapsto \mathscr{T}_h\in \mathbb{M}$ is continuous.
In other words, if $h_1, h_2, \ldots$ and $h$ satisfy the hypotheses of Definition~\ref{defRTree}, then
\[
  h_n \longrightarrow\; h\; \text{ in } D
  \quad\implies\quad
  \mathscr{T}_{h_n} \longrightarrow\; \mathscr{T}_h
  \;\text{ in }(\mathbb{M},\, \GHPdist).
\]
\end{lemma}

A self-contained proof can be found in the appendices.

From Lemma~\ref{lemLeGall}, we get the upcoming Proposition~\ref{propKeyToCRT}, which provides a
general recipe for proving convergence to the CRT in the rooted
Gromov--Hausdorff--Prokhorov topology, and is going to be our main tool for the
rest of this section.

\begin{definition} \label{defAdmissibleTraversal}
Let $(\mathscr{X}, r, d, \lambda)$ be a random rooted compact metric
probability space. A random càdlàg function $\phi\colon [0, 1] \to \mathscr{X}$ is
said to be a \emph{parametrization} of~$\mathscr{X}$ if
$\phi(0) = r$ and $\phi([0, 1])$ is almost surely dense in $\mathscr{X}$.

A parametrization $\phi$ is said to be \emph{admissible} if
can write $\phi(t) = \Phi(\mathscr{X},\Theta,t)$,
where $\Phi$ is a deterministic functional and $\Theta$ is a random variable
with values in $[0,1]$ that is independent from $\mathscr{X}$, in such a way
that the functions $t \mapsto \lambda(\phi([0, t]))$ and
$t \mapsto d(r, \phi(t))$ are well-defined random variables in the Skorokhod
space $D$.
\end{definition}

We require our parametrizations to be admissible for measurability issues --
namely, we need this assumption in order to use a variant of Skorokhod's
representation theorem in the proof of Proposition~\ref{propKeyToCRT} below.
In practice, admissibility should not be a restrictive requirement.
In our case, we will define a parametrization of our network $\mathscr{G}$
through a randomized traversal algorithm where, conditional on the network, the
additional randomness that is needed amounts to a finite number of coin tosses;
such a parametrization is readily checked to be admissible.

\begin{proposition} \label{propKeyToCRT}
Let $(\mathscr{X}_n, r_n, d_n, \lambda_n)_{n\geq 1}$ be a sequence of random
rooted compact metric probability spaces such that, for each $n\geq 1$,
there exists an admissible parametrization $\phi_n\colon[0,1]\to \mathscr{X}_n$.
Assume that, setting $h_n(t) = d_n(r_n,\phi_n(t))$:
\begin{mathlist}
  \item $\sup_{s,t\in [0,1]} \Abs[\big]{d_n(\phi_n(s),\phi_n(t)) - d_{h_n}(s, t)} \overset{d}{\longrightarrow} 0$. 
  \item $\sup_{t\in [0,1]}\Abs[\big]{\lambda_n(\phi_n([0,t]))-t\,}\overset{d}{\longrightarrow} 0$. 
  \item $(h_n(t))_{t\in [0,1]} \overset{d}{\longrightarrow} (h(t))_{t\in [0,1]}$
    for the Skorokhod topology,
    where $(h(t))_{t\in [0,1]}$ is a random càdlàg function.
\end{mathlist}
Then, $\mathscr{X}_n \overset{d}{\longrightarrow} \mathscr{T}_h$
for the rooted Gromov--Hausdorff--Prokhorov topology.
\end{proposition}

Again, this proposition is proved in the appendices.

\subsection{Conditioning on the number of colors} \label{secConditioningColors}

We now introduce some notation for conditioning $\mathscr{G}$ on its number of
colors. This notation will also be used in Section~\ref{secLocalLimit},
where we study the local weak limit of $\mathscr{G}$ conditioned to have $n$ colors.
First, recall that $\mathscr{G}$ can be viewed as the decorated
Galton--Watson tree $\mathscr{T}^\star$ obtained as follows:
\begin{enumerate}
  \item Sample a Galton--Watson tree $\mathscr{T}$ with offspring distribution $M$.
  \item Conditional on $\mathscr{T}$, decorate each vertex $k$ with the
    network $\mathscr{X}_k$ associated to an independent realization of the
    process $X$ conditioned on having $M_k$ mutations (where $M_k$ denotes the
    number of children of $k$ in $\mathscr{T}$).
\end{enumerate}
Let $A_n$ be the event \{$\mathscr{G}$ has $n$ colors\}, i.e.\ 
\{$\mathscr{T}$ has $n$ vertices\}. Since $A_n$ is a deterministic function
of $\mathscr{T}$ and that the networks $(\mathscr{X}_k)$ depend on $\mathscr{T}$
only, $\mathscr{G}_n \sim (\mathscr{G} \,|\,  A_n)$ can be
obtained by replacing $\mathscr{T}$ with $\mathscr{T}_n \sim (\mathscr{T} \,|\,
A_n)$ in step~2 of the construction above -- i.e.\ by decorating a
Galton--Watson tree with offspring distribution $M$ conditioned to
have $n$ vertices.

Note that we have not assumed that $\Expec{M}=1$. Thus,
$\mathscr{T}$ is not necessarily critical.
However, we know from Corollary~\ref{corExpTiltM} that there exists $\zeta > 0$
such that
\[
  \frac{\Expec{M \zeta^M}}{\Expec{\zeta^M}} \;=\; 1.
\]
Thus, letting $\hat{M}$ be a $\zeta$-tilt of $M$, i.e.\ a
random variable whose distribution is characterized by
\[
  \Expec{f(\hat{M})} \;=\; \frac{\Expec{f(M)\,\zeta^M}}{\Expec{\zeta^M}}
  \quad\text{for all bounded } f\colon \N \to \R, 
\]
by considering a Galton--Watson tree with offspring distribution $\hat{M}$ we
get a critical Galton--Watson tree~$\hat{\mathscr{T}}$. It is classic --
and straightforward to check by writing down the probability distributions
explicitly -- that $\mathscr{T}_n$ has the same distribution as
$\hat{\mathscr{T}}$ conditioned to have $n$ vertices.

Since we will be conditioning on $A_n$ and that
$(\mathscr{T} \,|\, A_n) \sim (\hat{\mathscr{T}} \,|\, \hat{A}_n)$, it
may not be clear at this point what the interest of working with
$\smash{\hat{\mathscr{T}}}$ instead of $\mathscr{T}$ is; this will become apparent
later -- see e.g.\ Remark~\ref{remCritical} -- but for now let us simply note
that for any nonnegative function~$f$ we have
$\Expec*{\normalsize}{f(\mathscr{\hat{T}}) \given \hat{A}_n} \leq
\Expec*{\normalsize}{f(\mathscr{\hat{T}})} / \Prob*{\normalsize}{\hat{A}_n}$
and that, as the following classic proposition shows, it is straightforward to get an
asymptotic equivalent of $\Prob*{\normalsize}{\hat{A}_n}$ as $n$ goes to infinity.

\begin{proposition} \label{propTotalProgenyGW}
Let $\hat{\mathscr{T}}$ be a critical Galton--Watson tree whose offspring
distribution has a finite variance $\sigma^2 > 0$ and is not supported on $k\N$, for any $k\geq 2$.
Let $\hat{A}_n$ denote the event
$\Set*{\hat{\mathscr{T}} \text{ has } n \text{ vertices}}$. Then,
\[
  \Prob*{\normalsize}{\hat{A}_n} \;\underset{n\to\infty}{\sim}\;
  \frac{1}{\sqrt{2\pi \sigma^2}}\, n^{-3/2}\,.
\]
\end{proposition}

This result is well-known (see e.g.~\cite[Lemma~1.11]{Kor2012} for
a more general statement), but since it is central to our study
we recall a short proof in the appendices.

\subsection{Technical lemmas} \label{secTechLemmas}

To clarify the proof of the convergence to the CRT given in
Section~\ref{secProofCRT}, we gather some of the more technical details here.
Lemma~\ref{lemLExpMomentUnderMBiais} is a result about the logistic
branching process with mutation. Once we have recognized that we are working
with a decorated Galton--Watson tree, this lemma is the key
specificity of our model for the convergence to the CRT.
Lemma~\ref{lemUniformBoundSumF} is a streamlined, model-agnostic synthesis of the
approach developed in \cite[Section~4]{Stu22}. It provides generic
concentration inequalities for sums of random variables associated to the
vertices/edges of a size-conditioned Galton--Watson tree.

Before stating our first lemma, 
recall the notation for the quantities associated to a generic color:
\begin{itemize}
  \item $X=(X_t)_{t\geq 0}$ denotes the trajectory of the number of individuals
    of that color, starting from a single individual at time~$t = 0$.
  \item $T=\inf\Set{t\geq 0\suchthat X_t=0}$ denotes the time of extinction of the color.
  \item $L=\int_0^\infty X(t)\, dt$ denotes the total length of the corresponding subnetwork.
  \item $M$ denotes the number of offspring of the color, i.e.\ the number of
    new colors that it produces by mutation.
\end{itemize}
Finally, recall that $\zeta > 0$ is the unique real number -- whose existence
is guaranteed by Corollary~\ref{corExpTiltM} -- such that
$\Expec*{\normalsize}{M\zeta^{M}}=\Expec*{\normalsize}{\zeta^{M}}$.

\begin{lemma} \label{lemLExpMomentUnderMBiais}
There exists $\eta>0$ such that
\[
  \Expec{M (\zeta\vee 1)^M e^{\eta L}} \;<\; +\infty \,.
\]
\end{lemma}

\begin{remark}
Since $T \leq L$, we also have $\Expec*{\normalsize}{M (\zeta\vee 1)^M e^{\eta T}} < +\infty$.
\end{remark}

\begin{proof}
Let us fix $s > \zeta\vee 1$ such that $\Expec*{normalsize}{s^M} < \infty$ --
such an $s$ exists by Theorem~\ref{thmM}.
Note that to prove the lemma it is sufficient to show that there exists
$\eta>0$ such that $\Expec*{\normalsize}{s^M e^{\eta L}}<\infty$; and that since
$0 \in A_s \defas \Set*{\eta\in \R \suchthat \Expec*{\normalsize}{s^M e^{\eta L}} < \infty}$,
it in fact suffices to show that $A_s$ is an open subset of $\R$.

Recall from Proposition~\ref{propMeasureChangeLogisticBP} that for any
numbers $\mu$ and $s$ and any nonnegative measurable function $f$,
\begin{equation} \label{eqProofExpMomentUnderBias}
  \Expec[\mu]{f(X)\,s^M} \;=\;
  \Expec[s\mu]{f(X)\,e^{(s-1)\mu L}} \,.
\end{equation}
Now, on the one hand by applying \eqref{eqProofExpMomentUnderBias} to
$f(X) = e^{\eta L}$ we get that for any $\eta$,
\[
  \Expec[\mu]{s^M e^{\eta L}} \;=\;
  \Expec[s\mu]{e^{((s-1)\mu + \eta)L}}.
\]
On the other hand, for $\eta < \mu$,  by taking $f(X) = 1$ and replacing
$(\mu, s)$ with $(\mu - \eta, \frac{s\mu}{\mu - \eta})$
in~\eqref{eqProofExpMomentUnderBias} we get
\[
  \ExpecBrackets[\mu-\eta]{\Big(\frac{s\mu}{\mu-\eta}\Big)^M}
  \;=\; \Expec[s\mu]{e^{((s-1)\mu + \eta)L}} \,.
\]
Combining these two equalities, we see that for $\eta < \mu$,
\[
  \Expec[\mu]{s^Me^{\eta L}} \;=\;\, \ExpecBrackets[\mu-\eta]{\Big(\frac{s\mu}{\mu-\eta}\Big)^M}\,.
\]
Now, from the explicit expression of the
probability generating function of $M$ given in Theorem~\ref{thmM}, we see that if
\[
  \eta \;\longmapsto\; \ExpecBrackets[\mu-\eta]{\Big(\frac{s\mu}{\mu-\eta}\Big)^M}
\]
is finite at $\eta=0$, then it is also finite in a neighborhood of $0$. This
concludes the proof.
\end{proof}

\begin{remark}
The relation  
$\Expec*[\mu]{\normalsize}{s^M} = \Expec*[s\mu]{\normalsize}{e^{(s-1)\mu L}}$
gives an explicit expression for the distribution of $M$
as a function of the family $(\mathscr{L}_\mu(L))_\mu$ of distributions of~$L$.
In particular, if conditional on~$L$ we let $Y$ be a
Poisson variable with parameter~$\mu L$, then
$\Expec*[\mu]{\normalsize}{s^M} = \Expec*[s\mu]{\normalsize}{s^Y}$.  However,
because in $\E_{s\mu}$ the distribution of $Y$ depends on $s$, this does not give a simple
construction of $M$ as a random function of $L$.
\end{remark}

We now give a simple Chernoff-type subpolynomial bound on the tail
probabilities of a sum of independent random variables with finite exponential
moments. The reasoning is classic -- see e.g.\ \cite{Stu22}, where it is used
repeatedly -- but we could not find a generic statement in the litterature;
so to streamline some of our proofs we state it as a lemma here.

\begin{lemma} \label{lemChernoffBound}
Let $Z_1, Z_2, \ldots$ be i.i.d.\ copies of a random variable $Z$ such that
$\Expec{Z} = 0$ and that there exists $\eta > 0$ for which
$\Expec*{\normalsize}{e^{\eta \Abs{Z}}} < \infty$.  Then
there exists $C > 0$ such that, for all $\epsilon > 0$ and all $n \geq 1$,
\[
  \Prob{\Abs[\Big]{\sum_{i = 1}^n Z_i} \geq n^{1/2\,+\,\epsilon}} \;\leq\; Ce^{- n^\epsilon}.
\]
\end{lemma}

\begin{proof}
We write $u_n = n^{1/2\,+\,\epsilon}$ to ease the notation.
Let us start by focusing on positive deviations. 
For any $\theta < \eta$, by taking the exponential of the sum, applying
Markov's inequality and using the independence of the $Z_i$'s, we get:
\begin{align*}
  \Prob{\sum_{i=1}^{n} Z_i \geq u_n}
  \;&=\; \Prob{\prod_{i=1}^{n} e^{\theta Z_i} \geq e^{\theta u_n}}\\
  \;&\leq\; e^{-\theta u_n}\, \Expec{\prod_{i=1}^{n} e^{\theta Z_i}}\\[0.5ex]
  \;&\leq\; e^{-\theta u_n} \big(\Expec*{\normalsize}{e^{\theta Z}}\big)^{n}\,.
\end{align*}
Now, since $\Expec{Z} = 0$, taking $K > \Expec{Z^2}/2$ we have
$\Expec*{\normalsize}{e^{\theta Z}} \leq 1 + K \theta^2$ for all $\theta$ small enough.
Thus, with $\theta_n = n^{-1/2}$ we get
\[
  \Prob{\sum_{i=1}^{n} Z_i \geq u_n} \;\leq\;
  \exp\big(K\,\theta_n^2\, n - \theta_n u_n\big)\;=\; \exp\big(K - n^{\epsilon}\big).
\]
The negative deviations are treated similarly (or, more directly, by applying this bound to the
variables $-Z_1, -Z_2, \ldots$), yielding 
\[
  \Prob{\sum_{i=1}^{n} Z_i \leq -u_n} \;\leq\;
  \exp\big(K - n^{\epsilon}\big).
\]
Therefore, 
$\Prob{\Abs{\sum_{i = 1}^n Z_i} \geq n^{1/2\,+\,\epsilon}} \leq 2 e^{K - n^\epsilon}$,
concluding the proof.
\end{proof}

Our next lemma is a general result about sums of decorations in
critically tiltable size-conditioned Galton--Watson trees. In what follows, by
\emph{vertex decorations} of a tree $\mathscr{T}$ we mean a family
$(F_k)_{k \geq 1}$ of real-valued random variables such that,
letting $k = 1, 2, \ldots$ be the vertices of $\mathscr{T}$, in arbitrary
order, and denoting by $M_1, M_2, \ldots$ their outdegrees, there exists i.i.d.\
random variables $\Theta_1, \Theta_2, \ldots$ that are independent
of~$\mathscr{T}$ and such that $F_k = F(M_k, \Theta_k)$ for some deterministic
function~$F$.

\newtext{Similarly, letting $E(\mathscr{T})$ denote the edge set of
$\mathscr{T}$ and writing $e=k\to\ell$ for the edge from vertex $k$ to
vertex $\ell$, oriented away from the root, we say that
$(G_e)_{e \in E(\mathscr{T})}$ are \emph{edge decorations} of $\mathscr{T}$ if,
for any vertex $k$, we have $(G_{k\to\ell}, \,\ell \text{ child of }k) = G(M_k,
\Theta_k)$ for some deterministic function $G$ such that, for all $m \geq 1$,
$G(m, \Theta_k)$ is an exchangeable real vector of length $m$.  In particular,
note that for each edge~$e$, $G_e$ is a real-valued random variable; and
that for a given $k$ the family $(G_{k\to\ell}, \,\ell \text{ child of }k)$ is
not assumed to be independent.  Note also that because of exchangeability, for
any edge
$e=k\to\ell$ the law of $G_e$ depends only on $M_k$, and so with a slight abuse
of notation we write $(M_k, G_{k\to \ell})$ for a typical pair -- for
instance the pair  of variables corresponding the edge from the root to a
uniformly chosen child of the root.}

\begin{lemma} \label{lemUniformBoundSumF}
Let $\TweakSpacingL{\mathscr{T}}$ be a Galton--Watson tree whose offspring
distribution $M$ is such that there exists $\zeta > 0$ for which
$\Expec*{\normalsize}{M \zeta^M} = \Expec*{\normalsize}{\zeta^M} < \infty$,
and let $k = 1, 2, \ldots$ denote its vertices, in arbitrary order.
Let $(F_k)_{k\geq 1}$ be vertex decorations of $\TweakSpacingLR{\mathscr{T}}$ such that
there exists $\eta > 0$ for which
$\Expec*{\normalsize}{e^{\eta \Abs{F_k}} \zeta^{M_k}} < \infty$.
\newtext{Let~$A_n$ denote the event $\{\mathscr{T}\,\text{has }n\text{ vertices}\}$.
Then, along any sequence of integers $n$ such that $\Prob{A_n}>0$:}
\begin{mathlist}
\item There exists $C > 0$ such that
  \[
    \Prob*{\big}{\max \Set{F_1,\dots,F_n} \geq C\log (n) \given A_n}
    \;\tendsto{n\to\infty} \; 0 \,.
  \]
\item Letting $v_1,v_2,\dots$ denote the vertices of $\TweakSpacingL{\mathscr{T}}$,
labeled in depth-first order:\newline conditional on~$A_n$, for all $\epsilon>0$,
\[
  \Delta \;\defas\; \max_{1\leq k\leq n}\; \Abs[\Big]{\sum_{\;i =1}^k
  (F_{v_i} - \hat{m})\,} \;=\;
  o_p(n^{1/2\,+\,\epsilon}) \,, \quad\text{where }\,
  \hat{m}= \frac{\Expec*{\normalsize}{F_k\zeta^{M_k}}}{\Expec{\zeta^{M}}} \,.
\]
\newtext{More precisely, $\Prob*{\normalsize}{\Delta \geq n^{1/2+\epsilon} \given A_n} \leq Ce^{-cn^{\epsilon}}$ for some constants $C,c>0$ that may depend on $\epsilon$, and all $n \in\N$.}
\end{mathlist}

Let
\newtext{$(G_e)_{e \in E(\mathscr{T})}$}
be edge decorations of $\TweakSpacingLR{\mathscr{T}}$ such that
there exists $\eta > 0$ for which we have
$\smash{\Expec*{\normalsize}{e^{\eta |G_{k\to\ell}|} M_k \zeta^{M_k}} < \infty}$.
Then, letting $\TweakSpacingL{\Gamma(v)}$ denote the path from the root
of $\TweakSpacingLR{\mathscr{T}}$ to its vertex~$v$:
\begin{mathlist}
\setcounter{enumi}{2}
\item Conditional on~$A_n$, for all $\epsilon>0$, 
  \[
    \Delta^* \,\defas\;
    \max_{v \in \mathscr{T}}\;
    \Abs[\Big]{\sum_{e \in \Gamma(v)}
    (G_e - m^*)} \;=\;
    o_p(n^{1/4\,+\,\epsilon}) \,, \quad\text{where }
    m^* = \frac{\Expec*{\normalsize}{G_{k\to\ell} M_k\zeta^{M_k}}}{\Expec{M\zeta^{M}}}.
  \]
\newtext{More precisely, $\Prob*{\normalsize}{\Delta^* \geq n^{1/4+\epsilon} \given A_n} \leq Ce^{-cn^{a}}$ for some constants $C,c,a>0$ that may depend on $\epsilon$, and all $n \in\N$.}
\end{mathlist}
\end{lemma}

\begin{proof}
First, note that the main difficulty comes from the fact that, 
under $\Prob{\,\cdot\given A_n}$, the decorations are not independent.

Let $\smash{\hat{\mathscr{T}}}$ be a Galton--Watson tree whose offspring
distribution $\hat{M}$ is the $\zeta$-tilt of~$M$, and recall from
Section~\ref{secConditioningColors} that $\smash{\hat{\mathscr{T}}}$ is
critical and satisfies
$(\mathscr{T} \,|\, A_n) \sim (\smash{\hat{\mathscr{T}}} \,|\, \hat{A}_n)$,
where $\hat{A}_n = \Set*{\hat{\mathscr{T}} \text{ has } n \text{ vertices}}$.
Choose $\smash{\hat{\mathscr{T}}}$ to be independent of
$\Theta_1, \Theta_2, \ldots$, and let $\smash{\hat{M}_1, \hat{M}_2, \ldots}$
be the number of children of its vertices. Finally, let
$\hat{F}_k = F(\hat{M_k}, \Theta_k)$. Then,
\begin{align} \label{eq1ProofUniformBoundSumF}
  \Prob*{\big}{\max \Set{F_1,\dots,F_n} \geq C\log (n) \given A_n}
  &=\; \Prob*{\big}{\max \Set*{\hat{F}_1,\dots,\hat{F}_n} \geq C\log (n) \given \hat{A}_n} \notag \\
  &\leq\; \sum_{k=1}^n \Prob{\hat{F}_k \,\geq\, C \log(n) \given \hat{A}_n} \notag \\
  &\leq\; \frac{n\, \Prob*{\normalsize}{\hat{F} \geq C \log(n)}}{\Prob*{\normalsize}{\hat{A}_n}} \,,
\end{align}
where we write $\hat{F}$ -- instead of, say,  $\hat{F}_1$ -- for brevity.
By assumption, there exists $\eta > 0$ such that
\[
  \Expec{e^{\eta \Abs[\small]{\hat{F}}}\,} \;=\;
  \frac{\Expec{e^{\eta \Abs{F}} \zeta^M}}{\Expec{\zeta^M}} \;<\; \infty\,.
\]
Therefore, for $\smash{C > \frac{5}{2} \eta^{-1}}$ Markov's inequality yields
$\Prob*{\normalsize}{\hat{F} \geq C \log(n)} = o(n^{-5/2})$.
Moreover, by Proposition~\ref{propTotalProgenyGW} we know that
$\Prob*{\normalsize}{\hat{A}_n} = \Theta(n^{-3/2})$.
Plugging these two estimates in~\eqref{eq1ProofUniformBoundSumF} proves point (i).

Point~(ii) is proved similarly: we fix $\epsilon>0$ and use a union bound to get
\begin{align*}
\Prob{\Delta \geq n^{1/2\,+\,\epsilon}\given A_n}
  \;&=\;\Prob{\hat{\Delta} \geq n^{1/2\,+\,\epsilon}\given \hat{A}_n} \\
    &\leq\; \frac{n\, \Prob{\Abs{\sum_{i=1}^{n} (\hat{F}_{v_i} - \hat{m})}
    \geq n^{1/2\,+\,\epsilon}}}{\Prob*{\normalsize}{\hat{A}_n}}\,.
\end{align*}
Under the unconditional probability $\P$, the random variables $\hat{F}_{v_i}$
are i.i.d.\ and their expected value is $\Expec*{\normalsize}{\hat{F}} = \hat{m}$.
Therefore, by applying Lemma~\ref{lemChernoffBound} we get
\[
  \Prob{\Abs[\Big]{\sum_{i=1}^{n} (\hat{F}_{v_i} - \hat{m})}\geq n^{1/2\,+\,\epsilon}}
  \;\leq\; Ce^{-n^{\epsilon}}
\]
for some $C>0$. Since $n / \Prob*{\normalsize}{\hat{A}_n} = \Theta(n^{5/2})$,
\newtext{this implies $\Prob*{\normalsize}{\Delta \geq n^{1/2\,+\,\epsilon}\given A_n} = O(e^{-cn^{\epsilon}})$ for some $c>0$.}

The proof of point~(iii) requires a few extra ingredients.
As for (i) and (ii), we start from $\Prob{\Delta^* \geq u_n \given A_n} =
\smash{\Prob*{\normalsize}{\hat{\Delta}^* \geq u_n \given \hat{A}_n}}$.
Next, we recall that the maximum of the distance between a vertex and
the root in a Galton--Watson tree $\mathscr{T}_n$ conditioned to have $n$
vertices is of order~$n^{1/2}$. More specifically, if we denote by
$H(\mathbf{t})$ the maximal distance to the root in a tree $\mathbf{t}$, then
$n^{-1/2} H(\mathscr{T}_n)$ converges in distribution as $n \to \infty$,
see e.g.~\cite{ADJ13}. Therefore,
for every $\epsilon > 0$ there exists $c > 0$ such that
$\smash{\Prob*{\normalsize}{H(\hat{\mathscr{T}}) > c\sqrt{n} \given \hat{A}_n} < \epsilon}$
for all $n$ large enough -- which in turns entails
\[
  \Prob*{\normalsize}{\;\cdot\, \given \hat{A}_n} \;\leq\;
  \Prob{\;\cdot\, \cap \Set*{H(\hat{\mathscr{T}}) \leq c\sqrt{n}} \given \hat{A}_n}
  \;+\; \epsilon \,.
\]
As a result, we can pick an integer sequence $(w_n)$ such that $\sqrt{n} = o(w_n)$ and
assume in what follows that, conditional on $\hat{A}_n$, we have $\smash{H(\hat{\mathscr{T}}}) \leq w_n$.

Let us denote by $\hat{\mathscr{T}}_{|h}$ the set of vertices at distance
$h$ from the root in $\hat{\mathscr{T}}$ and, to keep notation light, set
$S(v) \defas \Abs[\normalsize]{\sum_{(k\to\ell) \in \Gamma(v)} (G_{k\to\ell} - m^*)}$.
For any sequence $(u_n)$,
\begin{align} \label{eqIneqBeforeSpine}
  \Prob{\hat{\Delta}^* \geq u_n \given \hat{A}_n}
  &=\; \Prob*{\Big}{\max_{1 \leq h \leq w_n} \Set*[\Big]{\max_{v \in \hat{\mathscr{T}}_{|h}}
      \Set{S(v)}} \geq u_n \given \hat{A}_n} \notag \\ 
  &\leq\; \Expec*{\Big}{\,\sum_{h = 1}^{w_n}
      \sum_{\hspace{-0.9em}\mathrlap{v \in \vphantom{T}\smash{\hat{\mathscr{T}}_{|h}}}}
      \Indic{S(v) \geq u_n} \given \hat{A}_n} \notag\\
  &\leq\; w_n \max_{1 \leq h \leq w_n}
      \Expec*{\Big}{\sum_{\hspace{-0.9em}\mathrlap{v \in \vphantom{T}\smash{\hat{\mathscr{T}}_{|h}}}}
      \Indic{S(v) \geq u_n} \given \hat{A}_n} \notag \\
  &\leq\; w_n \,  \Prob*{\normalsize}{\hat{A}_n}^{-1}\! \max_{1 \leq h \leq w_n} 
      \Expec*{\Big}{\sum_{\hspace{-0.9em}\mathrlap{v \in \vphantom{T}\smash{\hat{\mathscr{T}}_{|h}}}}
      \Indic{S(v) \geq u_n}} \,.
\end{align}
Now, let $(\hat{\mathscr{T}}^{(h)}, v^*)$ be the random pointed tree obtained
in the following way:
\begin{itemize}
  \item Let $v_1$ be the root, and start with a path
    $v_1, v_2,\dots,v_{h+1}=v^*$ from $v_1$ to $v^*$. This path will be
    referred to as the \emph{spine} of the tree.
  \item For $k = 1$ to $h$, add $M^*_k-1$ children to $v_k$,
    where $M_k^*$ is an independent copy of the size-biasing of $\hat{M}$, i.e.\
    a random variable whose distribution is
    $\Prob{M_k^* = i} = i\,\zeta^i\,\Prob{M = i} / \Expec*{\normalsize}{M \zeta^M}$.
  \item Let each of the vertices added at the previous step,
    as well as $v^*$, be the root of an independent Galton--Watson tree with
    offspring distribution $\hat{M}$.
\end{itemize}
It is classic (see e.g.\ \cite[Section 4.2]{Stu22}) and readily checked that
for any fixed tree~$\mathbf{t}$ and each vertex $v \in \mathbf{t}_{|h}$, 
\[
  \ProbBrackets{(\hat{\mathscr{T}}^{(h)}, v^*) = (\mathbf{t}, v)} \;=\;
  \Prob*{\normalsize}{\hat{\mathscr{T}} = \mathbf{t}} \,.
\]
As a result, for any function $f$ on pointed trees, 
\[
  \Expec{f(\hat{\mathscr{T}}^{(h)},v^*)} \;=\;
  \sum_{\mathbf{t}} \sum_{\mathclap{\;\;v \in \mathbf{t}_{|h}}}
  f(\mathbf{t}, v) \,
  \ProbBrackets{(\hat{\mathscr{T}}^{(h)}, v^*) = (\mathbf{t}, v)} \;=\;
  \Expec*{\Big}{
    \sum_{\hspace{-0.9em}\mathrlap{v \in \vphantom{T}\smash{\hat{\mathscr{T}}_{|h}}}}
    f(\hat{\mathscr{T}},v)}\,.
\]
Applying this identity to $f(\hat{\mathscr{T}}, v) = \Indic{S(v) \geq u_n}$
in \eqref{eqIneqBeforeSpine}, we get
\begin{align} \label{eqAfterSpineTree}
  \Prob{\hat{\Delta} \geq u_n \given \hat{A}_n} \;\leq\;
  w_n\, \Prob*{\normalsize}{\hat{A}_n}^{-1}\!
  \max_{1\leq h\leq w_n} \Expec*{\big}{f(\hat{\mathscr{T}}^{(h)}, v^*) } \,.
\end{align}
By construction, on the spine of $\hat{\mathscr{T}}^{(h)}$ the
number of children of the vertices is distributed as the vector
$(M^*_1, \ldots, M^*_{h}, \smash{\hat{M}_{h+1}})$, whose components
are independent. As a result,
letting $G^*_k$ be the first component of the vector
$G(M_k^*, \Theta_k)$ and
$m^* = \Expec*{\normalsize}{G_{k\to\ell} M\zeta^{M}} /
\Expec*{\normalsize}{M\zeta^{M}}$ its expected value, if $u_n \to \infty$
as $n \to \infty$, then
\begin{align*}
  \Expec*{\big}{f(\hat{\mathscr{T}}^{(h)}, v^*) }
  \;&=\; \Prob{\Abs[\Big]{\sum_{\,k = 1}^{\,h-1} (G^*_k - m^*) + \hat{G}_{h} - m^*} \geq u_n} \\
  \;&\leq\; \Prob{\Abs[\Big]{\sum_{\,k = 1}^{\,w_n-1} (G^*_k - m^*) + \hat{G}_{w_n} - m^*} \geq u_n} \text{ for $n$ large enough.}
\end{align*}
Moreover, we then also have, as $n \to \infty$,
\[
  \Prob{\Abs[\Big]{\sum_{\,k = 1}^{\,w_n-1} (G^*_k - m^*) + \hat{G}_{w_n} - m^*}
  \geq u_n} \sim\;\Prob{\Abs[\Big]{\sum_{\,k = 1}^{w_n} (G^*_k - m^*)} \geq u_n}  \,.
\]
Finally, taking $u_n = n^{1/4 \,+\, \epsilon}$ for some $\epsilon > 0$ and
$w_n = \lfloor n^{1/2 \,+\, \delta} \rfloor$ for some $\delta > 0$ such that
$(1/2 + \delta)^2 < 1/4 + \epsilon$, Lemma~\ref{lemChernoffBound} yields
\[
  \Prob{\Abs[\Big]{\sum_{\,k = 1}^{w_n} (G^*_k - m^*)} \geq n^{1/4 \,+\, \epsilon}} 
  \;\leq\; Ce^{-n^{a}}
\]
for some positive constants $a$ and $C$.
Plugging this back in~\eqref{eqAfterSpineTree} and using that
$w_n\, \Prob*{\normalsize}{\hat{A}_n}^{-1} = \Theta(n^{2 + \delta})$
concludes the proof.
\end{proof}

\begin{remark} \label{remCritical}
This proof illustrates the point of working with a critical Galton--Watson
tree: for instance, even though we also have
\[
  \Prob*{\big}{\max \Set{F_1,\dots,F_n} \geq C\log (n) \given A_n} \;\leq\;
  \frac{n\, \Prob*{\normalsize}{F \geq C \log(n)}}{\Prob*{\normalsize}{A_n}}\,, 
\]
because in the non-critical case $\Prob{A_n}$ decays exponentially, the mere
fact that $F$ has finite exponential moments would not have been sufficient
to get an adequate upper bound on the expression above.
\end{remark}

\subsection{Proof of the convergence to the CRT} \label{secProofCRT}

We are now ready to prove the main theorem of this section.
Recall that $\zeta > 0$ is the unique real
number such that
$\Expec*{\normalsize}{M\zeta^{M}}=\Expec*{\normalsize}{\zeta^{M}}$.

\begin{theorem}\label{thmCRTconvergence}
Let $(\mathscr{G}_n, r_n, d_{\mathscr{G}_n}, \lambda_{\mathscr{G}_n})$
denote the random rooted metric probability space
$(\mathscr{G}, r, d_{\mathscr{G}}, \lambda_{\mathscr{G}})$ conditioned to
have $n$ colors, and let $\mathscr{C}$ be the Brownian CRT. Then,
as $n \to \infty$,
\[
  \big(\mathscr{G}_n,\, r_n,\, \tfrac{C}{\sqrt{n}} d_{\mathscr{G}_n},\,
  \tfrac{1}{|\mathscr{G}_n|}\lambda_{\mathscr{G}_n}\big)
  \;\overset{d}{\longrightarrow}\; \mathscr{C}\\[1ex]
\]
for the rooted Gromov--Hausdorff--Prokhorov topology, with
\[
  C \,\defas\; \frac{\hat{\sigma}}{2\Expec{U^*}} \;=\;
  \frac{\sqrt{\Expec*{\normalsize}{\zeta^M}
      (\Expec*{\normalsize}{M^2\zeta^M} - \Expec{\zeta^M})}}{2\,\Expec*{\normalsize}{\sum_{t \in \mathcal{M}} t\,\zeta^M}}\,, 
\]
where $U^*$ is sampled uniformly at random among the mutation times of the biased
process $X^* \sim \mathscr{L}(X \biasedby M \zeta^M)$, and $\TweakSpacingL{\hat{\sigma}^2}$
is the variance of $\TweakSpacingL{\hat{M} \sim \mathscr{L}(M \biasedby \zeta^M)}$.
Moreover, 
\[
  |\mathscr{G}_n| \;\defas\; \lambda_{\mathscr{G}_n}(\mathscr{G}_n) \;=\;
  \frac{n\,\Expec*{\normalsize}{L \zeta^M}}{\Expec{\zeta^M}} \;+\;
  o_p(n^{1/2\,+\,\epsilon}) \, , \quad \forall \epsilon > 0.
\]
\end{theorem}

\begin{remark} \label{remExpressCViaNuCirc}
Using the probability measure $\nu_\circ$ introduced in
Proposition~\ref{propNuCircXfromMut}, the constant $C$ can also be expressed
in terms of $T$, the extinction time of the logistic branching process.
Indeed, as explained after the proof of Theorem~\ref{thmCRTconvergence},
\[
  \Expec{U^*} \,=\; 
  \frac{\zeta \, \Expec{M}}{\Expec*{\normalsize}{\zeta^M}}
  \sum_{k > 0} \nu_\circ(k) \,
  \Expec*[k]{\normalsize}{T\, \zeta^M} \,
  \Expec*[k-1]{\normalsize}{\zeta^M} \, . \qedhere
\]
\end{remark}

\begin{remark} \label{remCvMoments}
\newtext{Addario-Berry et al.~\cite{ADJ13} have shown that for critical
Galton--Watson trees with finite-variance offspring distribution, the normalized
height (and width) of the tree conditioned to have $n$ vertices satisfy
uniform sub-Gaussian tail bounds.
More precisely, letting $H(\mathscr{T}_n)$ denotes the height of the Galton--Watson
tree $\mathscr{T}_n$, there exist $K,k>0$ such that for every $n\in \N$ and
every $x\in \R_+$ we have
\begin{equation} \label{eqAddarioBerry}
 \Prob{H(\mathscr{T}_n) \geq h\sqrt{n}} \;\leq\; K e^{-k h^2}.
\end{equation}
Now consider, as in Lemma~\ref{lemUniformBoundSumF}, any family of edge decorations $(G_e)_{e\in E(\mathscr{T})}$ 
such that $\Expec*{\normalsize}{e^{\eta \Abs{G_{k\to \ell}}} M_k\zeta^{M_k}} < \infty$ for some $\eta > 0$.
Let us write
\[
  \widetilde{\Delta}_n \;=\; \max_{v\in \mathscr{T}} 
  \frac{1}{\sqrt{n}}\,\Abs[\Big]{\!\!\sum_{\;\;e\in\Gamma(v)} \!\!\! G_e}.
\]
With the notation of Lemma~\ref{lemUniformBoundSumF}\,(iii), we have $\widetilde{\Delta}_n \leq (\Delta^*+H(\mathscr{T})m^*)/\sqrt{n}$.
Thus, applying \eqref{eqAddarioBerry} and Lemma~\ref{lemUniformBoundSumF}\,(iii) with
$\epsilon=1/4$ yields
\begin{align*}
  \Prob{\widetilde{\Delta}_n \geq x \mid A_n}
  \;&\leq\; \Prob{\Delta^* \geq x\sqrt{n}/2 \mid A_n} \;+\;
            \Prob{H(\mathscr{T}_n) \geq x\sqrt{n}/2}\\
  &\leq\; K\big(e^{-kx^{a}} + e^{-kx^2}\big)
\end{align*}
for some positive constants $K$, $k$ and $a$ that do not depend on $n$.  This
uniform tail bound implies that for any $p\geq 1$, the sequence
$(\widetilde{\Delta}_n^p)_{n\geq 1}$ is tight.  In our context, if we let
$G_{k\to\ell}$ be the distance between the root of
$\mathscr{X}_k$ and the mutation corresponding to the vertex $\ell$, then
$\widetilde{\Delta}_n$ is the height of $(\mathscr{G}_n, r_n,
d_{\mathscr{G}_n}/\sqrt{n})$.  Therefore, we conclude that in addition to the
convergence in distribution of Theorem~\ref{thmCRTconvergence}, all moments of the
height of the network, diameter and related quantities converge to the
moments of the properly rescaled Brownian CRT.}
\end{remark}

\begin{proof}[Proof of Theorem~\ref{thmCRTconvergence}]
The proof is an application of Proposition~\ref{propKeyToCRT}. Set
\[
  d_n \;=\; C\, n^{-1/2}\, d_{\mathscr{G}_n} , 
  \qquad
  \lambda_n \;=\; |\mathscr{G}_n|^{-1} \lambda_{\mathscr{G}_n},
\]
and let us define an admissible parametrization
$\phi_n\colon[0, 1]\to (\mathscr{G}_n, r_n, d_n, \lambda_n)$.
Recall that, in the forward-time process defining $\mathscr{G}$,
a \emph{branching point} is a point where a lineage splits and
a \emph{coalescence point} is a point where two lineages merge.
Using some arbitrary procedure, distinguish one of the two outgoing
lineages of each branching point of $\mathscr{G}_n$ and one of the
two incoming lineages of each coalescence point. Note that by (1) disconnecting
the tip of each of the distinguished lineages that correspond to coalescences
points from those coalescence points and (2) drawing distinguished lineages
that correspond to branching points to the right of their undistinguished
counterparts, we get a rooted plane $\R$-tree~$\mathscr{G}_n^{\#}$ (not to be
confused with $\mathscr{T}_n$, the combinatorial tree encoding the genealogy
of the colors of~$\mathscr{G}_n$).

\newpage \enlargethispage{2ex}

Now, pick a depth-first ordering of the vertices of $\mathscr{T}_n$, and
visit the points of $\mathscr{G}_n$ as follows:
\begin{itemize}
  \item Visit the subnetworks corresponding to the
    vertices of $\mathscr{T}_n$ in depth-first order.
  \item Within each subnetwork $\mathscr{X}_k$, do a depth-first
    traversal of the corresponding ``unreticulated'' $\R$-tree
    $\mathscr{X}_k^\#$, that is: starting from the root, travel along the
    lineages at constant speed $n L_k = n \lambda_{\mathscr{G}_n}(\mathscr{X}_k)$,
    in depth-first order and visiting the ``left'' subtree first when
    encountering a branching point.
\end{itemize}
This construction is illustrated in Figure~\ref{figPhi}. Note that each jump of
$\phi_n$ corresponds to either the tip of a lineage or the second visit
of a coalescence point, and that those jumps can be negative (typical case) or
positive (which can only happen when finishing the exploration of a
color and moving to a new one). Moreover, each point of $\mathscr{G}_n$ is
visited exactly once, except for:
\begin{itemize}
  \item The tips of lineages, which -- with the exception of $\phi_n(1)$ --
    correspond to the left-limits of some of the jumps of $\phi_n$.
  \item Branching points, which are visited twice.
\end{itemize}
As a result, $\phi_n([0, 1])$ is dense in $\mathscr{G}_n$ and $\phi_n$ is
an admissible parametrization of~$\mathscr{G}_n$.

\begin{figure}[h!]
  \centering
  \includegraphics[width=0.7\linewidth]{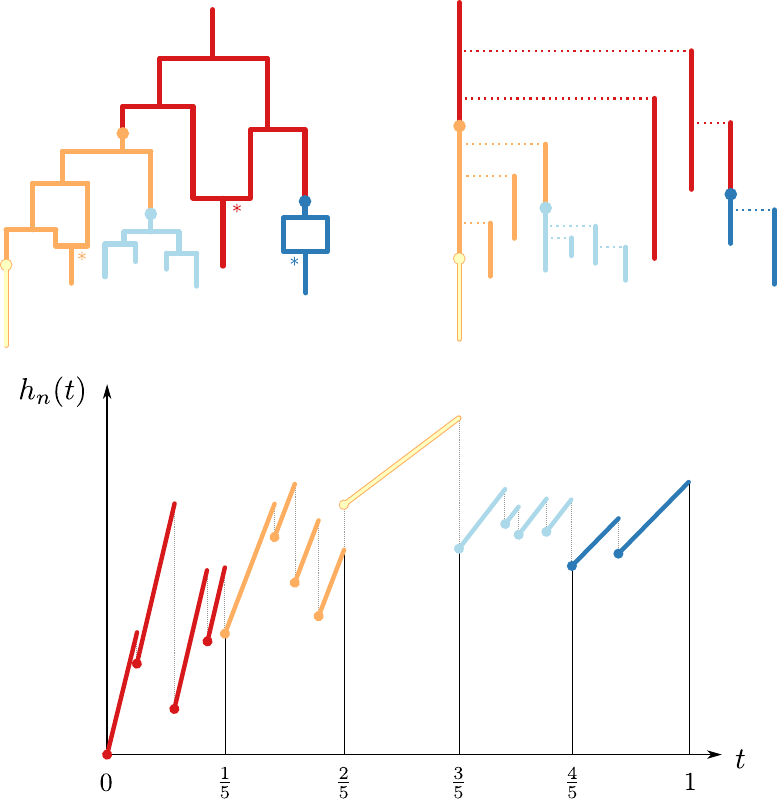}
  \caption{Illustration of the construction of the admissible parametrization 
  $\phi_n$ used in the proof. Left, a realization of $\mathscr{G}_n$
  for $n = 5$, with the same drawing conventions as in
  Figure~\ref{figDefProcess}.  The distinguished lineages associated to
  coalescences are indicated by asterisks and, to avoid cluttering, the
  distinguished lineages associated to branchings are taken to be the lineages
  drawn to the right. Right, the rooted plane $\R$-tree $\mathscr{G}_n^\#$
  obtained by ``disconnecting'' coalescence points of
  $\mathscr{G}_n$. This tree is to provide us
  with a natural order in which to visit the lineages of $\mathscr{G}_n$.
  Bottom, the height function $h_n\colon t \mapsto d_n(r_n, \phi_n(t))$
  associated to $\phi_n$.  The speed of travel along the lineages of the
  subnetwork corresponding to a given color is proportional to the total length
  of that subnetwork, ensuring that each color is allotted the same amount of
  time by $\phi_n$.}
  \label{figPhi}
\end{figure}

Next, let us show that $\phi_n$ satisfies assumptions~(i--iii) of
Proposition~\ref{propKeyToCRT}. Starting with (i), pick $s, t \in [0, 1]$ with
$s < t$, and let $(x, y) = (\phi_n(s), \phi_n(t))$ be the corresponding
points of $\mathscr{G}_n$. Let then $x\wedge y$ be the most recent common ancestor
of $x$ and $y$ in $\mathscr{G}_n$, i.e.\ the (unique) oldest point in a (non
necessarily unique) shortest path between $x$ and $y$, and let $\mathscr{X}_c$
be the subnetwork containing $x \wedge y$. Let $z_c$ be the root of
$\mathscr{X}_c$ and, for $i \in \{x, y\}$, let $z_i$ be: $z_c$ if
$i \in \mathscr{X}_c$; otherwise, the root of the subnetwork through which
every path from $x\wedge y$ to~$i$ exits $\mathscr{X}_c$.
These definitions are illustrated in Figure~\ref{figProofCRT}.

\begin{figure}[h!]
  \centering
  \includegraphics[width=0.7\linewidth]{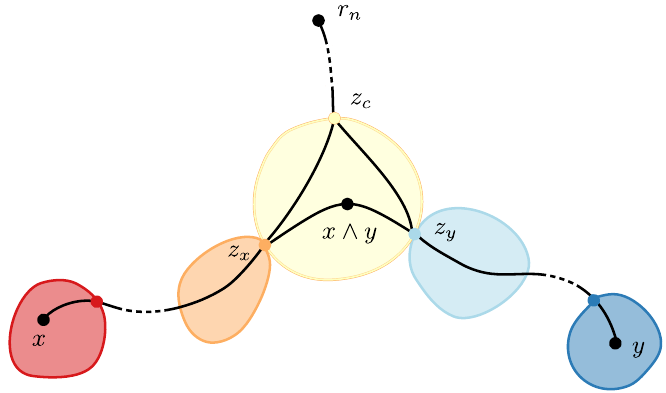}
  \caption{Graphical depiction of some of the notation used in the proof. The
  black lines correspond to shortest paths between various points of
  $\mathscr{G}_n$, and the colored blobs to the subnetworks associated to
  the colors. Although there can be several shortest paths between $x$ and $y$,
  each of these paths goes through $z_x$ and $z_y$, hence
  $d_n(x, y) = d_n(x, z_x) + d_n(z_x, z_y) + d_n(z_y, y)$. Similarly,
  $d_n(r_n, x) = d_n(r_n, z_x) + d_n(z_x, x)$. Finally, note that within
  each subnetwork $\mathscr{X}_k$ the distances are bounded above by $2\,T_k$,
  where $T_k$ is the extinction time of the corresponding process $X_k$.}
  \label{figProofCRT}
\end{figure}

With this notation, and recalling that $h_n(t) = d_n(r_n, \phi_n(t))$, observe that
\begin{align} \label{ineqProofThmCRT01}
  & \Abs[\big]{d_n(x, y) - h_n(s) - h_n(t) + 2 \inf_{[s, t]} h_n} \notag \\
  =\;\;& \Abs[\big]{d_n(x, z_x) + d_n(z_x, z_y) + d_n(z_y, y)
                     - d_n(r_n, x) - d_n(r_n, y) + 2 \inf_{[s, t]} h_n} \notag \\
  \leq\;\;& d_n(z_x, z_y)  
          \;+\;\Abs[\big]{d_n(r_n, z_x) -  \inf_{[s, t]} h_n} 
          \;+\; \Abs[\big]{d_n(r_n, z_y) -  \inf_{[s, t]} h_n}  \,.
\end{align}
Now, $d_{\mathscr{G}_n}(z_x, z_y) < 2\,T_c$, where $T_c$ is the extinction time
of the logistic process $X_c$ associated to the color $c$. Therefore, 
\begin{equation} \label{ineqProofThmCRT02}
  d_n(z_x, z_y) \;<\; 2\,C n^{-1/2}\,T_c \,.
\end{equation}
Moreover, since by construction
of $\phi_n$ the vertices of $\mathscr{T}_n$ are visited in depth-first order,
for all $u \in [s, t]$ the point $\phi_n(u)$ belongs to either
$\mathscr{X}_c$ or one of its descendants, which implies
that $h_n(u) \geq d_n(r_n, z_c)$; and since there exists $u \in [s, t]$ such
that $\phi_n(u)$ is at distance 0 from $\mathscr{X}_c$ (indeed, if $y \in
\mathscr{X}_c$ one can take $u = t$, and if $y \notin \mathscr{X}_c$ then one
can take $u$ such that $\phi_n(u) = z_y$), which in turns implies
$h_n(u) \leq d_n(r_n, z_c) + Cn^{-1/2}T_c$, we get 
\[
  d_n(r_n, z_c) \;\leq\; \inf_{[s, t]} h_n \;\leq\; d_n(r_n, z_c) + Cn^{-1/2}T_c\,.
\]
Similarly, for $i \in \{x, y\}$ we have
$d_{n}(r_n, z_c)\leq d_{n}(r_n, z_i) \leq d_{n}(r_n, z_c)+C n^{-1/2}\, T_c$, 
so that
\begin{equation} \label{ineqProofThmCRT03}
  \Abs[\big]{d_n(r_n, z_i) -  \inf_{[s, t]} h_n}  \;\leq\;
   C n^{-1/2}\,T_c,\qquad i\in \Set{x,y}.
\end{equation}
Plugging \eqref{ineqProofThmCRT02} and~\eqref{ineqProofThmCRT03} in
\eqref{ineqProofThmCRT01}, we get
$\Abs{d_n(x, y) - d_{h_n}(s, t)} < 4\,C\, n^{-1/2}\,T_c$. Therefore, 
\[
  \sup_{s, t\in [0, 1]}\, \Abs[\big]{d_n(\phi_n(s), \phi_n(t)) - d_{h_n}(s, t)}
  \;\leq\; 4\, C\, n^{-1/2} \max(T_1, \ldots, T_n) \,.
\]
Applying point~(i) of Lemma~\ref{lemUniformBoundSumF} to the extinction times
$T_1, \ldots, T_n$, which is made possible by the fact that we know from
Lemma~\ref{lemLExpMomentUnderMBiais} that $T$ has finite exponential moments
under $\mathscr{L}(\,\cdot\,\biasedby M \zeta^M)$, we get
$\max(T_1, \ldots, T_n) = o_p(n^\epsilon)$ for any $\epsilon > 0$,
which in turns implies
\[
  \sup_{s, t\in [0, 1]}\, \Abs[\big]{d_n(\phi_n(s), \phi_n(t)) - d_{h_n}(s, t)}
  \;\;\overset{d}{\longrightarrow}\;\; 0,
\]
thereby proving that $\phi_n$ satisfies assumption~(i) of
Proposition~\ref{propKeyToCRT}.

Let us now turn to assumption~(ii). Let $\mathscr{X}_1, \ldots, \mathscr{X}_n$ be
the subnetworks of $\mathscr{G}_n$, in order of their visit by $\phi_n$.
By construction of $\phi_n$, for all $t \in [0,1]$ we have
$\phi_n(t) \in \mathscr{X}_{c_t}$, where $c_t \defas (\lfloor tn \rfloor + 1)\wedge n$
Moreover,
\[
  \sum_{k = 1}^{c_t - 1} L_k
  \;\leq\; \lambda_{\mathscr{G}_n}(\phi_n([0, t])) \;\leq\;
  \sum_{k = 1}^{c_t} L_k\,, 
\]
where $L_k = \lambda_{\mathscr{G}_n}(\mathscr{X}_k)$.
Applying point~(ii) of Lemma~\ref{lemUniformBoundSumF} to $L_1, \ldots, L_n$,
which again is made possible by Lemma~\ref{lemLExpMomentUnderMBiais},  we get
that for any $\epsilon > 0$, 
\[
  \lambda_{\mathscr{G}_n}(\phi_n([0, t])) \;=\; t n \ell \;+\;
  o_p(n^{1/2\,+\,\epsilon})\,, 
\]
uniformly in $t \in [0, 1]$ and with
$\ell = \Expec*{\normalsize}{L\zeta^M} / \Expec*{\normalsize}{\zeta^M}$.
Taking $t = 1$, we see that $|\mathscr{G}_n| = n \ell + o_p(n^{1/2\,+\,\epsilon})$,
as claimed in the statement of the theorem. From there, we get
\[
  \sup_{t \in [0, 1]} \Abs[\big]{\lambda_n(\phi_n([0, t])) - t\,} \;=\; 
  o_p(n^{-1/2\,+\,\epsilon})\,, 
\]
which shows that $\phi_n$ satisfies assumption~(ii).

To show that $\phi_n$ satisfies assumption~(iii), let
$h_{\mathscr{T}_n}$ be the
height process of $\mathscr{T}_n$, that is
\[
  h_{\mathscr{T}_n}\colon\begin{cases}
    \displaystyle \Set{1, \ldots, n} &\longrightarrow\; \N \\
    \displaystyle \hfill k &\longmapsto \; d_{\mathscr{T}_n}(v_1, v_{k}) \,, 
  \end{cases}
\]
where $v_1, \ldots, v_n$ are the vertices of $\mathscr{T}_n$, in order of
their visit by $\phi_n$ and $d_{\mathscr{T}_n}(u, v)$ is the number of edges of
the path joining $u$ and $v$.  Since $\mathscr{T}_n \sim
(\smash{\hat{\mathscr{T}}\mid \hat{A}_n})$ is a critical Galton--Watson tree
conditioned to have $n$ vertices, it is well-known -- see e.g.\ Corollary~1 in
\cite{MaM03}, from which this readily follows -- that, as $n\to \infty$,
\[
  \mleft(\tfrac{1}{\sqrt{n}} h_{\mathscr{T}_n}(c_t)\mright)_{t \in [0, 1]}
  \;\;\overset{d}{\longrightarrow}\;\;
  \big(\tfrac{2}{\hat{\sigma}} e(t)\big)_{t \in [0, 1]}
  \quad\text{in the Skorokhod space } D,
\]
where $(e(t))_{t\in[0, 1]}$ is a standard Brownian excursion and
$\hat{\sigma}^2 = \Var[\normalsize]{\hat{M}}$ is the variance of the offspring
distribution of $\smash{\hat{\mathscr{T}}}$, and $c_t =(\lfloor tn \rfloor + 1)\wedge n$. Therefore, to conclude the proof
of Theorem~\ref{thmCRTconvergence} it suffices to show that
\begin{equation} \label{eqProofThmCRT}
  \sup_{t \in [0, 1]}\, \Abs[\big]{\,h_{\mathscr{G}_n}(t) -
  \Expec{U^*}\, h_{\mathscr{T}_n}(c_t)\,}
  \;=\; o_p(n^{1/2}) \,, 
\end{equation}
where $h_{\mathscr{G}_n}(t) = d_{{\mathscr{G}_n}}(r_n, \phi_n(t))$ and
$U^*$ is the time of a mutation sampled uniformly at random among the mutations
of the process $X^* \sim \mathscr{L}(X \biasedby M \zeta^M)$.

For this, for all $k \in \Set{1,
\ldots, n}$ let us denote by $z_{k}$ the root of $\mathscr{X}_{k}$. As a
result, letting $\Gamma(k) = (i_1 \to \ldots \to i_{p})$ be such that
$(v_{i_1} = v_1, \ldots, v_{i_p} = v_k)$ is the path from $v_1$ to $v_k$
in $\mathscr{T}_n$, and recalling that
$\phi_n(t) \in \mathscr{X}_{c_t}$, we
see that, for all $t \in [0, 1]$,
\[
  h_{\mathscr{G}_n}(t)\;=\;
  \sum_{\qquad\mathclap{(i\to j) \in \Gamma(c_t)}} d_{\mathscr{G}_n}(z_i, z_j) \;+\;
  d_{\mathscr{G}_n}(z_{c_t}, \phi_n(t)) \,.
\]
As we have already seen, for any $\epsilon > 0$,
$d_{\mathscr{G}_n}(z_{c_t}, \phi_n(t)) < T_{c_t} = o_p(n^\epsilon)$, uniformly
in~$t$. Since, by definition of $h_{\mathscr{T}_n}$,
the number of edges of $\Gamma(c_t)$ is $h_{\mathscr{T}_n}(c_t)$,
we get that for any constant $\kappa$,
\[
  h_{\mathscr{G}_n}(t) \;-\; \kappa\, h_{\mathscr{T}_n}(c_t)
  \;\;= \sum_{\qquad\mathclap{(i\to j) \in \Gamma(c_t)}}
  (d_{\mathscr{G}_n}(z_i, z_j) - \kappa)  \;+\; o_p(n^\epsilon) \,.
\]
Moreover, along $\Gamma(c_t)$ each $d_{\mathscr{G}_n}(z_i, z_j)$
is the time elapsed between the creation of~$\mathscr{X}_i$ and
that of $\mathscr{X}_j$ -- which, conditional on
$\mathscr{X}_i$, is distributed as the random functional
$U_i = U(\mathscr{X}_i)$ giving the time of a mutation sampled uniformly at
random among the mutations of $\mathscr{X}_i$.
As a result, applying point~(iii) of
Lemma~\ref{lemUniformBoundSumF}, we get
\[
  \sup_{t \in [0, 1]}\, \Abs[\Big]{\sum_{\qquad\mathclap{(i\to j) \in \Gamma(c_t)}}
  \!\!\!\big(d_{\mathscr{G}_n}(z_i, z_j) - \Expec{U^*}\big)}
  \;=\; o_p(n^{1/2}) \,, 
\]
where $\Expec{U^*} = \Expec*{\normalsize}{U M \zeta^M}/{\Expec*{\normalsize}{\zeta^M}}$
does indeed correspond to the expected value of the time of a mutation sampled
uniformly at random among the mutations of the biased process $X^* \sim
\mathscr{L}(X \biasedby M \zeta^M)$. Putting the pieces together, this
proves~\eqref{eqProofThmCRT}, thereby concluding the proof of
Theorem~\ref{thmCRTconvergence}.
\end{proof}

Finally, before closing this section, let us justify the expression of
$\Expec{U^*}$ given in Remark~\ref{remExpressCViaNuCirc}. To make things
simpler, we work with the ``extended'' process~$\bar{X}$, which, in addition
to the trajectory of $X$, contains the information of which jump corresponds to
a mutation.  Thus, $M = M(\bar{X})$ is a deterministic function of~$\bar{X}$.
First, note that
\[
  \Expec{U^*} \;=\; \Expec*{\normalsize}{\zeta^M}^{-1}\,
  \Expec*{\big}{\sum_{\;t \in \mathcal{M}}\!\! t\, \zeta^M\,} \,, 
\]
and that, considering the shift operators $(\Theta_t)_{t\in \R}$ defined by
$\Theta_t\bar{X}\defas(\bar{X}_{t + s})_{-t \leq s \leq T - t}$
and the function
\[
  F\colon\bar{X} \;\longmapsto\; \Expec*{\normalsize}{\zeta^M}^{-1}\,\zeta^{M(\bar{X})}\,
  \sup\Set{t \suchthat \bar{X}_{-t} > 0}\,, 
\]
we have
\[
  \Expec*{\normalsize}{\zeta^M}^{-1}\, \Expec*{\big}{\sum_{\;t \in \mathcal{M}}\!\! t\, \zeta^M\,}
  \;=\;
  \Expec*{\big}{\sum_{\;t \in \mathcal{M}}\!\!F(\Theta_t \bar{X})} \,.
\]
Recalling the definition of the process $\Xmut$ introduced in
Section~\ref{secLogistic}, this is also
\[
  \Expec*{\big}{\sum_{\;t \in \mathcal{M}}\!\!F(\Theta_t \bar{X})} \;=\;
  \Expec{M}\,\Expec*{\big}{F(\bar{X}^\mathfrak{m})} \,.
\]
Therefore, using the $\bar{X}^{\mathfrak{m}} \overset{d}{=} \bar{X}' \wr \bar{X}''$
decomposition given in Proposition~\ref{propNuCircXfromMut}, together with the
fact that
\[
  F(\bar{X}' \wr \bar{X}'') \;=\;
  \Expec*{\normalsize}{\zeta^M}^{-1}
  \zeta^{M(\bar{X}') + M(\bar{X}'') + 1} \, T(\bar{X}')\,, 
\]
where $T(\bar{X}')$ denotes the extinction time of $\bar{X}'$, we get
\[
  \Expec*{\big}{F(\bar{X}^\mathfrak{m})} \;=\;
  \zeta \, 
  \Expec*{\normalsize}{\zeta^M}^{-1}
  \sum_{k > 0} \nucirc(k) \, \Expec*[k]{\normalsize}{T \zeta^M} \, 
  \Expec*[k-1]{\normalsize}{\zeta^M} \,.
\]
Putting the pieces together, this yields the expression given in
Remark~\ref{remExpressCViaNuCirc}.

\section{Local weak limit} \label{secLocalLimit}

In this section, we describe the structure of $\mathscr{G}_n$ around a
uniformly chosen point.  More specifically, we give an algorithmic
construction of the local weak limit of~$\mathscr{G}_n$ around a focal point
distributed according to the normalized measure
$\lambda_{\mathscr{G}_n}/\Abs{\mathscr{G}_n}$.
The notion of local weak convergence after uniform
rooting was introduced by Benjamini and Schramm in \cite{BeS01},
and is therefore also known as \emph{Benjamini--Schramm} convergence; see
e.g.\ \cite[Section~2.2]{Stu20} or \cite[Section~1.2]{Cur18} for a general introduction.
Throughout this section, unless specified otherwise the term
\emph{local weak limit} will always refer to the to the Benjamini--Schramm limit.

This section is organized as follows: first, we briefly lay out the topological
notions that are used in our proof of the local convergence.
We then describe the local weak limit $(\mathscr{G}^\dagger, x^\dagger)$
as a decorated random tree. This random tree is a biased -- that is, non-uniformly
rooted -- local weak limit of the size-conditioned Galton--Watson tree $\mathscr{T}_n$
giving the genealogy of the colors of $\mathscr{G}_n$
(see Section~\ref{secConditioningColors}), and the decorations are modifications
of the subnetwork $\mathscr{X}$ corresponding to a generic color.
We close the section by describing the geometry of these various
modifications of $\mathscr{X}$.

\subsection{Local topology} \label{secLocalTopology}

\newcommand{\locspace}{\mathfrak{T}_{\mathrm{loc}}}

In order to define our local topology, we first need to specify a local space
of decorated trees. A locally finite pointed rooted plane tree -- henceforth
simply referred to as a \emph{pointed tree} for brevity -- is a pair
$(\mathscr{T}, v^*)$ where $\mathscr{T}$ is a rooted plane tree in which every
vertex has a finite degree, and $v^*$ is a vertex of $\mathscr{T}$ known as the
focal vertex. Note that in the case where $\mathscr{T}$ is infinite, its root
can be located at infinity: in that case, instead of corresponding to a
vertex, the root corresponds to a topological end of $\mathscr{T}$. Another
way to see this is that $\mathscr{T}$ being rooted actually means for any pair of
adjacent vertices $(u, v)$ we know who is the parent and who is the child.

A \emph{decorated pointed tree} $(\mathscr{T}, v^*, (D_v)_{v \in \mathscr{T}})$
is a pointed tree where each vertex $v \in \mathscr{T}$ is associated
to a random variable $D_v$ taking value in a
Polish space $\mathscr{D}$.  In our setting, $\mathscr{D}$ will be a space in
which the color networks $(\mathscr{X}_v)$ used in the construction of
$\mathscr{G}$ as a decorated tree are well-defined Borel-measurable random
variables; but for now let us view it simply as an abstract Polish space. We
denote by $\locspace$ the space of decorated pointed trees.

The \emph{local topology} on $\locspace$ is the topology generated by the
following basis of open sets:
\[
  U\big(r,\mathbf{t},(V_v)_{v\in \mathbf{t}} \big)
  = \Set*[\big]{(\mathscr{T}, v^*, (D_v)_{v \in \mathscr{T}})\in \locspace \suchthat B_{\mathscr{T}}(v^*\!,\, r) = \mathbf{t}
  \text{ and } \forall v\in \mathbf{t}, \;D_{v} \in V_v},
\]
where $r$ runs over the positive integers, $\mathbf{t}$ over the finite
pointed trees, and $(V_u)_{u\in \mathbf{t}}$ over the opens
sets of $\mathscr{D}$. The notation $B_{\mathscr{T}}(v^*, r)$ stands for
the ball of radius $r$ centered at $v^*$ in $\mathscr{T}$.

To make our description of $\locspace$ fully explicit, we would need to
give a formal definition of the space $\mathscr{D}$ of decorations.
While this is relatively straightforward to do, this is not only tedious
but also uninformative. Therefore, we leave it to the reader to convince
themself that this can be done while ensuring that the following properties
hold:
\begin{itemize}
  \item The decorations are pointed networks, i.e.\ pairs
    $(\mathscr{X}, x^*)$ where $x^*\in \mathscr{X}$ and, as previously, 
    the network~$\mathscr{X}$ -- which is meant to represent the subnetwork
    of $\mathscr{G}$ that corresponds to a given color -- can be seen as a
    collection of segments glued together at their endpoints
    (see Section~\ref{secResults}). We denote by $\lambda_{\mathscr{X}}$ the
    Lebesgue measure on $\mathscr{X}$.
    
    To keep the notation light, the fact that the decorations are pointed 
    will be considered implicit: we occasionally write
    $\mathscr{X}$ instead of $(\mathscr{X},x^*)$ when the focal point $x^*$
    is irrelevant.
  \item The map $\mathscr{X} \mapsto L_{\mathscr{X}} \defas
    \int d\lambda_{\mathscr{X}}$ is continuous.
  \item For all continuous bounded maps $F\colon\mathscr{D}\to \R$, the map
  \[
  \mathscr{X} \;\longmapsto\;
  \int F(\mathscr{X},x) \,\lambda_{\mathscr{X}}(dx)
  \]
  is continuous.
\end{itemize}

\subsection{Construction of the limit as a decorated tree} \label{secConstructionLocLim}

First, recall from Section~\ref{secConditioningColors} that if
$\smash{\hat{\mathscr{T}}}$ is Galton--Watson tree whose offspring distribution
$\hat{M}$ is given by
\[
  \Prob*{\normalsize}{\hat{M}=k} \;=\;
  \frac{\zeta^k\, \Prob{M = k}}{\Expec{\zeta^M}},  \qquad k\geq 0,
\]
where $\zeta$ is as in Corollary~\ref{corExpTiltM}, 
then $\smash{\hat{\mathscr{T}}}$ conditioned to have $n$ vertices
has the same distribution as the tree
$\mathscr{T}_n$ used to construct $\mathscr{G}_n$ as a decorated tree.

Next, let us describe $(\mathscr{T}^*, v^*)$, the local weak limit of
$\mathscr{T}_n$. The local weak limit of size-conditioned critical
Galton--Watson trees after random rooting is the invariant random sin-tree
introduced by Aldous in~\cite{Ald91Fringe} -- see \cite{Stu19} for a detailed
presentation.  With our notation, this pointed tree
$(\mathscr{T}^*, v^*)$ can be constructed as follows:
\begin{itemize}
  \item Let $v^*$ be the focal vertex and let
    $(v^*, v_1, v_2,\dots)$ be the \emph{spine} of $\mathscr{T}^*$, i.e.\
    an infinite path going towards the root (thus, $v_1$ is the parent of $v^*$,
    $v_2$ is the parent of $v_1$, etc).
  \item For each $k\geq 1$, add $M^*_k-1$ children to $v_k$,
    where $(M^*_k)_{k\geq 1}$ is an i.i.d.\ sequence with the
    size-biased distribution of $\hat{M}$.
  \item Let $v^*$, as well as each of the vertices added at the previous step
    be the root of a Galton--Watson tree with offspring distribution $\hat{M}$,
    and call $\mathscr{T}^*$ the resulting infinite random tree.
\end{itemize}

Now, let us consider the infinite random network $\mathscr{G}^*$
obtained by decorating $\mathscr{T}^*$ using the same procedure as when
decorating $\mathscr{T}_n$ to obtain~$\mathscr{G}_n$:
conditional on~$\mathscr{T}^*$, let us decorate each vertex $v$,
independently of everything else, with a random network~$\mathscr{X}_v$
having the distribution of the generic color network $\mathscr{X}$ conditioned
to have a number of mutations equal to the number of children of $v$.
Note that the subnetwork $\mathscr{X}^*$
corresponding to the focal vertex $v^*$ plays a special role in $\mathscr{G}^*$:
we refer to that subnetwork as the \emph{focal network}.

The pair $(\mathscr{G}^*\!,\,\mathscr{X}^*)$ is not the local weak limit
of $\mathscr{G}_n$: indeed, it describes the limit of neighborhoods
of a \emph{color network} picked uniformly at random in $\mathscr{G}_n$, 
rather than the limit of the neighborhoods of a \emph{point} picked uniformly
at random in~$\mathscr{G}_n$.  To see why the two differ, note in particular that
picking the focal point $x^*$ according to the normalized length measure
$\lambda_{\mathscr{G}_n} / |\mathscr{G}_n|$ biases the focal network by its
total length $L^* = \lambda_{\mathscr{X}^*}(\mathscr{X}^*)$.

To construct the local weak limit of $\mathscr{G}_n$,
conditional on $(\mathscr{G}^*\!,\,\mathscr{X}^*)$ let
$x^* \sim \lambda_{\mathscr{X}^*}/L^*$ be a random point of $\mathscr{X}^*$.
Note that -- to fall back on the topological framework of
Section~\ref{secLocalTopology} -- the pointed network $(\mathscr{G}^*\!,\, x^*)$ can be
seen as an element of
$\locspace$ by identifying it with a copy of $(\mathscr{T}^*\!,\, v^*)$ where
the focal vertex $v^*$ is decorated with the pointed network
$(\mathscr{X}^*\!,\, x^*)$, and the decorations of the other vertices
are arbitrarily pointed.
Now recall from Notation~\ref{notationBiasedTrajectory} that
$\mathscr{L}(\, \cdot \biasedby L^*)$ denotes the distribution
under the $L^*$-biased probability measure $\Expec*{\normalsize}{\Indic{\cdot}L^*} / \Expec{L^*}$,
and let $(\mathscr{G}^\dagger,\, x^\dagger)$ be the random pointed network
rooted at infinity characterized by
\[
  (\mathscr{G}^\dagger,\,x^\dagger) \;\sim\;
  \mathscr{L}\big((\mathscr{G}^*\!,\, x^*) \biasedby[big] L^*\big)\,. 
\]

\begin{theorem} \label{thmLocLim}
The pointed network $(\mathscr{G}^\dagger,\, x^\dagger)$ is the local
weak limit of $(\mathscr{G}_n)_{n\geq 1}$.
\end{theorem}

Before proving this theorem, let us point out that $(\mathscr{G}^\dagger, x^\dagger)$
can also be constructed as follows; we leave it to the reader to convince themself of the
equivalence of the definitions:
\begin{itemize}
  \item Let $v^\dagger$ be the focal vertex, and let
    $(v^\dagger, v_1, v_2,\dots)$ be an infinite spine going towards the root.
  \item For each $k\geq 1$, add $M^*_k-1$ children to $v_k$,
    where $(M^*_k)_{k\geq 1}$ is an i.i.d.\ sequence with the
    size-biased distribution of $\hat{M}$.
  \item Let $\mathscr{X}^\dagger \sim \mathscr{L}(\mathscr{X} \biasedby L\zeta^M)$
    where $\mathscr{X}$ is a generic color network and $M$ and $L$
    are respectively its number of mutations and total length.
    Write $M^\dagger$ and $L^\dagger$ for the corresponding quantities in
    $\mathscr{X}^\dagger$, and add $M^\dagger$ children to $v^\dagger$.
  \item Let each children of $v^\dagger$, as well as each of the children
    that were added to the nodes $(v_k)_{k\geq 1}$, be the root of a
    Galton--Watson tree with offspring distribution $\hat{M}$.
    Let $\mathscr{T}^\dagger$ be the resulting tree.
  \item Decorate each node $v\in \mathscr{T}^\dagger$, $v \neq v^*$,
    with a network $\mathscr{X}_v$ having the law of a
    generic color network $\mathscr{X}$ conditioned to have as many mutations as
    the number of children of $v$.
    Let $\mathscr{X}^\dagger$ be the decoration of $v^\dagger$.
  \item Conditional on $\mathscr{X}^\dagger$, let
    $x^\dagger \sim \lambda_{\mathscr{X}^\dagger}/L^\dagger$ be the focal
    point of $(\mathscr{G}^\dagger, x^\dagger)$.
\end{itemize}

\begin{proof}[Proof of Theorem~\ref{thmLocLim}]
Let $x_n$ be a uniformly chosen point of $\mathscr{G}_n$, and
let $v^*_n$ be the vertex of $\mathscr{T}_n$ such that
$x_n \in \smash{\mathscr{X}_{v^*_n}}$.  Remember that, since here we
view $\mathscr{G}_n$ as the tree $\mathscr{T}_n$ decorated with \emph{pointed}
networks, if we let $x_n$ be the focal point of
$\smash{\mathscr{X}_{v^*_n}}$ then $(\mathscr{G}_n, x_n)$ and
$(\mathscr{T}_n, v^*_n)$ can be seen as the same object. Thus, we will
use these two notations interchangeably.
  
To prove the theorem, it suffices to show that
\begin{equation}\label{eqGoalLocLim}
  \Expec*{\big}{F(\mathscr{G}_n, x_n)} \;\tendsto{\,n\to\infty\,}\;
  \Expec{F(\mathscr{G}^\dagger,x^\dagger)}
\end{equation}
for any function $F\colon\locspace\to\R$ of the form
\[
  F(\mathscr{\mathscr{G}}\!,\, x) \;=\;
  F(\mathscr{\mathscr{T}}\!,\, v^*) \;=\;
  \Indic{B_{\mathscr{T}}(v^*\!,\,r)=\mathbf{t}}\,
  \prod_{v\in \mathbf{t}} F_v(\mathscr{X}_v, x_v)\,,
\]
where $r$ is a positive integer; $\mathbf{t}$ is a finite pointed rooted plane tree; and
$(F_v)_{v\in \mathbf{t}}$ is a family of nonnegative continuous bounded maps $\mathscr{D}\to\R$
such that for all $v \neq v^*\!$, $F_v(\mathscr{X}_v, x_v) = \tilde{F}_v(\mathscr{X}_v)$
depends only on $\mathscr{X}_v$.
\newtext{Note that any such map $F$ is continuous for the local topology and bounded, with $\Norm[\infty]{F} \defas \sup_{(\mathscr{G},x)} \Abs{F(\mathscr{G},x)}\leq \prod_{v\in \mathbf{t}}\sup_{(\mathscr{X},x)}\Abs{F_v(\mathscr{X},x)}<\infty$.}
Moreover, since the map $\mathscr{X}\mapsto L_{\mathscr{X}}$ giving 
the total length of a network is continuous, we can restrict ourselves to
functions~$F$ for which there exists $\ell>0$ such that, for all
$v\in \mathbf{t}$, $F_v(\mathscr{X}, x)=0$ if $L_{\mathscr{X}}>\ell$.

\newcommand{\Lhat}{\smash{\hat{\scalebox{0.9}{\ensuremath L}}}\vphantom{L}}

Let us show that to finish the proof it suffices to show that
\begin{equation} \label{eqProofLocLimit02}
  \Expec*{\bigg}{
    \Big\lvert\Big(\!\!\!\sum_{\;\;\;v\in \mathscr{T}_n} \!\!\! L_v\Big)^{-1}
    \!\!-\, \frac{1}{n\Expec{\Lhat}}\Big\rvert
    \sum_{v\in \mathscr{T}_n}\int F(\mathscr{G}_n, x)\,\lambda_{\mathscr{X}_v}(dx)}
   \;\tendsto{\,n\to\infty\,}\; 0 \,,
\end{equation}
where $\Lhat \sim \mathscr{L}(L \biasedby \zeta^M)$ is the total length
of a generic decoration of the critical Galton--Watson tree
$\smash{\hat{\mathscr{T}}}$ such that
$\mathscr{T}_n \sim (\hat{\mathscr{T}} \,|\, \hat{\mathscr{T}} \text{ has } n \text{ vertices})$.
By definition of $x_n$,
\[
  \Expec*{\big}{F(\mathscr{G}_n,x_n)} \;=\;
  \Expec*{\bigg}{\Big(\!\!\!\sum_{\;\;\;v\in \mathscr{T}_n} \!\!\! L_v\Big)^{-1}
  \!\sum_{v\in \mathscr{T}_n} \int\! F(\mathscr{G}_n,x) \,\lambda_{\mathscr{X}_v}(dx)}.
\]
Thus, if \eqref{eqProofLocLimit02} holds we have
\begin{align*}
  \lim_{n\to\infty}\Expec*{\big}{F(\mathscr{G}_n,x_n)}  \;&=\;
  \lim_{n\to\infty}\, \frac{1}{n \, \Expec{\Lhat}}\,
  \Expec{\!\!\sum_{\;\;v\in \mathscr{T}_n} \! \int\! F(\mathscr{G}_n,x) \,\lambda_{\mathscr{X}_v}(dx)}\\
  \;&=\; \frac{1}{\Expec{\Lhat}}\, \lim_{n\to\infty}  \Expec{\frac{1}{n}
  \!\!\sum_{\;\;v\in \mathscr{T}_n} \!\! G(\mathscr{T}_n, v)} \,, 
\end{align*}
where the map
\begin{align*}
  G(\mathscr{T}_n, v)
  \;\defas&\; \int\! F(\mathscr{G}_n, x) \,\lambda_{\mathscr{X}_v}(dx) \\
  \;=&\;\; \Indic{B_{\mathscr{T}_n}(v^*\!,\,r)=\mathbf{t}}
  \mleft(\prod_{\;\;\;\mathclap{v\neq v^*}}\tilde{F}_v(\mathscr{X}_v)\mright)
  \int\! \, F_{v^*}(\mathscr{X}_{v^*}, x) \,\lambda_{\mathscr{X}_{v^*}}(dx)
\end{align*}
is continuous (and bounded by $\Norm[\infty]{F}\ell$) for the local topology, because the maps
$\mathscr{X} \mapsto \int\! F_v(\mathscr{X},x) \,\lambda_{\mathscr{X}}(dx)$
are continuous on $\mathscr{D}$. Since the pointed tree $(\mathscr{T}^*\!,\,
v^*)$ used in the construction of $(\mathscr{G}^\dagger, x^\dagger)$ is the
local weak limit of $\mathscr{T}_n$, we get
\begin{align*}
  \lim_{n\to\infty}  \Expec{\frac{1}{n}
  \!\!\sum_{\;\;v\in \mathscr{T}_n} \!\! G(\mathscr{T}_n, v)}
  \;&=\; \Expec*{\big}{G(\mathscr{T}^*\!,\, v^*)} \\[-2ex]
    &=\; \Expec*{\bigg}{\int\! F(\mathscr{G}^*\!,\, x)\,\lambda_{\mathscr{X}_{v^*}}\!(dx)} \\
    &=\; \Expec{\Lhat}\,\Expec{F(\mathscr{G}^\dagger,x^\dagger)}\,,
\end{align*}
where the last equality holds because, by definition, 
$(\mathscr{G}^\dagger,\,x^\dagger) \sim\mathscr{L}((\mathscr{G}^*\!,\, x^*) \biasedby L^*)$
where $L^* = \int d\lambda_{\mathscr{X}_{v^*}} \sim \Lhat$.
Putting the pieces together, this proves \eqref{eqGoalLocLim}.

Let us now prove~\eqref{eqProofLocLimit02}, i.e.\ show that $\Expec{Y_n} \to 0$,
where
\[
  Y_n \;=\;
  \Big\lvert\Big(\!\!\!\sum_{\;\;\;v\in \mathscr{T}_n} \!\!\! L_v\Big)^{-1}
  \!\!-\, \frac{1}{n\Expec{\Lhat}}\Big\rvert
  \sum_{v\in \mathscr{T}_n}\int F(\mathscr{G}_n, x)\,\lambda_{\mathscr{X}_v}(dx) \,.
\]
For this, on the one hand, note that $F$ is bounded and
\begin{equation} \label{eqProofLocLimit03}
  \sum_{v\in \mathscr{T}_n}\int F(\mathscr{G}_n, x)\,\lambda_{\mathscr{X}_v}(dx)
  \;\leq\;
  \Norm[\infty]{F} \! \sum_{\;\;v \in \mathscr{T}_n} \!\!\! L_v,
\end{equation}
and that on the other hand, since we have assumed that if the total
length of the subnetwork containing $x$ is greater than $\ell$ then
$F(\mathscr{G}, x)=0$, we also have
\[
  \sum_{v\in \mathscr{T}_n}\int F(\mathscr{G}_n, x)\,\lambda_{\mathscr{X}_v}(dx)
  \;\leq\;
   \Norm[\infty]{F}\, n\, \ell\,.
\]
As a result, $Y_n \leq \Norm[\infty]{F}(1 + \ell/\Expec{\Lhat})$. Thus, by dominated
convergence, to prove that $\Expec{Y_n} \to 0$ it suffices to show that $Y_n \to 0$
in probability. Using again~\eqref{eqProofLocLimit03},
\[
  0 \;\leq\; Y_n \;\leq\;
  \Big\lvert n\, \Expec{\Lhat} - \!\!\sum_{\;\;\;v\in \mathscr{T}_n} \!\!\! L_v \Big\rvert
  \; \frac{\Norm[\infty]{F}}{n\Expec{\Lhat}} \,.
\]
Finally, by Lemma~\ref{lemLExpMomentUnderMBiais} we can
apply point (ii) of Lemma~\ref{lemUniformBoundSumF} to the random variables
$(L_v)_{v\in \mathscr{T}_n}$ to get that for any $\epsilon > 0$,
\[
  \Big\lvert n\, \Expec{\Lhat} - \!\!\sum_{\;\;\;v\in \mathscr{T}_n} \!\!\! L_v \Big\rvert
  \;=\;o_p(n^{1/2\,+\,\epsilon}) \,,
\]
concluding the proof.
\end{proof}

\subsection{Geometry of the focal and spinal networks}

In order to complete the picture of the local weak limit of
$(\mathscr{G}_n)_{n\geq 1}$, let us zoom in on the decorations composing
$\mathscr{G}^\dagger$ and describe their distributions more finely than in the
previous section.  Specifically, we are interested in
\begin{itemize}
  \item $(\mathscr{X}^\dagger, x^\dagger)$, the focal network and its distinguished point;
  \item $(\mathscr{X}^\diamond, x^\diamond)$, which we call a \emph{spinal network}.
    This network is distributed as the color network that is
    the parent of the focal network, and its
    distinguished point is the mutation point that corresponds to the root of the focal network.
\end{itemize}

Recall that, by the construction of $\mathscr{G}^\dagger$ given
in Section~\ref{secConstructionLocLim}, these objects satisfy, for any positive measurable functional $F$ on pointed color networks:
\begin{gather}\label{eqFocalNetwork}
  \Expec{F(\mathscr{X}^\dagger, x^\dagger)} \;=\;
     \frac{1}{\Expec{L\zeta^M}}\, \Expec*{\Big}{\zeta^M\!\int F(\mathscr{X},x)\,\lambda_{\mathscr{X}}(dx)}\\
  \Expec{F(\mathscr{X}^\diamond, x^\diamond)} \;=\;
     \frac{1}{\Expec{M\zeta^M}}\,\Expec*{\Big}{\zeta^M\! \sum_{x\in \mathcal{M}}\!\!F(\mathscr{X},x)}\,
\end{gather}
where, by a slight abuse of notation, $\mathcal{M}$ denotes the point process
of mutations on the space $\mathscr{X}$ (previously, $\mathcal{M}$
denoted the point process on $\R$ corresponding to the mutation times).

Our next result shows that focal and spinal networks can be constructed by
``glueing'' two half-networks that are independent conditional on their number
of tips. Moreover, there is an explicit procedure to built these networks from
their profile, i.e.\ from the process giving their number of lineages as a
function of time. Let us start by introducing some notation.

Let $I = [t_0, t_1]$, with $t_0 < 0 \leq t_1$, and let $\gamma=(\gamma_t)_{t\in
I}$ be a càdlàg, positive except at time $t_1$, integer-valued trajectory consisting of a finite
number of $\pm 1$ jumps, starting at $1$ and ending with a jump to $0$.
As usual, let $\mathscr{X}$ denote a generic color network, and let
$X$ be the corresponding logistic branching process.
With a slight abuse of notation, we will write $\Set{X=\gamma}$ for the event on which the trajectory of the Markov chain $X$, started from~$1$ at time $t_0$, is exactly $\gamma$.
\newtext{Note that for any $t_0\in \R$, it makes sense to consider a random network $\mathscr{X}$ started from a single individual at time $t_0$, for which the logistic branching process $X$ started from $1$ at time $t_0$ is the ``number of lineages'' process.
The distribution of $\mathscr{X}$ does not depend on $t_0$.}

\begin{definition} \label{defXscrGamma}
The random pointed network $\mathscr{X}[\gamma]$ is defined as
\[
   \mathscr{X}[\gamma] \;\sim\;
   \mathscr{L}(\,\mathscr{X} \;|\; X = \gamma\,)\,,
\]
and the focal point is chosen uniformly among the points of the networks that
correspond to lineages alive at time~0. See Figure~\ref{figNetworkFromPath} for an illustration.
In the case where $\gamma$ has a downward jump at time $0$, we also define
\[
  \mathscr{X}^{\mathfrak{m}}[\gamma] \;\sim\;
  \mathscr{L}\big(\,\mathscr{X}[\gamma] \;\big|\; \mathscr{X}[\gamma] \text{ has a mutation a time 0}\big)\,,
\]
and the focal point is the point that corresponds to the mutation at
time $0$.  
\end{definition}

Note that it is straightforward to sample the networks $\mathscr{X}[\gamma]$ and 
$\mathscr{X}^{\mathfrak{m}}[\gamma]$ introduced in Definition~\ref{defXscrGamma}:
in the case of $\mathscr{X}[\gamma]$, start with a single lineage, at time $t_0$. Then,
going through the jumps of $\gamma$ in chronological order:
\begin{itemize}
  \item For each jump from $k$ to $k+1$ at time $t$, pick one a lineage
    uniformly at random among the lineages alive at time $t$, and let it split
    into two lineages.
  \item For each jump from $k$ to $k-1$ at time $t$, choose one of the following
    possibilities: with probability $\mu/\rho_k$,
    pick a lineage alive at time $t$ uniformly at random, and let it mutate;
    with probability $\alpha/\rho_k$, pick a lineage similarly and
    let it die; and with probability $1 - (\alpha+\mu)/\rho_k$,
    pick a pair of lineages uniformly at random and let them merge together.
\end{itemize}
The network $\mathscr{X}^{\mathfrak{m}}[\gamma]$ is obtained similarly, with 
the additional constraint that the jump from $k$ to $k-1$ at time $0$ is a mutation.

\begin{figure}[ht!] \centering
  \includegraphics[width=0.9\linewidth]{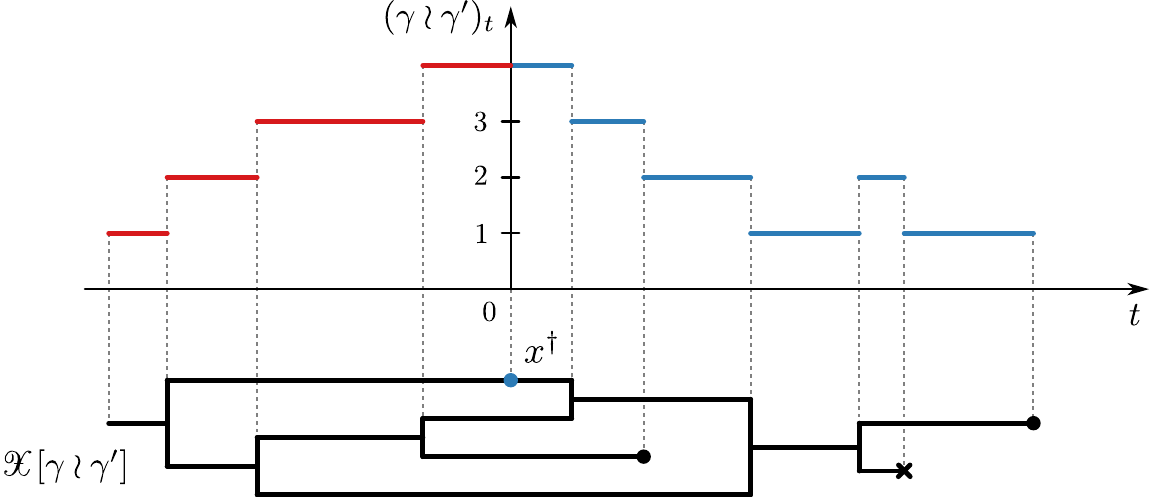}
  \caption{Construction of $\mathscr{X}[\gamma\paste\gamma']$.
The trajectory on top is the back-to-back pasting of two trajectories
$\gamma$ and $\gamma'$ started from $4$: the red part corresponds to the
time-reversal of~$\gamma$ and the blue one to $\gamma'$.
The network $\mathscr{X}[\gamma\paste\gamma']$ is represented in black on
the bottom.  The two black dots correspond to mutations, and the cross to a
death.  The blue dot represents the focal point $x^{\dagger}$, which is
chosen uniformly at random among the lineages alive at time~$0$.
} \label{figNetworkFromPath}
\end{figure}

Now, recall the following notation, introduced in Section~\ref{secLogistic}:
$\nucirc$ is the probability measure on the positive integers characterized by $\nucirc(n) \propto \prod_{k=1}^{n}\frac{1}{\rho_k}$,
see Eq.~\eqref{eqDefNuCirc}; $\mathscr{L}(A\biasedby B)$ denotes the
distribution of $A$ biased by $B$, see Notation~\ref{notationBiasedTrajectory};
and $\gamma \wr \gamma'$ is the back-to-back pasting of two càdlàg
trajectories, see Definition~\ref{defBackToBackPasting}.

\begin{proposition} \label{propFocalSpinalNetworks}
For each $k \geq 0$, let $X'_k$ and $X''_k$ be independent realizations of~$X$
that are started from $k$ and also independent of everything else,
and let $K\sim \nucirc$.
\begin{mathlist}
  \item $(\mathscr{X}^\dagger,x^\dagger)$ is distributed as
  $\displaystyle\mathscr{L}\big(\mathscr{X}[X'_K \paste X''_K] \;\biasedby[big]\; \zeta^{M^\dagger}\,\big)$;
  \item $(\mathscr{X}^\diamond,x^\diamond)$ is distributed as
  $\displaystyle\mathscr{L}\big(\mathscr{X}^{\mathfrak{m}}[X'_K\paste X''_{K-1}] \;\biasedby[big]\; \zeta^{M^\diamond}\big)$;
\end{mathlist}
where $M^\dagger$ denotes the number of mutations of
$\,\mathscr{X}[X'_K \paste X''_K]$ and $M^\diamond$ the number of mutations
of $\,\mathscr{X}^{\mathfrak{m}}[X'_K\paste X''_{K-1}]$.
\end{proposition}

\begin{remark}
This construction of the focal/spinal networks makes it possible to get
expressions for some characteristics of the local weak limit
$(\mathscr{G}^\dagger, x^\dagger)$. For instance, if we let $N$ be the number
of lineages of the same color as $x^\dagger$ that are alive at the same time as
$x^\dagger$, then by (i) we have, for a normalizing constant $C\geq 0$,
\[
  \Prob{N=k} \;=\; C\,\nucirc(k)\,\Expec*[k]{\normalsize}{\zeta^M}^2,
\]
where $\E_{k}$ denotes the expectation conditional on $\Set{X_0=k}$.
The limitation comes from the fact that if $\zeta\neq 1$, then the expressions
$\Expec*[k]{\normalsize}{\zeta^M}$ are not explicit. However, they can
expressed as continuous fractions, which would it possible compute them numerically
(see Theorem~\ref{thmM}).

Similarly, if $T$ denotes the time since the last mutation
in the ancestry of $x^\dagger$, then for any bounded measurable function $F\colon\R\to\R$,
\[
  \Expec{F(T)\Indic{N=k}} \;=\; C\,\nucirc(k)\,\Expec*[k]{\normalsize}{\zeta^M}\,
  \Expec[k]{F(T_0)\,\zeta^M},
\]
where here $T_0$ denotes the hitting time of $0$ for the process $X$.
\end{remark}

\begin{proof}[Proof of Proposition~\ref{propFocalSpinalNetworks}]
The proof is very similar to that of Proposition~\ref{propNuCircXfromMut}, and
also relies on the path decomposition Markov chains presented in
Section~\ref{secPathDecomposition}.

Let $Y$ be the birth-death chain on $\N$ that goes from $k$ to $k+1$ at rate 1
and from $k$ to $k-1$ at rate $\rho_k$. Note that $Y$ is distributed as
the chain $X$ slowed-down by a factor $k$ when in state~$k$, and resurrected
at rate $1$ when it hits 0. Thus, $Y$ is positive recurrent -- and
therefore, reversible (as any positive recurrent birth-death chain).
Moreover,  it is straightforward to check that the stationary distribution of $Y$ is
the probability distribution $\pi$ defined by
\[
  \pi(n) \;=\; C\prod_{k=1}^n\frac{1}{\rho_k},\qquad n\geq 0,
\]
where $C$ is a normalizing constant.
Note that by definition of $\nu_\circ$, if $K_0$ is a random variable with
distribution $\pi$, then its conditional distribution given $\{K_0 \geq 1\}$
is~$\nu_\circ$.

Let $Y$ be started from $1$, and denote $T_0$ the hitting time of $0$.
Conditional on~$T_0$, let $U$ be uniform on $[0,T_0]$, and set $K\defas Y_U$.
Now, from the trajectory of $(Y_t)_{t\in [0,T_0]}$, construct a path with the
same distribution as $X$ by speeding up time by a factor~$k$ when in state $k$.
Let $V$ be the point corresponding to $U$ in the new timescale.
Note that this shows that $T_0$, the hitting time of $0$ by $Y$, has the same
distribution as $L = \int_0^\infty X_t\, dt$, since in this construction
the two quantities are equal.
  
\newcommand{\XscrTilde}{\scalebox{0.95}{$\tilde{\mathscr{X}}$}}

Now, consider the biased probability measure
\[
  \P_{0}(\cdot) \;\defas\; \frac{\Expec*{\normalsize}{\Indic{\cdot}T_0}}{\Expec{T_0}} \,.
\]
By Proposition~\ref{propPathDecompFromUniformPoint}, under $\P_{0}$ we have
$K\sim \nucirc$ and, conditional on $K$, letting
  $\Theta_u^t\{Y\} \defas(Y_{u + s})_{-u \leq s \leq t - u}$, we have
\[
  \Theta^{T_0}_U\{Y\} \;\overset{d}{=}\; Y'\paste Y''
\]
where $Y'$ and $Y''$ are independent copies of $Y$ started from $K$.
In other words, for any measurable positive functional $F$ of trajectories:
\[
  \ExpecBrackets*{\big}{T_0\, F\big(\Theta^{T_0}_U\{Y\}\big)} \;=\;
  \Expec{T_0}\,\ExpecBrackets*{\big}{F(Y'_K\paste Y''_K)}.
\]
Using our coupling of $X$ and $Y$, and recalling that $T_0 = L$,
this yields:
\begin{equation} \label{eqPathDec1}
  \ExpecBrackets*{\big}{L\, F\big(\Theta^{\tau_0}_V\{X\}\big)} \;=\;
  \Expec{L}\,\ExpecBrackets*{\big}{F(X'_K\paste X''_K)}.
\end{equation}
where $\tau_0$ is the extinction time of $X$, and $X'_K$
and $X''_K$ are independent copies of $X$ started from $K$, as they are
defined in the statement of the proposition.

Let us now define a pointed network $(\XscrTilde, x^*)$ as follows:
conditional on the trajectory of $X$ constructed above, using
Definition~\ref{defXscrGamma} let
$(\XscrTilde,x^*)= \mathscr{X}[\Theta^{\tau_0}_V\{X\}]$. Thus,
$\XscrTilde$ is distributed as a standard color network whose root is
located at time~$-V$.
Recall that by definition, conditional on~$\XscrTilde$, the focal point $x^*$ is chosen uniformly at random
among the points that correspond to the $K$ lineages alive at time~0.

By construction, $x^*$ is uniform on
$\XscrTilde$ with respect to its length measure. Therefore, for any
measurable positive functional $F$ on pointed networks:
\[
  \Expec{\zeta^M\!\! \int\! F(\mathscr{X},x)\,\lambda_{\mathscr{X}}(dx)}\;=\;
  \Expec{L\,\zeta^M\!\! \int\! F(\mathscr{X},x)\,\tfrac{1}{L}\lambda_{\mathscr{X}}(dx)}\;=\;
  \Expec{L\, \zeta^MF(\XscrTilde, x^*)}\,, 
\]
where $M$ is the number of mutations of $\XscrTilde$.
Moreover, by
applying Equation~\eqref{eqPathDec1} to
the functional
$\gamma \mapsto \Expec{\zeta^{M(\mathscr{X}[\gamma])}F(\mathscr{X}[\gamma])}$,
we get
\[
  \Expec{L\,\zeta^MF(\tilde{\mathscr{X}}, x^*)} \;=\;
  \Expec{L}\,\Expec{\zeta^{M^\dagger} F(\mathscr{X}[X'_K\paste X''_k])},
\]
where $M^\dagger$ is the total number of mutations of 
$\mathscr{X}[X'_K\paste X''_k]$. Therefore,
\[
  \Expec{\zeta^{M^\dagger} F(\mathscr{X}[X'_K\paste X''_k])} \;=\;
  \frac{1}{\Expec{L}}\Expec{\zeta^M\! \int F(\mathscr{X},x)\,\lambda_{\mathscr{X}}(dx)} \,.
\]
Finally, taking $F\equiv 1$, we get
$\Expec*{\normalsize}{\zeta^{M^\dagger}} = \Expec*{\normalsize}{L\zeta^M} /
\Expec{L}$, and so comparing the previous display with
Equation~\eqref{eqFocalNetwork} characterizing the law of
$(\mathscr{X}^\dagger, x^\dagger)$, we see that
\[
  \Expec*{\normalsize}{\zeta^{M^\dagger}}^{-1}\, 
  \Expec{\zeta^{M^\dagger} F(\mathscr{X}[X'_K\paste X''_k])} 
  \;=\; \Expec{F(\mathscr{X}^\dagger, x^\dagger)}\,, 
\]
finishing the proof of point~(i).

Point~(ii) is proved similarly, but using
Proposition~\ref{propPathDecompFromTransition} instead of
Proposition~\ref{propPathDecompFromUniformPoint} to view the network from a
uniform mutation point instead of from a uniform point.
\end{proof}

\section*{Acknowledgments}

FB was supported by Dr.\ Max Rössler, the Walter Haefner Foundation and the ETH
Zürich Foundation. 

\addcontentsline{toc}{section}{References}
\bibliographystyle{abbrvnat}
\bibliography{refs}

\begin{thebibliography}{49}
\providecommand{\natexlab}[1]{#1}
\providecommand{\url}[1]{\texttt{#1}}
\expandafter\ifx\csname urlstyle\endcsname\relax
  \providecommand{\doi}[1]{doi: #1}\else
  \providecommand{\doi}{doi: \begingroup \urlstyle{rm}\Url}\fi

\bibitem[Acosta-Hum{\'a}nez et~al.(2020)Acosta-Hum{\'a}nez, Capit{\'a}n, and
  Morales-Ruiz]{ACM20}
P.~B. Acosta-Hum{\'a}nez, J.~A. Capit{\'a}n, and J.~J. Morales-Ruiz.
\newblock Integrability of stochastic birth-death processes via differential
  galois theory.
\newblock \emph{Mathematical Modelling of Natural Phenomena}, 15:\penalty0 70,
  2020.
\newblock \doi{10.1051/mmnp/2020005}.

\bibitem[Addario-Berry et~al.(2013)Addario-Berry, Devroye, and Janson]{ADJ13}
L.~Addario-Berry, L.~Devroye, and S.~Janson.
\newblock Sub-gaussian tail bounds for the width and height of conditioned
  {G}alton--{W}atson trees.
\newblock \emph{The Annals of Probability}, 41\penalty0 (2):\penalty0
  1072--1087, 2013.
\newblock \doi{10.1214/12-AOP758}.

\bibitem[Aldous(1991{\natexlab{a}})]{Ald91CRT}
D.~Aldous.
\newblock The continuum random tree {I}.
\newblock \emph{The Annals of Probability}, 19\penalty0 (1):\penalty0 1--28,
  1991{\natexlab{a}}.
\newblock \doi{10.1214/aop/1176990534}.

\bibitem[Aldous(1991{\natexlab{b}})]{Ald91Fringe}
D.~Aldous.
\newblock Asymptotic fringe distributions for general families of random trees.
\newblock \emph{The Annals of Applied Probability}, 1\penalty0 (2):\penalty0
  228--266, 1991{\natexlab{b}}.
\newblock \doi{10.1214/aoap/1177005936}.

\bibitem[Bansaye et~al.(2016)Bansaye, M{\'e}l{\'e}ard, and Richard]{BMR16}
V.~Bansaye, S.~M{\'e}l{\'e}ard, and M.~Richard.
\newblock Speed of coming down from infinity for birth-and-death processes.
\newblock \emph{Advances in Applied Probability}, 48\penalty0 (4):\penalty0
  1183--1210, 2016.
\newblock \doi{10.1017/apr.2016.70}.

\bibitem[Bapteste et~al.(2013)Bapteste, {van Iersel}, Janke, Kelchner, Kelk,
  McInerney, Morrison, Nakhleh, Steel, Stougie, and Whitfield]{Bapteste2013}
E.~Bapteste, L.~{van Iersel}, A.~Janke, S.~Kelchner, S.~Kelk, J.~O. McInerney,
  D.~A. Morrison, L.~Nakhleh, M.~Steel, L.~Stougie, and J.~Whitfield.
\newblock Networks: {e}xpanding evolutionary thinking.
\newblock \emph{Trends in Genetics}, 29\penalty0 (8):\penalty0 439--441, 2013.
\newblock \doi{10.1016/j.tig.2013.05.007}.

\bibitem[Benjamini and Schramm(2001)]{BeS01}
I.~Benjamini and O.~Schramm.
\newblock Recurrence of distributional limits of finite planar graphs.
\newblock \emph{Electronic Journal of Probability}, 6:\penalty0 1 -- 13, 2001.
\newblock \doi{10.1214/EJP.v6-96}.

\bibitem[Bergthorsson et~al.(2003)Bergthorsson, Adams, Thomason, and
  Palmer]{Bergthorsson2003}
U.~Bergthorsson, K.~L. Adams, B.~Thomason, and J.~D. Palmer.
\newblock Widespread horizontal transfer of mitochondrial genes in flowering
  plants.
\newblock \emph{Nature}, 424:\penalty0 197--201, 2003.
\newblock \doi{10.1038/nature01743}.

\bibitem[Bienvenu et~al.(2022)Bienvenu, Lambert, and Steel]{BLS22}
F.~Bienvenu, A.~Lambert, and M.~Steel.
\newblock Combinatorial and stochastic properties of ranked tree-child
  networks.
\newblock \emph{Random Structures \& Algorithms}, 60:\penalty0 653--689, 2022.
\newblock \doi{10.1002/rsa.21048}.

\bibitem[Billingsley(1999)]{Bil99}
P.~Billingsley.
\newblock \emph{Convergence of Probability Measures}.
\newblock Wiley Series in Probability and Statistics. {John Wiley \& Sons,
  Inc., New York}, second edition, 1999.
\newblock ISBN 0-471-19745-9.
\newblock \doi{10.1002/9780470316962}.

\bibitem[Bosq and Nguyen(1996)]{BoNg96}
D.~Bosq and H.~T. Nguyen.
\newblock \emph{A Course in Stochastic Processes: Stochastic Models and
  Statistical Inference}.
\newblock Springer Science+Business Media, 1996.
\newblock \doi{10.1007/978-94-015-8769-3}.

\bibitem[Caraceni et~al.(2022)Caraceni, Fuchs, and Yu]{CFY22}
A.~Caraceni, M.~Fuchs, and G.-R. Yu.
\newblock Bijections for ranked tree-child networks.
\newblock \emph{Discrete Mathematics}, 345\penalty0 (9):\penalty0 112944, 2022.
\newblock \doi{10.1016/j.disc.2022.112944}.

\bibitem[Casanova et~al.(2021)Casanova, Pardo, and P{\'e}rez]{CaPL21}
A.~G. Casanova, J.~C. Pardo, and J.~L. P{\'e}rez.
\newblock Branching processes with interactions: Subcritical cooperative
  regime.
\newblock \emph{Advances in Applied Probability}, 53\penalty0 (1):\penalty0
  251--278, 2021.
\newblock \doi{10.1017/apr.2020.59}.

\bibitem[Crump and Mode(1968)]{CrMo68}
K.~S. Crump and C.~J. Mode.
\newblock A general age-dependent branching process. {I}.
\newblock \emph{Journal of mathematical analysis and applications}, 24\penalty0
  (3):\penalty0 494--508, 1968.
\newblock \doi{10.1016/0022-247X(68)90005-X}.

\bibitem[Crump and Mode(1969)]{CrMo69}
K.~S. Crump and C.~J. Mode.
\newblock A general age-dependent branching process. {II}.
\newblock \emph{Journal of mathematical analysis and applications}, 25\penalty0
  (1):\penalty0 8--17, 1969.
\newblock \doi{10.1016/0022-247X(69)90210-8}.

\bibitem[Curien(2018)]{Cur18}
N.~Curien.
\newblock Random graphs: {t}he local convergence point of view.
\newblock \emph{Lecture notes}, 2018.
\newblock
  \url{https://www.imo.universite-paris-saclay.fr/~curien/cours/cours-RG.pdf}.

\bibitem[Dagan and Martin(2009)]{Dagan2009getting}
T.~Dagan and W.~Martin.
\newblock Getting a better picture of microbial evolution en route to a network
  of genomes.
\newblock \emph{Philosophical Transactions of the Royal Society B: Biological
  Sciences}, 364\penalty0 (1527):\penalty0 2187--2196, 2009.
\newblock \doi{10.1098/rstb.2009.0040}.

\bibitem[Doolittle and Bapteste(2007)]{Doolittle2007pattern}
W.~F. Doolittle and E.~Bapteste.
\newblock Pattern pluralism and the {T}ree of {L}ife hypothesis.
\newblock \emph{Proceedings of the National Academy of Sciences}, 104\penalty0
  (7):\penalty0 2043--2049, 2007.
\newblock \doi{10.1073/pnas.0610699104}.

\bibitem[Durrett(2010)]{Dur10}
R.~Durrett.
\newblock \emph{Probability: Theory and Examples}.
\newblock Cambridge University Press, Cambridge, 2010.
\newblock \doi{10.1017/CBO9780511779398}.

\bibitem[Dvoretzky and Motzkin(1947)]{DvM47}
A.~Dvoretzky and T.~Motzkin.
\newblock A problem of arrangements.
\newblock \emph{Duke Mathematical Journal}, 14\penalty0 (2):\penalty0 305--313,
  1947.
\newblock \doi{10.1215/S0012-7094-47-01423-3}.

\bibitem[Dwass(1969)]{Dwa69}
M.~Dwass.
\newblock The total progeny in a branching process and a related random walk.
\newblock \emph{Journal of Applied Probability}, 6\penalty0 (3):\penalty0
  682--686, 1969.
\newblock \doi{https://doi.org/10.2307/3212112}.

\bibitem[Evans(2008)]{Eva08}
S.~N. Evans.
\newblock \emph{Probability and Real Trees}.
\newblock \'Ecole d'\'Et{\'e} de Probabilit{\'e}s de {S}aint-{F}lour
  {XXXV}-2005. Springer Berlin, Heidelberg, 2008.
\newblock \doi{10.1007/978-3-540-74798-7}.

\bibitem[Flajolet and Sedgewick(2009)]{FlSe09}
P.~Flajolet and R.~Sedgewick.
\newblock \emph{Analytic combinatorics}.
\newblock Cambridge University Press, 2009.
\newblock \doi{10.1017/CBO9780511801655}.

\bibitem[Foucart et~al.(2020)Foucart, Li, and Zhou]{FLZ20}
C.~Foucart, P.-S. Li, and X.~Zhou.
\newblock On the entrance at infinity of {{Feller}} processes with no negative
  jumps.
\newblock \emph{Statistics \& Probability Letters}, 165:\penalty0 108859, 2020.
\newblock \doi{10.1016/j.spl.2020.108859}.

\bibitem[Fuchs et~al.(2022)Fuchs, Liu, and Yu]{FLY22}
M.~Fuchs, H.~Liu, and T.-C. Yu.
\newblock Limit theorems for patterns in ranked tree-child networks.
\newblock \emph{arXiv preprint}, 2022.
\newblock \arxiv{2204.07676}.

\bibitem[Hao and Golding(2004)]{Hao2004patterns}
W.~Hao and G.~Golding.
\newblock Patterns of bacterial gene movement.
\newblock \emph{Molecular biology and evolution}, 21\penalty0 (7):\penalty0
  1294--1307, 2004.
\newblock \doi{10.1093/molbev/msh129}.

\bibitem[Jagers(1975)]{Jag75}
P.~Jagers.
\newblock \emph{Branching processes with biological applications}.
\newblock Wiley, 1975.
\newblock (note: out-of-print, not available online).

\bibitem[Janson(2012)]{Jan12}
S.~Janson.
\newblock Simply generated trees, conditioned {G}alton--{W}atson trees, random
  allocations and condensation.
\newblock \emph{Probability Surveys}, 9:\penalty0 103--252, 2012.
\newblock \doi{10.1214/11-PS188}.

\bibitem[Kalinkin(2002)]{Kal02}
A.~V. Kalinkin.
\newblock Markov branching processes with interaction.
\newblock \emph{Russian Mathematical Surveys}, 57\penalty0 (2):\penalty0
  241--304, 2002.
\newblock \doi{10.1070/RM2002v057n02ABEH000496}.

\bibitem[Karlin and McGregor(1957{\natexlab{a}})]{KaMc57b}
S.~Karlin and J.~McGregor.
\newblock The classification of birth and death processes.
\newblock \emph{Transactions of the American Mathematical Society}, 86\penalty0
  (2):\penalty0 366--400, 1957{\natexlab{a}}.
\newblock \doi{10.2307/1993021}.

\bibitem[Karlin and McGregor(1957{\natexlab{b}})]{KaMc57a}
S.~Karlin and J.~L. McGregor.
\newblock The differential equations of birth-and-death processes, and the
  {S}tieltjes moment problem.
\newblock \emph{Transactions of the American Mathematical Society}, 85\penalty0
  (2):\penalty0 489--546, 1957{\natexlab{b}}.
\newblock \doi{10.2307/1992942}.

\bibitem[Kong et~al.(2022)Kong, Pons, Kubatko, and Wicke]{KPKW22}
S.~Kong, J.~C. Pons, L.~Kubatko, and K.~Wicke.
\newblock Classes of explicit phylogenetic networks and their biological and
  mathematical significance.
\newblock \emph{Journal of Mathematical Biology}, 84\penalty0 (47), 2022.
\newblock \doi{10.1007/s00285-022-01746-y}.

\bibitem[Kortchemski(2012)]{Kor2012}
I.~Kortchemski.
\newblock Invariance principles for {G}alton--{W}atson trees conditioned on the
  number of leaves.
\newblock \emph{Stochastic Processes and Their Applications}, 122\penalty0
  (9):\penalty0 3126--3172, 2012.
\newblock \doi{https://doi.org/10.1016/j.spa.2012.05.013}.

\bibitem[Lambert(2005)]{Lam05}
A.~Lambert.
\newblock The branching process with logistic growth.
\newblock \emph{The Annals of Applied Probability}, 15\penalty0 (2):\penalty0
  1506--1535, 2005.
\newblock \doi{10.1214/105051605000000098}.

\bibitem[Le~Gall(2005)]{LeGall05}
J.-F. Le~Gall.
\newblock Random trees and applications.
\newblock \emph{Probability Surveys}, 2:\penalty0 245--311, 2005.
\newblock \doi{10.1214/154957805100000140}.

\bibitem[Linder and Rieseberg(2004)]{Linder2004reconstructing}
C.~R. Linder and L.~H. Rieseberg.
\newblock Reconstructing patterns of reticulate evolution in plants.
\newblock \emph{American journal of botany}, 91\penalty0 (10):\penalty0
  1700--1708, 2004.
\newblock \doi{10.3732/ajb.91.10.1700}.

\bibitem[Mallet(2007)]{Mallet2007hybrid}
J.~Mallet.
\newblock Hybrid speciation.
\newblock \emph{Nature}, 446:\penalty0 279--283, 2007.
\newblock \doi{doi.org/10.1038/nature05706}.

\bibitem[Marckert and Mokkadem(2003)]{MaM03}
J.-F. Marckert and A.~Mokkadem.
\newblock The depth first processes of {G}alton--{W}atson trees converge to the
  same {B}rownian excursion.
\newblock \emph{The Annals of Probability}, 31\penalty0 (3):\penalty0
  1655--1678, 2003.
\newblock \doi{10.1214/aop/1055425793}.

\bibitem[Miermont(2009)]{Mie09}
G.~Miermont.
\newblock Tessellations of random maps of arbitrary genus.
\newblock \emph{Annales scientifiques de l'{\'E}cole normale sup{\'e}rieure},
  42\penalty0 (5):\penalty0 725--781, 2009.
\newblock \doi{10.24033/asens.2108}.

\bibitem[Mower et~al.(2004)Mower, Stefanovi{\'c}, Young, and
  Palmer]{Mower2004gene}
J.~P. Mower, S.~Stefanovi{\'c}, G.~J. Young, and J.~D. Palmer.
\newblock Gene transfer from parasitic to host plants.
\newblock \emph{Nature}, 432:\penalty0 165--166, 2004.
\newblock \doi{10.1038/432165b}.

\bibitem[Nerman(1981)]{Ner81}
O.~Nerman.
\newblock On the convergence of supercritical general ({C}-{M}-{J}) branching
  processes.
\newblock \emph{Zeitschrift f{\"u}r {W}ahrscheinlichkeitstheorie und verwandte
  {G}ebiete}, 57\penalty0 (3):\penalty0 365--395, 1981.
\newblock \doi{10.1007/BF00534830}.

\bibitem[Ojeda and Pardo(2020)]{OjPa22}
G.~B. Ojeda and J.~C. Pardo.
\newblock Branching processes with pairwise interactions.
\newblock \emph{arXiv preprint}, 2020.
\newblock \arxiv{2009.11820}.

\bibitem[Parsons(2018)]{Par18}
T.~L. Parsons.
\newblock Invasion probabilities, hitting times, and some fluctuation theory
  for the stochastic logistic process.
\newblock \emph{Journal of Mathematical Biology}, 77\penalty0 (4):\penalty0
  1193--1231, 2018.
\newblock \doi{10.1007/s00285-018-1250-x}.

\bibitem[Schertzer and Simatos(2018)]{SS18a}
E.~Schertzer and F.~Simatos.
\newblock Height and contour processes of {{Crump-Mode-Jagers}} forests
  ({{I}}): General distribution and scaling limits in the case of short edges.
\newblock \emph{Electronic Journal of Probability}, 23:\penalty0 1--43, 2018.
\newblock \doi{10.1214/18-EJP151}.

\bibitem[Srivastava(1998)]{Sri98}
S.~M. Srivastava.
\newblock \emph{A Course on {{Borel}} Sets}, volume 180 of \emph{Graduate
  {{Texts}} in {{Mathematics}}}.
\newblock {Springer Berlin Heidelberg}, {Berlin, Heidelberg}, 1998.
\newblock \doi{10.1007/978-3-642-85473-6}.

\bibitem[Stufler(2019)]{Stu19}
B.~Stufler.
\newblock Local limits of large {G}alton--{W}atson trees rerooted at a random
  vertex.
\newblock In \emph{Annales de l'Institut Henri Poincar{\'e}, Probabilit{\'e}s
  et Statistiques}, volume~55, pages 155--183, 2019.
\newblock \doi{10.1214/17-AIHP879}.

\bibitem[Stufler(2020)]{Stu20}
B.~Stufler.
\newblock Limits of random tree-like discrete structures.
\newblock \emph{Probability Surveys}, 17, 2020.
\newblock \doi{10.1214/19-PS338}.

\bibitem[Stufler(2022)]{Stu22}
B.~Stufler.
\newblock A branching process approach to level-$k$ phylogenetic networks.
\newblock \emph{Random Structures \& Algorithms}, 61\penalty0 (2):\penalty0
  397--421, 2022.
\newblock \doi{10.1002/rsa.21065}.

\bibitem[{v}an~{d}er {H}ofstad(2016)]{Van16}
R.~{v}an~{d}er {H}ofstad.
\newblock \emph{Random Graphs and Complex Networks}, volume~1.
\newblock Cambridge University Press, 2016.
\newblock \doi{10.1017/9781316779422}.

\end{thebibliography}

\appendix

\newpage 

\section*{Appendices}
\addcontentsline{toc}{section}{Appendices}
\refstepcounter{section} 

\subsection{Path decompositions of Markov chains} \label{secPathDecomposition}

In this appendix, we give a description of the trajectory of a Markov chain
as seen from a random point in time. The ideas are standard, but we could
not find the two propositions below in the literature.
Their proofs are elementary, but somewhat tedious; so since they are very
similar we present only the most involved of the two and leave the other one to
the reader. This appendix also contains the proof of Proposition~\ref{propNuCircXfromMut}.

Let $E$ be a countable set and let $Y$ be a continuous-time Markov chain on~$E$
with transition rate matrix $Q = (q_{ij})_{i,j \in E}$, started from the
initial state $0 \in E$.
Let us write $\tau$ for the first jump time of the chain and $T_0$ for the return
time to $0$, i.e.\ $T_0=\inf\Set{t \geq \tau \suchthat Y_t=0}$.  Assume that $Y$ is
positive recurrent, with stationary distribution $\pi$, and define the reversed
chain $Y'$ as the continuous-time Markov chain with transition rate matrix
$Q' = (q'_{ij})_{i,j\in E}$, where
\[
  q'_{ij} \;=\; \frac{\pi_j}{\pi_i}\, q_{ji} \,.
\]
\newtext{We will consider time-shifted trajectories $\gamma$ of the Markov
chain $Y$ killed upon reaching~$0$.  For this purpose, let us formally define a
convenient Skorokhod-like space of trajectories.  Let $E'$ denote
$E\cup\{\Delta\}$, where $\Delta\notin E$ is arbitrary, and let $\Gamma$ denote the
set of càdlàg functions $\gamma:\R\to E'$ such that for all $t$ with $\Abs{t}$
large enough we have $\gamma(t)=\Delta$.  With a slight abuse, for any $a<b\in
\R$ and $\gamma:\COInterval{a,b}\to E$ we identify $\gamma$ with $t\mapsto
\gamma(t)\Indic{t\in \COInterval{a,b}}+\Delta\Indic{t\notin \COInterval{a,b}}$.
The space $\Gamma$ can be endowed with the metric $d$ defined by:
\[
	d(\gamma, \gamma') = \begin{cases}
  1\wedge\int_{\R}\Indic{\gamma(t)\neq \gamma'(t)}\,dt &\text{if }\gamma\text{ and }\gamma'\text{ have the same number of jumps,}\\
  1 &\text{otherwise.}
  \end{cases}
\]
Note that time-shifts are then continuous in $\Gamma$.}

Recall from Notation~\ref{notationBiasedTrajectory} that for two random
variables $A$ and $B$ defined on the same probability space, we write
$\mathscr{L}(A\biasedby B)$ for the distribution of $A$ biased by $B$,
that is, under the probability
$\Prob{\,\cdot\biasedby B} = \Expec*{\normalsize}{\Indic{\cdot} B} \,/\, \Expec{B}$.
Conditional on $T_0$, let $U$ be a uniform random variable on $\ClosedInterval{\tau, T_0}$.
We are interested in the random trajectory
\[
 Z \;\sim\;
 \mathscr{L}\big((Y_{U+t})_{\tau-U \leq t < T_0 - U} \biasedby[big] T_0 - \tau\big),
\]
\newtext{seen as a random variable in $\Gamma$.}
This trajectory can be conveniently described by decomposing it into
its left and right parts.
For this, recall the ``back-to-back pasting'' operation introduced in
Definition~\ref{defBackToBackPasting}, which to two càdlàg functions
$f\colon\COInterval{0,T_f}\to E$ and $g\colon\COInterval{0,T_g}\to E$ associates
the function $f \paste g \colon \COInterval{-T_f, T_g} \to E$ defined by
\[
  f\paste g\colon t \mapsto
  \begin{cases}
    \lim_{s\downarrow t}f(-s) & \text{if } t < 0 \\
    g(t) & \text{if } t\geq 0\,.
  \end{cases}
\]

\begin{proposition} \label{propPathDecompFromUniformPoint}
With the definitions above, we have
\[
Z \;\overset{d}{=}\; Y' \paste Y'',
\]
where $Y'$ is the reversed chain and $Y''$ has the same transitions as $Y$.
Both chains are started from $Y'_0 = Y''_0 \sim \pi^*$, where
$\pi^*_i = \frac{\pi_i}{1-\pi_0}$ for all $i\in E\setminus\{0\}$, and stopped upon
reaching~$0$. Conditional on their common starting point, they are independent.
\end{proposition}

As discussed above, the proof is left to the reader.

Assume now that the transitions of the process $Y$ are associated with weights:
each transition $i\to j$ has weight $w_{ij} \geq 0$.  Define a random measure
$\mathcal{W}$ by
\[
  \mathcal{W} \;=\;\;
  \sum_{\mathclap{\;\;\;i\to j \text{ at }t}} w_{ij}\,\delta_{t},
\]
where the sum is over all (finitely many) $(i, j, t)$ such that $Y$ jumps from $i$ to $j$ at
time~$t$ along a trajectory started from $0$ and stopped upon reaching $0$.
Define $W=\int d\mathcal{W}$ as the total weight accumulated along the
trajectory and assume that $0<\Expec{W} < \infty$.  We are now interested in
the distribution of
\[
  Z^w \;\sim\;
  \mathscr{L}\big((Y_{U+t})_{\tau-U \leq t < T_0 - U} \biasedby[big] W \big),
\]
where the conditional distribution of $U$ given $\mathcal{W}$ is $\frac{1}{W}\mathcal{W}$.
In other words, the distribution of $Z^w$ is characterized by
\[
  \Expec*{\big}{F(Z^w)} \;=\; \frac{1}{\Expec{W}}\,
  \ExpecBrackets{\sum_{\mathrlap{\!\!\!\!\! i\to j \text{ at } u}}
  w_{ij}\, F\big((Y_{u+t})_{\tau-u\leq t < T_0-u}\big)},
\]
for any measurable bounded functional $F$.

\begin{proposition} \label{propPathDecompFromTransition}
We have
\[
  Z^w \;\overset{d}{=}\; Y' \paste Y'',
\]
where $Y'$ is the reversed chain and $Y''$ has the same transitions as $Y$.
Conditional on their starting points, $Y'$ and $Y''$ are independent.
They are started from a pair of states $(Y'_0, Y''_0)$ chosen
according to the probability
\[
  \nu(i,j) \;\defas\; \frac{\pi_i\,  w_{ij}\,  q_{ij}}{\pi_0\, \Expec{W}},
  \qquad i \neq j 
\]
and stopped upon reaching $0$. If one of the chains is started from $0$,
its trajectory is reduced to a single point.
\end{proposition}

\begin{proof}
The proof is a series of elementary Markov chain calculations.
We use the standard notation $q_i=\sum_j q_{ij}$.

Consider two starting states $i\neq j$, a trajectory $f$ from~$i$ to~0 and
and a trajectory~$g$ from~$j$ to~0. Let $\gamma_0, \ldots, \gamma_{\nf+1}$ and
$\xi_0, \ldots, \xi_{n_g+1}$ be the successive states visited by $f$ and~$g$,
respectively, and let $x_0, \ldots, x_{\nf}$
and $y_0, \ldots, y_{n_g}$ be the corresponding holding times.
Writing $\Prob{Z^w\in dh}$ for the probability density of $Z^w$
evaluated in a specific trajectory~$h$,
by definition of $Z^w$, we have
\begin{align*}
  & \Expec{W}\, \Prob*{\big}{Z^w\in d(f\paste g)} \\
  &\quad=\; w_{ij}
  \bigg(\prod_{k=0}^{\nf}\! q_{\gamma_{k+1}\gamma_{k}}\bigg)
  q_{ij}
  \bigg(\prod_{k=0}^{n_g}\! q_{\xi_{k}\xi_{k+1}}\bigg)
  \exp\mleft(-\sum_{k=0}^{\nf}\! q_{\gamma_k}x_k
            -\sum_{k=0}^{n_g}\! q_{\xi_k}y_k\mright)\, dx\,dy,
\end{align*}
where $dx = dx_0\cdots dx_{\nf}$ and $dy = dy_0\cdots dy_{n_{\!g}}$.
Rearranging the terms, we get
\begin{align*}
  & \Expec{W}\,\Prob*{\big}{Z^w\in d(f\paste g)} \\
  &\quad=\; \frac{\pi_i}{\pi_0}\, w_{ij}\, q_{ij}
     \times \bigg(\prod_{k=0}^{\nf} \frac{\pi_{\gamma_{k+1}}}{\pi_{\gamma_k}}\, q_{\gamma_{k+1}\gamma_{k}}\bigg)
     \exp\mleft(-\sum_{k=0}^{\nf}q_{\gamma_k}x_k\mright) \,dx\\
  & \qquad\qquad\qquad \times \bigg(\prod_{k=0}^{n_g} \! q_{\xi_{k}\xi_{k+1}}\bigg)
  \exp\mleft(-\sum_{k=0}^{n_g}\! q_{\xi_k}y_k\mright) \,dy\\
  &\quad=\; \frac{\pi_i}{\pi_0}\, w_{ij}\,  q_{ij} \;\Prob[i]{Y'\in df} \;\Prob[j]{Y''\in dg},
\end{align*}
which concludes the proof.
\end{proof}

We close this appendix by proving Proposition~\ref{propNuCircXfromMut}
from the main text, whose statement we reproduce here for convenience.
Recall that $\Xmut$ denotes the process $X$ ``as seen from a uniform mutation time'',
that is,
\[
  \Xmut \;\sim\; \mathscr{L}\big((X_{U+t})_{-U\leq t < T-U} \biasedby[big] M\big)\,,
\]
where $U$ is chosen uniformly at random
among the atoms of the point process $\mathcal{M}$ giving the times of the
mutations associated to the trajectory of $X$.

\begin{repproposition}{propNuCircXfromMut}
Let $\nucirc$ be the probability distribution on the positive integers
defined by
\[
  \nucirc(n) \;=\; C \prod_{k=1}^{n}\frac{1}{\rho_k}\,,
\]
with $C$ the corresponding normalizing constant. Let $\TweakSpacingL{K \sim \nucirc}$ and,
conditional on~$K$, let $X'\!$ and $X''\!$ be two independent realizations
of the logistic branching process~$X$ started from $X'_0 = K$ and $X''_0 = K - 1$.
Then,
\[
  \Xmut \;\overset{d}{=}\; X'\paste X''.
\]
\end{repproposition}

\begin{proof}
Let $X^\circ$ be a Markov chain started from $0$ with the same transition rates
as $X$, except for an additional ``rebirth'' transition from 0 to 1 at an
arbitrary positive rate. Thus, $X^\circ$ is positive recurrent, and it is
straightforward to check that its stationary distribution $(\pi_i)_{i \geq 0}$
satisfies, for $i \geq 1$, 
\begin{equation} \label{eqProofPropNuCircXfromMut}
  \pi_i \;\propto\; i^{-1} \prod_{k = 1}^i \frac{1}{\rho_k}\,.
\end{equation}
Now, since the excursions of $X^\circ$ away from 0 are distributed as the
restriction of~$X$ to~$[0, T]$, we have
\[
  \Xmut \;\sim\; \mathscr{L}\big((X^\circ_{U+t})_{\tau-U\leq t < T^\circ-U} \biasedby[big] M^\circ\big)\,,
\]
where $\tau$ is the first jump time of $X^\circ$; $T^\circ$ its time of first
return to 0; $M^\circ$ its number of mutations on $[\tau, T^\circ]$; and
$U$ the time of a mutation chosen uniformly at random among the mutations
on $[\tau, T^\circ]$. Moreover, each downward jump of $X^\circ$ from $i$ to $i-1$
corresponds to a mutation with probability $\mu / \rho_i$, independently.
Therefore, if conditional on $X^\circ$ we let $\mathcal{W}$ be the measure
defined by
\[
  \mathcal{W} \;=\; \sum_{\quad\mathclap{i\,\searrow \text{ at } t}} \mu \rho_i^{-1}\, \delta_t\,,
\]
where the sum is over all $(i, t)$ such that $X^\circ$ goes from $i$ to $i-1$
at time $t \in [\tau, T^\circ]$, then, conditional on $X^\circ\!$, 
$\mathcal{W}$ is the intensity measure of $\mathcal{M}^\circ$, the Poisson
point process of mutations on $[\tau, T^\circ]$. Thus,
$\Expec{W} = \Expec{M}$, where $W = \int d\mathcal{W}$, 
and for any bounded measurable functional $f$,
\[
  \Expec{\sum_{\,t \in \mathcal{M}^\circ}\! f(X^\circ\!,\, t) \given X^\circ}
  \;=\; \int \! f(X^\circ\!,\, t) \,\mathcal{W}(dt) \;=\;  \sum_{\quad\mathclap{i\,\searrow \text{ at } t}} \mu \rho_i^{-1} f(X^\circ, t)\,,
  \]
As a result, conditional on $X^\circ$, letting $V\sim \frac{1}{W} \mathcal{W}$ we have
\[
  \mathscr{L}\big((X^\circ_{U+t})_{\tau-U\leq t < T^\circ-U} \biasedby[big] M^\circ\big)
  \;=\;
  \mathscr{L}\big((X^\circ_{V+t})_{\tau-V\leq t < T^\circ-V} \biasedby[big] W\big) \,.
\]

Therefore, by applying Proposition~\ref{propPathDecompFromTransition} to
$(X^\circ\!,\, \mathcal{W})$, 
we get that $\Xmut \overset{d}{=} X'\wr X''$, where:
\begin{itemize}
  \item $X'$ and~$X''$ are independent, $X''$ has the same transition rates
    as $X^\circ\!$, and $X'$ has the same transition rates as the time-reversed chain of $X^\circ\!$. 
  \item $(X', X'')$ is started from $(i, i-1)$ with probability
    \[
      \frac{\pi_i\, (\mu \rho_i^{-1})\, (i \rho_i)}{\pi_0\,\Expec{W}}
      \;=\;
      C\prod_{k = 1}^i \frac{1}{\rho_k} \,.
    \]
    where $C$ is the normalization constant (and we have used the expression
    of~$\pi_i$ given in \eqref{eqProofPropNuCircXfromMut} to get the right-hand
    side).
\end{itemize}
Finally, since every positive recurrent birth-death chain is reversible -- a
standard fact that follows from Kolmogorov's criterion for time-reversibility
-- $X'$ in fact has the same transition rates as $X^\circ$; and since $X'$ and
$X''$ are both killed upon reaching~0, these two chain also have the same
transitions rates as $X$.
\end{proof}

\subsection{Proofs for Section~\ref{secTopology}} \label{appProofLemLeGall}

In this appendix, we prove Lemma~\ref{lemLeGall} and
Proposition~\ref{propKeyToCRT}, whose statements we will reproduce
below for convenience.

Let us start by recalling how the notions of correspondence and distortion
can be used to tackle Gromov--Hausdorff--Prokhorov convergence more conveniently
than by working directly with the definition of the metric.
Note that, in order to deal with the Prokhorov component of the metric,
we will use definitions that differ slightly from
those traditionally used in the Gromov--Hausdorff setting.

Let $(\mathscr{X}, r, d, \lambda)$ and $(\mathscr{X}', r', d', \lambda')$ be
two rooted compact metric probability spaces.
Since we view a subset $R\subset \mathscr{X}\times \mathscr{X}'$ as a binary relation,
we write $xR\,x'$ to indicate that $(x,x')\in R$. For any $A\subset \mathscr{X}$,
we let $AR = \Set{x'\in \mathscr{X}' \suchthat \exists x\in A \text{ with }xRx'}$
and we define $RB$ similarly for any subset $B\subset \mathscr{X}'$.
In what follows, we use the term \emph{correspondence
from $\mathscr{X}$ to $\mathscr{X}'$} to refer to any
nonempty subset $R\subset \mathscr{X}\times\mathscr{X}'$.
Note that it is sometimes required that $R$ satisfies
$\mathscr{X}\!R=\mathscr{X}'$ and $\mathscr{X}=R\mathscr{X}'$ to be called
a correspondence, but in our setting it will be more convenient to drop
this restriction.

We now introduce a modified version of the notion of distortion
of a correspondence. In what follow, $A^\epsilon$ denotes the
$\epsilon$-neighborhood of a set $A$.

\begin{definition} \label{defDistortion}
The \emph{Prokhorov distortion} of a correspondence $R$ from a compact metric
probability space $(\mathscr{X}, d, \lambda)$ to another $(\mathscr{X}', d', \lambda')$, which we denote by $\dis(R)$, is the infimum of
the $\epsilon > 0$ such that:
\begin{mathlist}
  \item For all $(x,x')\in R$ and $(y,y')\in R$, $\Abs{d(x,y)-d'(x',y')} \leq \epsilon$.
  \item $(\mathscr{X}\!R)^{\epsilon/2}=\mathscr{X}'$ and $\mathscr{X}=(R\mathscr{X}')^{\epsilon/2}$.
  \item For any Borel set $A \subset \mathscr{X}$, $\lambda'((AR)^{\epsilon/2})+\epsilon \geq \lambda(A)$.
\qedhere
\end{mathlist}
\end{definition}

The usual notion of distortion only takes (i) into account: (ii) is added to be
able to relax the usual definition of correspondence, as discussed above; and
(iii) controls the Prokhorov part of the Gromov--Hausdorff--Prokhorov topology.
It may seem that by replacing the $\epsilon/2$ with $\epsilon$ would yield
a more natural definition; however, this $\epsilon/2$ makes several
calculations neater. Finally, note that because we have only imposed one
inequality in (iii), this definition is not symmetric:
if we let $R^{-1} = \Set{(x', x) \suchthat (x, x') \in R}$ then a priori
$\dis(R) \neq \dis(R^{-1})$.

As the next lemma shows, correspondences and their distortions provide a simple
characterization of the (rooted) Gromov--Hausdorff--Prokhorov convergence.

\begin{lemma} \label{lemCorrespToGHP}
Let $(\mathscr{X}, r, d, \lambda)$ and $(\mathscr{X}', r', d', \lambda')$ be
two rooted compact metric probability spaces.  If there exists a
correspondence $R\subset \mathscr{X}\times \mathscr{X}'$ satisfying $rR\,r'$
and $\dis(R) \leq \epsilon$, then
$\GHPdist(\mathscr{X}, \mathscr{X}')\leq \epsilon$.
\end{lemma}
\begin{proof}
The proof is a straightforward adaptation of the classic analogous result
for the usual notion of correspondence and the
Gromov--Hausdorff metric, see for instance \cite[Theorem 4.11]{Eva08} 
  
Let us define a metric $\delta$ on the disjoint union
$\mathscr{X}\sqcup \mathscr{X}'$ by:
  \begin{itemize}
    \item $\forall x,y\in \mathscr{X}, \delta(x,y)= d(x,y)$;
    \item $\forall x',y'\in \mathscr{X}', \delta(x',y')= d'(x,y)$;
    \item $\forall (x,x')\in \mathscr{X}\!\times\mathscr{X}',
      \delta(x,x')=\epsilon/2+\inf\Set{d(x,y)+d'(x',y')\suchthat (y,y')\in R}$.
  \end{itemize}
It is readily checked that $\delta$ is indeed a metric. Moreover, for any
$(x,x')\in R$, we have $\delta(x,x') = \epsilon/2$.  This implies that
$\delta(r,r') \leq \epsilon/2$, and that for each Borel set $A \subset \mathscr{X}$,
$(AR)^{\epsilon/2} \subset A^{\epsilon}$.  Therefore, it follows from point~(ii) of
Definition~\ref{defDistortion} that the Hausdorff distance
between $\mathscr{X}$ and $\mathscr{X}'$ in $(\mathscr{X} \sqcup \mathscr{X}', \delta)$
is at most $\epsilon$.  Similarly, it follows from point~(iii) 
of Definition~\ref{defDistortion} that
the Prokhorov distance between the extensions of $\lambda$ and
$\lambda'$ to $\mathscr{X} \sqcup \mathscr{X}'$ is also at most $\epsilon$.
  
Therefore, $\GHPdist(\mathscr{X},\mathscr{X}') \leq \epsilon$ and the
proof is complete.
\end{proof}

We are now ready to prove Lemma~\ref{lemLeGall}.
First, recall from Definition~\ref{defRTree} how to obtain
a rooted compact metric probability space $\mathscr{T}_h$ from
a nonnegative càdlàg function $h$ such that {$h(0) = 0$}.

\begin{replemma}{lemLeGall}
The map $h\in D \mapsto \mathscr{T}_h\in \mathbb{M}$ is continuous.
In other words, if $h_1, h_2, \ldots$ and $h$ satisfy the hypotheses of Definition~\ref{defRTree}, then
\[
  h_n \longrightarrow\; h\; \text{ in } D
  \quad\implies\quad
  \mathscr{T}_{h_n} \longrightarrow\; \mathscr{T}_h
  \;\text{ in }(\mathbb{M},\, \GHPdist).
\]
\end{replemma}

\begin{proof}
It is classic \cite{Bil99} that the Skorokhod topology can be metrized by the following metric:
for two càdlàg functions $f$ and $g:[0,1]\to \R$, define
\[
d_{\mathrm{Sk}}(f,g) \;=\; \inf_{\theta} \big((\mathrm{Lip}(\theta)-1)\vee \Norm[\infty]{f-g\circ\theta}\big),
\]
where $\theta$ runs over the set of continuous increasing bi-Lipschitz bijections from $[0,1]$ into itself, and
\[
  \mathrm{Lip}(\theta) \;\defas \sup_{0\leq x<y\leq 1} \Bigg(\frac{\theta(y)-\theta(x)}{y-x} \vee\frac{y-x}{\theta(y)-\theta(x)}\Bigg).
\]
Note that, since $\theta(0)=0$ and $\theta(1)=1$ for every such bijection $\theta$, 
if $\mathrm{Lip}(\theta) < 1+\epsilon$ then $\Norm[\infty]{\theta-\mathrm{Id}} < \epsilon$, where $\mathrm{Id}$ is the identity map.

Let $f$ and $g$ be two nonnegative càdlàg functions such that $f(0) = g(0) = 0$, 
and let $(\mathscr{T}_f, \rf, d_f, \lambda_f)$ and
$(\mathscr{T}_g, \rg, d_g, \lambda_g)$ be the corresponding 
metric spaces. We will show that $d_{\mathrm{Sk}}(f,g) < \epsilon \implies
\GHPdist(\mathscr{T}_f, \mathscr{T}_g) \leq 4\epsilon$.
By definition of $d_{\mathrm{Sk}}$, let us choose a continuous increasing
bijection $\theta:[0,1]\to[0,1]$ such that
$\mathrm{Lip}(\theta) < 1+\epsilon$ and $\Norm[\infty]{f-g\circ\theta} < \epsilon$.

Let $\phi \colon [0, 1] \to \mathscr{T}_f$ denote the
quotient map from $[0, 1]$ to $\mathscr{T}_f$. Note that, because we have
completed the quotient space $[0, 1] /\!\! \sim_{d_f}$ in order to obtain
$\mathscr{T}_f$, this function $\phi$ may be non-surjective, but that $\phi([0, 1])$
is dense in $\mathscr{T}_f$.
Define $\psi:[0,1] \to \mathscr{T}_g$ similarly.

Let $R\subset \mathscr{T}_f\times \mathscr{T}_g$ be the correspondence defined by:
\[
  xR\,x' \;\iff\; \exists t \in [0,1] \st  x = \phi(t)\text{ and } x' = \psi(\theta(t)).
\]
By Lemma~\ref{lemCorrespToGHP}, to show that
$\GHPdist(\mathscr{T}_f,\mathscr{T}_g) \leq 4\epsilon$, it is sufficient to
check that $\dis(R) \leq 4\epsilon$ and that $\rf R\,\rg$.
The latter point is trivial since $\rf=\phi(0)$ and $\rg=\psi(0)=\psi(\theta(0))$, therefore we need to check that the following three points hold:
\begin{mathlist}
  \item $(\mathscr{T}_fR)^{2\epsilon} = \mathscr{T}_g$ and $\mathscr{T}_f = (R\,\mathscr{T}_g)^{2\epsilon}$;
  this is also immediate since $\mathscr{T}_fR=\psi([0,1])$
  is dense in $\mathscr{T}_g$ and since $R\,\mathscr{T}_g=\phi([0,1])$ is dense in $\mathscr{T}_f$.
  \item For all $(x,x'),\, (y,y')\in R$, we have $\Abs{d_f(x,y)-d_g(x',y')} \leq 4\epsilon$.
  \item For any Borel subset $A\subset \mathscr{T}_f$, we have $\lambda_g((AR)^{2\epsilon})+4\epsilon \geq \lambda_f(A)$.
\end{mathlist}

To prove (ii), consider $s<t\in [0,1]$. We need to show that
\[
\Abs{d_f(\phi(s),\phi(t))-d_f(\psi\circ\theta(s),\psi\circ\theta(t))} \;\leq\; 4\epsilon.
\]
This is readily seen, since
\begin{align*}
&\Abs{d_f(\phi(s),\phi(t))-d_f(\psi\circ\theta(s),\psi\circ\theta(t))}\\
&\qquad =\; \Abs[\big]{f(s)+f(t)-2\inf_{[s,t]}f - \big(g(\theta(s))+g(\theta(t)) - 2\inf_{[s,t]}g\circ\theta\big)}\\
&\qquad \leq\; 4\Norm[\infty]{f-g\circ\theta}
\end{align*}

To show (iii), consider a Borel subset $A\subset \mathscr{T}_f$, and let
$\ell$ denote the Lebesgue measure on $[0, 1]$, so that
$\lambda_f(A) = \ell(\phi^{-1}(A))$ and $\lambda_g(A) = \ell(\psi^{-1}(A))$.
Notice that, by definition, $AR=\psi\circ\theta(\phi^{-1}(A))$.  Therefore,
\[
  \lambda_g(AR)
  \;=\; \ell\Big(\psi^{-1}\big(\psi\circ\theta(\phi^{-1}(A))\big)\Big)
  \;\geq\; \ell\big(\theta(\phi^{-1}(A))\big).
\]
Now, it suffices to notice that $\mathrm{Lip}(\theta)<1+\epsilon$ implies that
$\ell(\theta(B)) \geq (1+\epsilon)\ell(B)$ for all Borel subsets $B\subset [0,1]$.
Indeed, this is easily checked for intervals, and extended to Borel sets 
by a monotone class argument.  From this, we obtain
\[
  \lambda_g(AR) \;\geq\; \ell\big(\theta(\phi^{-1}(A))\big) \;\geq\; (1-\epsilon)\ell(\phi^{-1}(A)) \;=\; (1-\epsilon)\lambda_f(A) \;\geq\; \lambda_{f}(A)-\epsilon.
\]
Finally, $\lambda_g((AR)^{2\epsilon}) + 4\epsilon \geq \lambda_g(AR)+\epsilon \geq \lambda_f(A)$.
Therefore,  $\GHPdist(\mathscr{T}_f,\mathscr{T}_g) \leq 4 \epsilon$
and we have proved that 
$d_{\mathrm{Sk}}(f, g) < \epsilon \implies \GHPdist(\mathscr{T}_f,\mathscr{T}_g) \leq 4 \epsilon$,
finishing the proof.
\end{proof}

Let us now turn to Proposition~\ref{propKeyToCRT}. Recall that an admissible
parametrization of $(\mathscr{X}, r, d, \lambda)$ is a càdlàg function
$\phi \colon [0, 1] \to \mathscr{X}$ such that $\phi([0, 1])$ is dense in 
$\mathscr{X}$ and that $t \mapsto \lambda(\phi([0, t]))$ and
$t \mapsto d(r, \phi(t))$ are well-defined random variables.

\begin{repproposition}{propKeyToCRT}
Let $(\mathscr{X}_n, r_n, d_n, \lambda_n)_{n\geq 1}$ be a sequence of random
rooted compact metric probability spaces such that, for each $n\geq 1$,
there exists an admissible parametrization $\phi_n\colon[0,1]\to \mathscr{X}_n$.
Assume that, setting $h_n(t) = d_n(r_n,\phi_n(t))$:
\begin{mathlist}
  \item $\sup_{s,t\in [0,1]} \Abs[\big]{d_n(\phi_n(s),\phi_n(t)) - d_{h_n}(s, t)} \overset{d}{\longrightarrow} 0$. 
  \item $\sup_{t\in [0,1]}\Abs[\big]{\lambda_n(\phi_n([0,t]))-t\,}\overset{d}{\longrightarrow} 0$. 
  \item $(h_n(t))_{t\in [0,1]} \overset{d}{\longrightarrow} (h(t))_{t\in [0,1]}$
    for the Skorokhod topology,
    where $(h(t))_{t\in [0,1]}$ is a random càdlàg function.
\end{mathlist}
Then, $\mathscr{X}_n \overset{d}{\longrightarrow} \mathscr{T}_h$
for the rooted Gromov--Hausdorff--Prokhorov topology.
\end{repproposition}

\begin{proof}
Since by Lemma~\ref{lemLeGall} the assumption (iii) ensures that
$d_{\mathrm{GHP}}(\mathscr{T}_{h_n}, \mathscr{T}_{h}) \to 0$, to prove the
proposition it suffices to show that the
assumptions (i) and (ii) imply that
$d_{\mathrm{GHP}}(\mathscr{X}_n, \mathscr{T}_{h_n})\to 0$.
  
First, by a straightforward extension of Skorokhod's representation theorem  (namely
Theorem~\ref{thmSkorokhod} in Appendix~\ref{appSkorokhod} below), we can assume that the
convergences in assumptions~{(i-iii)} hold almost surely, rather than in distribution.
Note that the fact that $\phi_n$ is assumed to be an admissible parametrization
allows to view the pairs $(\mathscr{X}_n,\phi_n)$ as random variables valued
on a Polish space, which is the key to applying Theorem~\ref{thmSkorokhod}.

Now, as previously,
let $\psi_n \colon [0, 1] \to \mathscr{T}_{h_n}$ be the quotient map
in the construction of $\mathscr{T}_{h_n}$. Remember that, because we have
completed $[0, 1]/\!\!\sim_{d_{h_n}}\!$ to obtain $\mathscr{T}_{h_n}$, the map
$\psi_n$ is not surjective; however, $\psi_n([0, 1])$ is dense in
$\mathscr{T}_{h_n}$.
Let then $R_n$ be the correspondence from
$\mathscr{X}_n$ to $\mathscr{T}_{h_n}$ defined as
\[
  R_n \;=\; \Set{(\phi_n(t), \psi_n(t)) \suchthat t \in [0, 1]} \,.
\]
Since $\phi_n(0) = \rn$ (because $\phi_n$ is an admissible parametrization of
$\mathscr{X}_n$) and since the root of $\mathscr{T}_{h_n}$ is by construction
$\rn' \defas \psi_n(0)$, we have $\rn\, R_n\, \rn'$.
Thus, by Lemma~\ref{lemCorrespToGHP} to show that
$d_{\mathrm{GHP}}(\mathscr{X}_n, \mathscr{T}_{h_n})\to 0$ it suffices to show
that $\dis(R_n) \to 0$.

First, $\mathscr{X}_n R_n = \psi_n([0, 1])$ is dense in $\mathscr{T}_{h_n}$
and $R_n \mathscr{T}_{h_n} = \phi_n([0, 1])$ is dense in $\mathscr{X}_n$.
Therefore,
\begin{equation} \label{eqProofKeyToCRT01}
  \inf \Set{\epsilon > 0 \suchthat
    (\mathscr{X}_n R_n)^{\epsilon / 2} \subset \mathscr{T}_{h_n}
    \text{ and }
    (R_n \mathscr{T}_{h_n})^{\epsilon / 2} \subset \mathscr{X}_n}
    \;=\; 0  \,.
\end{equation}

Second, let $(x, x')$ and $(y, y')$ be any two elements of $R_n$, i.e.\ let
$s, t \in [0, 1]$ and set $(x, x') = (\phi_n(s), \psi_n(s))$ and
$(y, y') = (\phi_n(t), \psi_n(t))$. Then, by assumption (i),
\begin{equation} \label{eqProofKeyToCRT02}
  \Abs{d_n(x, y) \,-\, d_{\mathscr{T}_{h_n}}\!(x', y')} \;=\;
  \Abs[\big]{d_n(\phi_n(s), \phi_n(t)) \,-\, d_{h_n}(s, t)}
  \;\tendsto[\text{unif.\ in.\  } s, t\,]{n\to\infty} 0 \,.
\end{equation}

Let us fix $\epsilon>0$; we now show that for all $n$ large enough, for all Borel subset $A\subset \mathscr{X}_n$, we have
\begin{equation} \label{eqProofKeyToCRT03}
  \lambda_{\mathscr{T}_{h_n}}((AR_n)^{\epsilon})+\epsilon \geq \lambda_n(A).
\end{equation} 
First, note that because $h_n$ converges in $D$, there exists $k\geq 1$ such that for all $n$ large enough, there exists $t^n_0=0 < t^n_1 < \dots < t^n_k=1$ such that for all $j\in \{1,\dots,k\}$, $\mathrm{diam}(\psi_n(I^n_j)) < \epsilon$, where $I^n_j=\COInterval[\normalsize]{t^n_{j-1},t^n_j}$ for $j < k$ and $I^n_k=[t_{k-1},1]$.
Fix $n_0\geq 1$ so that for all $n\geq n_0$, for all $j$,
\[
\Abs{\lambda_n\big(\phi_n\big([0,t^n_j[\big)\big) - t^n_j} < \frac{\epsilon}{4k}, \quad \text{ so that }\quad \Abs{\lambda_n\big(\phi_n\big(I^n_j\big)\big) - (t^n_{j}-t^n_{j-1})} < \frac{\epsilon}{2k}.
\]
Now consider $n\geq n_0$ and choose a Borel subset $A\subset \mathscr{X}_n$.
Writing $B=AR_n$, define
\[
  J = \Set{j\suchthat B\cap \psi_{n}(I^n_j)\neq \emptyset} \supset \Set{j\suchthat A\cap \phi_{n}(I^n_j)\neq \emptyset},
\]
and notice that because $\mathrm{diam}(\psi_n(I^n_j)) < \epsilon$, we have $B^\epsilon\supset \bigcup_{j\in J}\psi_n(I^n_j)$.
Then we have
\[
\lambda_{\mathscr{T}_{h_n}}((AR_n)^{\epsilon}) \geq \sum_{j\in J}\lambda_{\mathscr{T}_{h_n}}(\psi_n(I^n_j)) \geq \sum_{j\in J} (t_{j+1}-t_j).
\]
But $(t_{j+1}-t_j) \geq \lambda_n(\phi_n(I^n_j))-\frac{\epsilon}{2k}$, so
\[
\sum_{j\in J_B} (t_{j+1}-t_j) \geq \sum_{j\in J_B} \lambda_n(\phi_n(I^n_j)) - \frac{\epsilon}{2} \geq \lambda_n(A\cap\phi_n([0,1])) - \frac{\epsilon}{2}.
\]
Finally, note that we fixed $n$ large enough so that $\lambda_n(\phi_n([0,1]))\geq 1-\frac{\epsilon}{2k}\geq 1-\frac{\epsilon}{2}$, hence
\[
  \lambda_n(A\cap\phi_n([0,1])) - \frac{\epsilon}{2} \geq \lambda_n(A) - \epsilon.
\]
Putting the last three displays together, we have proved \eqref{eqProofKeyToCRT03}.

Finally, combining \eqref{eqProofKeyToCRT01}, \eqref{eqProofKeyToCRT02} and
\eqref{eqProofKeyToCRT03}, and recalling the Definition~\ref{defDistortion}
of the Prokhorov distortion, we see that $\dis(R_n) \to 0$. This concludes
the proof.
\end{proof}

\subsection{An extension of Skorokhod's representation theorem} \label{appSkorokhod}

In this appendix, we discuss the extension of Skorokhod's representation
theorem used at the beginning of the proof of Proposition~\ref{propKeyToCRT}.
For the sake of rigour and completeness, we give a formal statement and a proof.

\begin{theorem} \label{thmSkorokhod}
Let $X$ and $Y$ be two Polish spaces endowed with their Borel $\sigma$-field,
 $f\colon X\to Y$ be a Borel
function and $(\mu_n)_{n\geq 1}$ be a sequence of probability measures on $X$
such that $f_*\mu_n$ converges weakly. Then, there exists a probability
measure~$\theta$ on $X^{\N}$ whose $n$-th marginal is $\mu_n$ and which is
such that, letting $(x_n) \sim \theta$, $f(x_n)$ converges almost surely.
\end{theorem}

\begin{proof}
Let $(\epsilon_m)_{m\geq 1}$ be a sequence of positive numbers such that 
$\sum_m \epsilon_m < \infty$ and let $\nu$ be the weak limit of $f_*\mu_n$.
Since $Y$ is Polish, for all $m\geq 1$ we can find a finite, measurable
partition $B_1^m,\dots, B_{k_m}^m$ of $Y$ satisfying:
\begin{mathlist}
  \item $\nu(\partial B_i^m) =0$ for all $1\leq i \leq k_m$.
  \item $\sum_{i} \nu(B_i^m) < \epsilon_m$, where the sum runs over the indices
    $i$ s.t.\ {$\mathrm{diam}(B_i^m) \geq \epsilon_m$}.
\end{mathlist}
Furthermore, we can assume that these partitions are refining as $m$ increases;
more specifically, that for each $m\geq 1$, there exist
$1 \leq j_1 < j_2 <\dots < j_{k_m} = k_{m+1}$ such that, for all $1\leq i \leq k_m$,
\[
  B_i^m \; = \bigcup_{j = j_{i-1}+1}^{j_i} \!\!\! B_j^{m+1},
\]
with the convention $j_0=0$.
Now, let us define a measurable partition $A_1^m,\dots,A_{k_m}^m$ of $X$ by setting
$A_i^m=f^{-1}(B_i^m)$.
Since $f_*\mu_n \to \nu$ weakly, by the Portmanteau theorem
$\mu_n(A_i^m) \to \nu(B_i^m)$ for all $i$ as $n\to\infty$.
Therefore, for any sequence $\delta_m > 0$
we can find an increasing sequence $(N_m)_{m\geq 1}$ such that for all $m \geq 1$
we have: for all $n\geq N_m$ and $1\leq i \leq k_m$,
\begin{equation} \label{eqAppCDelta}
  \Abs{\sum_{j=1}^{i} \mu_n(A_j^m) -\sum_{j=1}^{i} \nu(B_j^m)} \;<\; \delta_m.
\end{equation}

Let us define, for all positive integers $n,m$ and for all $1\leq i\leq k_m$,
\begin{align*}
\begin{dcases}
  I_{n,i}^m \;=\; \textstyle
    \COInterval{\sum_{j=1}^{i-1} \mu_n(A_j^{m}),\, \sum_{j=1}^{i} \mu_n(A_j^{m})},\\
  J^m_i \;=\; \textstyle \COInterval{\sum_{j=1}^{i-1} \nu(B_j^{m}) +
    \delta_m,\, \sum_{j=1}^{i} \nu(B_j^{m}) - \delta_m},
\end{dcases}
\end{align*}
where, by convention, $J^m_i=\varnothing$ whenever $\sum_{j=1}^{i-1}
\nu(B_j^{m}) + \delta_m \geq \sum_{j=1}^{i} \nu(B_j^{m}) - \delta_m$.
Note that, by construction, $J^m_i \subset I^m_{n,i}$ for all $n \geq N_m$.

We now build a sequence of
random variables $(x_n)$ such that $x_n \sim \mu_n$.
For $n\geq N_1$, let $m_n$ be the integer such that $N_{m_n} \leq n < N_{m_n+1}$.
It is classic that Borel subsets of a Polish space are standard Borel (see e.g.\ \cite[Proposition 3.3.7]{Sri98}),
and this implies~-- either by a direct construction if the subset is countable,
or using the Borel isomorphism theorem \cite[Theorem 3.3.13]{Sri98} otherwise --
that for each {$1\leq i \leq k_{m_n}$}\linebreak such that $\mu_n(A_i^{m_n})>0$, there
exists a measurable map $\phi_{n,i} : I^{m_n}_{n, i} \to A_i^{m_n}$
such that when $U$ is drawn according to the normalized Lebesgue measure on
$I^{m_n}_{n,i}$, the distribution of $\phi_{n,i}(U)$ is $\mu_n(\,\cdot \mid A_i^{m_n})$.
Thus, defining $\phi_n:\COInterval{0,1}\to X$ by
\[
  \phi_n(t) \;=\; \sum_{i=1}^{k_{m_n}} \phi_{n,i}(t)\, \Indic{t\in I^{m_n}_{n, i}}\,,
\]
and then taking $x_n = \phi_n(U)$, where $U$ is a uniform variable
on $\COInterval{0,1}$, we get a random sequence $(x_n)$ such that
$x_n \sim \mu_n$ for all $n$.
Furthermore, notice that, by construction, if $n\geq N_m$ and $U\in I^m_{n,i}$,
then $f(x_n) \in B^m_i$.

Assume without loss of generality that the sequence $(\delta_m)$ satisfies
{$\sum_m \delta_m k_m < +\infty$}. We will now show that this implies that
$(f(x_n))$ is almost surely a Cauchy sequence.
We say that the interval $J^m_i$ is $m$-good if $\mathrm{diam}(B^m_i) \leq \epsilon_m$.
By construction, for a fixed integer $m\geq 1$, the union $G_m = \bigcup_iJ^m_i$ of
$m$-good intervals has Lebesgue measure at least $1-\epsilon_m - 2(k_m+1)\delta_m$.
Therefore, we have
\[
  \sum_{m\geq 1}\Prob{U\notin G_m} \;\leq\;
  \sum_{m\geq 1} \big(\epsilon_m + 2(k_m+1)\delta_m\big) \;<\; \infty,
\]
and so, by the Borel--Cantelli lemma, there almost surely exists $m^*$ such
that for all $m \geq m^*$, there is a unique index $i_m$ such that $U$ is in
the $m$-good interval~$J_{i_m}^m$.
This implies that, almost surely, for all $n'\geq n\geq N_{m^*}$ and writing
$i=i_{m_n}$ to avoid clutter, $U \in I^{m_n}_{n,i}\cap I^{m_n}_{n',i}$
and so $f(x_n),f(x_{n'})\in B^{m_n}_i$ with
$\mathrm{diam}(B^{m_n}_i)\leq \epsilon_{m_n}$.  This shows that $(f(x_n))$ is
almost surely a Cauchy sequence, concluding the proof.
\end{proof}

\subsection{Tail of the size of critical Galton--Watson trees}

In order to make this article as self-contained as possible, we provide a short
proof of Proposition~\ref{propTotalProgenyGW} for the asymptotic equivalent of
the probability that a critical Galton--Watson tree has size~$n$, which
is a key element in our study. As previously, we repeat the statement of the
proposition here for convenience.

\begin{repproposition}{propTotalProgenyGW}
Let $\hat{\mathscr{T}}$ be a critical Galton--Watson tree whose offspring
distribution has a finite variance $\sigma^2 > 0$ and is not supported on $k\N$, for any $k\geq 2$.
Let $\hat{A}_n$ denote the event
$\Set*{\hat{\mathscr{T}} \text{ has } n \text{ vertices}}$. Then,
\[
  \Prob*{\normalsize}{\hat{A}_n} \;\underset{n\to\infty}{\sim}\;
  \frac{1}{\sqrt{2\pi \sigma^2}}\, n^{-3/2}\,.
\]
\end{repproposition}

\begin{proof}
Let $\hat{M}$ denote the offspring distribution of $\hat{\mathscr{T}}$, let
$\xi_1,\xi_2,\dots$ be i.i.d.\ copies of $\hat{M} - 1$ and
set $S_n = \sum_{i=1}^n \xi_i$. As was first noted by Dwass in~\cite{Dwa69},
\begin{align*}
  \Prob*{\normalsize}{\hat{A}_n}
  \;&=\;\, \Prob{S_i\geq 0 \text{ for }1\leq i\leq n-1\text{ and } S_n = -1} \\
  \;&=\;\, \frac{1}{n}\, \Prob{S_n = -1}\,. 
\end{align*}
The first equality is easily seen by marking the root as \emph{to-visit}, and
then at each step removing a vertex from the \emph{to-visit} pile and adding
its children to it: the procedure ends where there are no vertices left to
visit -- which happens after exactly $n$ steps, where $n$ is the total number
of vertices in the tree; and if we let~$\xi_i$ denote the number of
vertices added/removed from the pile at step~$i$, then the number of vertices
on the pile after step~$i$ is exactly $S_i + 1$.

The fact that $\Prob{S_i\geq 0 \text{ for }1\leq i\leq n-1, S_n = -1} = \frac{1}{n}\, \Prob{S_n = -1}$
has become folklore and is a special case of a result sometimes
known as Kemperman's formula or as the hitting time theorem -- see
e.g.~\cite[Theorem~3.14]{Van16}. We mention a simple proof based on
Dvoretzky and Motzkin's cycle lemma~\cite{DvM47}: 
let $\mathcal{S}$ denote the set of vectors $\mathbf{k} = (k_1, \ldots, k_n)$
of increments of $(S_n)_{n \geq 0}$ that are such that $S_n = -1$.
For every permutation $\sigma$ of $\Set{1, \ldots, n}$, if $\mathbf{k} \in \mathcal{S}$
then $\mathbf{k}_\sigma \defas (k_{\sigma(1)}, \ldots, k_{\sigma(n)}) \in \mathcal{S}$
and $\Prob{\xi_1 = k_1, \ldots, \xi_n = k_n} =
\Prob*{\normalsize}{\xi_1 = k_{\sigma(1)}, \ldots, \xi_n = k_{\sigma(n)}}$.

Now, let $\mathfrak{C}(n)$ denote the set of cyclic shifts of $\Set{1, \ldots, n}$,
and define an equivalence relation $\circlearrowleft$ on $\mathcal{S}$ by
saying that $\mathbf{k} \circlearrowleft \mathbf{k}'$ when there exists
$\sigma \in \mathfrak{C}(n)$ such that $\mathbf{k}' = \mathbf{k}_\sigma$.
By the cycle lemma, each equivalence class of $\mathcal{S}_{/\circlearrowleft}$ has exactly
one member such that the corresponding trajectory of $(S_n)$ 
satisfies $S_i\geq 0$ for $1\leq i\leq n-1$ and $S_n = -1$. Let
that member be the representative of its class, and denote by $\mathcal{S}^\star$
the set of those representatives. The cycle lemma also implies
that each equivalence class of $\mathcal{S}_{/\circlearrowleft}$ has cardinal
$n$: indeed, if there existed $\mathbf{k} \in \mathcal{S}^\star$ and
$\sigma, \sigma' \in \mathfrak{C}(n)$ such that $\sigma \neq \sigma'$ and
$\mathbf{k}_\sigma = \mathbf{k}_{\sigma'}$, then by taking
$\rho = \sigma^{-1} \circ \sigma'$ we would have $\rho \neq \mathrm{Id}$ and
yet $\mathbf{k}_\rho \in \mathcal{S}^\star$, contradicting the lemma.
As a result, 
\begin{align*}
  \Prob{S_n = -1} \;
  &=\; \sum_{\sigma \in \mathfrak{C}(n)} \sum_{\,\mathbf{k} \in \mathcal{S}^\star\!} 
    \; \prod_{i = 1}^n \, \Prob{\xi_i = k_{\sigma(i)}} \\
  &=\; n\, \Prob{S_i\geq 0, 1\leq i\leq n-1; S_n = -1} \,.
\end{align*}
Finally, applying a local limit theorem -- see e.g,
\cite[Theorem~3.5.2]{Dur10} -- to $(S_n)_{n\geq 0}$ yields
$\Prob{S_n = -1} \sim 1 / \sqrt{2\pi\sigma^2 n}$, thereby concluding the proof.
\end{proof}

\end{document}